\documentclass[10pt, reqno]{amsart}

\usepackage{hyperref}
\hypersetup{
	colorlinks=true,
	citecolor=blue,
	linkcolor=blue,
	filecolor=magenta,      
	urlcolor=cyan,
}
\usepackage{caption} 
\usepackage{adjustbox}

\usepackage{fullpage}

\usepackage{setspace}

\usepackage{amsmath, amsthm, amssymb, graphicx, dsfont, stmaryrd, extarrows, enumitem}
\usepackage[T1]{fontenc}
\usepackage{textcomp}
\usepackage{lmodern}
\usepackage{microtype} 
\usepackage{amscd}
\usepackage{mathrsfs}
\usepackage{tikz-cd}
\usepackage{tikz}
\usetikzlibrary{matrix,patterns}
\usepackage[colorinlistoftodos]{todonotes}
\usepackage{color}
\usepackage{bbm}
\usepackage{verbatim}
\usepackage[mathscr]{euscript}
\usepackage{bm}
\usepackage[capitalize,noabbrev,nameinlink]{cleveref}

\numberwithin{equation}{section}
\theoremstyle{plain}
\newtheorem{thm}[equation]{Theorem}
\newtheorem*{thm*}{Theorem}

\newtheorem{prop}[equation]{Proposition}
\crefname{prop}{Proposition}{Propositions}
    
\newtheorem{cor}[equation]{Corollary}       
\newtheorem{lem}[equation]{Lemma}
\crefname{lem}{Lemma}{Lemmas}

\theoremstyle{definition} 
\newtheorem{defn}[equation]{Definition} 

\newtheorem{ex}[equation]{Example}

\newtheorem{rem}[equation]{Remark}

\newtheorem{chunk}[equation]{}
\crefformat{chunk}{(#2#1#3)}

\newcommand{\Z}{\mathbb{Z}}

\newcommand{\mf}{\mathfrak}
\newcommand{\Hom}{\mathrm{Hom}}
\newcommand{\iHom}{\underline{\mathrm{Hom}}}

\newcommand{\msf}[1]{\mathsf{#1}}

\newcommand{\mc}[1]{\mathcal{#1}}
\newcommand{\mrm}[1]{\mathrm{#1}}
\newcommand{\mbb}[1]{\mathbb{#1}}
\newcommand{\scr}[1]{\mathscr{#1}}

\newcommand{\p}{\mathfrak{p}}

\newcommand{\m}{\mathfrak{m}}

\newcommand{\T}{\mathsf{T}}
\newcommand{\X}{\mathsf{X}}

\newcommand{\1}{\mathds{1}}

\newcommand{\upcl}{\lor\!}

\newcommand{\Spc}{\msf{Spc}}

\newcommand{\supp}{\msf{supp}}
\newcommand{\supph}{\msf{supp}^\msf{h}_\Sigma}

\renewcommand{\mod}[1]{\mathrm{Mod}_{#1}}

\newcommand{\A}{\mathsf{A}}
\newcommand{\B}{\mathsf{B}}

\newcommand{\D}{\mathsf{D}}

\renewcommand{\mod}[1]{\msf{mod}(#1)}
\newcommand{\Mod}[1]{\msf{Mod}(#1)}
\newcommand{\Flat}[1]{\msf{Flat}(#1)}
\newcommand{\ab}{\msf{Ab}}
\newcommand{\Thick}{\msf{Thick}}
\newcommand{\ti}{\textit}
\renewcommand{\tilde}[1]{\widetilde{#1}}
\newcommand{\rlim}{\varinjlim}

\renewcommand{\t}{\text}
\renewcommand{\c}{\mrm{c}}

\newcommand{\op}{\mrm{op}}

\renewcommand{\Flat}[1]{\msf{Flat}(#1)}

\newcommand{\closed}[1]{\msf{KZg}^\Sigma_{\msf{Cl}}(#1)}
\newcommand{\y}{\msf{y}}
\renewcommand{\bar}{\overline}
\newcommand{\hspec}[1]{\msf{Spc}^{\msf{h}}(#1)}

\renewcommand{\phi}{\varphi}
\newcommand{\shspec}[1]{\msf{Spc}_{\Sigma}^\msf{h}(#1)}
\newcommand{\sspec}[1]{\msf{Spc}_{\Sigma}(#1)}
\newcommand{\Irr}[1]{\msf{IrrRk}({#1})}
\newcommand{\ssupp}[1]{\msf{supp}_\Sigma(#1)}
\newcommand{\hssupp}[1]{\msf{supp}_{\Sigma}^{\msf{h}}(#1)}
\newcommand{\rsupp}[1]{\msf{supp}^{\msf{rk}}(#1)}
\newcommand{\hsupp}[1]{\msf{supp}^\msf{h}(#1)}
\usepackage{todonotes}

\title{The shift-homological spectrum and parametrising kernels of rank functions}
\begin{document}
	\author{Isaac Bird}
	\address[Bird]{Department of Algebra, Faculty of Mathematics and Physics, Charles University in Prague, Sokolovsk\'{a} 83, 186 75 Praha, Czech Republic}
	\email{bird@karlin.mff.cuni.cz}
	\author{Jordan Williamson}
	\address[Williamson]{Department of Algebra, Faculty of Mathematics and Physics, Charles University in Prague, Sokolovsk\'{a} 83, 186 75 Praha, Czech Republic}
	\email{williamson@karlin.mff.cuni.cz}
	\author{Alexandra Zvonareva}
	\address[Zvonareva]{Institute of Mathematics of the Czech Academy of Sciences, Žitná 25, 115 67 Prague,
Czech Republic}
	\email{zvonareva@math.cas.cz}
    
\begin{abstract}
For any compactly generated triangulated category we introduce two topological spaces, the shift-spectrum and the shift-homological spectrum. We use them to parametrise a family of thick subcategories of the compact objects, which we call radical. 
These spaces can be viewed as non-monoidal analogues of the Balmer and homological spectra arising in tensor-triangular geometry: we prove that for monogenic tensor-triangulated categories the Balmer spectrum is a subspace of the shift-spectrum. 
To construct these analogues we utilise quotients of the module category, rather than the lattice theoretic methods which have been adopted in other approaches.
We characterise radical thick subcategories and show in certain cases, such as the perfect derived categories of tame hereditary algebras or monogenic tensor-triangulated categories, that every thick subcategory is radical. We establish a close relationship between the shift-homological spectrum and the set of irreducible integral rank functions, and provide necessary and sufficient conditions for every radical thick subcategory to be given by an intersection of kernels of rank functions. In order to facilitate these results, we prove that both spaces we introduce may equivalently be described in terms of the Ziegler spectrum.
\end{abstract}
	\maketitle
 \setcounter{tocdepth}{1}
 \tableofcontents

\section{Introduction}
Since the paper of Hopkins--Smith \cite{hopkinssmith} which classified thick subcategories of the homotopy category of finite spectra, there has been a systematic programme to classify thick subcategories of small triangulated categories. In many established examples this classification is phrased in terms of a topological space. For instance, this can be found in the work of Hopkins \cite{Hopkins} and Neeman \cite{Neemanchrom}, 
who showed that for a commutative noetherian ring $R$, the thick subcategories of the perfect derived category are in bijection with the specialisation closed subsets of $\msf{Spec}(R)$. 
In \cite{bcr}, Benson, Carlson and Rickard initiated an investigation into the thick subcategories of the stable module category of group rings, and provided a complete classification for $p$-groups
in terms of the projective scheme associated to the group cohomology ring. 

In all of these examples, the triangulated categories in question are tensor-triangulated categories: they have a symmetric monoidal product $\otimes$ and furthermore each thick subcategory is a $\otimes$-ideal. A universal approach for classifying thick $\otimes$-ideals, unifying all previous classifications, was established by Balmer. In \cite{Balmer}, he defined a space $\msf{Spc}(\scr{K})$, now known as the Balmer spectrum, which parametrises radical thick $\otimes$-ideals of $\scr{K}$ using a support theory. More explicitly, the Balmer spectrum has as points the prime thick $\otimes$-ideals $\p$ of $\scr{K}$, and a basis of closed sets given by supports of objects in $\scr{K}$: 
\[
\supp(A)=\{\p\in\msf{Spc}(\scr{K}):A\not\in\p\}.
\]
Balmer's classification result is that $\supp$ establishes a bijection between the radical thick $\otimes$-ideals of $\scr{K}$ and Thomason subsets of $\msf{Spc}(\scr{K})$. Furthermore, the pair $(\msf{Spc}(\scr{K}),\supp)$ is universal: it is the terminal support datum.

Away from the tensor-triangular world, other classifications of all thick subcategories have been obtained. However, these results rely on a case-by-case analysis of the structure of the specific  triangulated category in question, due to the lack of an overarching method. Recently, however, approaches to constructing a non-monoidal analogue of the Balmer spectrum have been forthcoming, for example \cite{GratzStevenson, Krausecentral, Matsui}. 

Unfortunately, one immediately encounters major obstacles. One issue is that, unlike in the tensor-triangular setting, the lattice $\msf{Thick}(\scr{K})$ of thick subcategories is not distributive - even for categories as simple as the perfect derived category of the path algebra of the $A_2$ quiver - a property which is satisfied by lattices of nice (for instance, Thomason) subsets of a topological space. Another obstacle is that, if one wants to consider a universal support datum, it has been shown by \cite{balmerocal}, building on the work of \cite{GratzStevenson}, that the space one obtains is essentially trivial: the universal support datum consists of the space whose points are the thick subcategories themselves. In other words, to topologically understand all thick subcategories one has to know them all to begin with. Consequently, one cannot have a complete analogue of the Balmer spectrum - there is no universal smaller topological space from which one can decompress information to obtain a classification of all thick subcategories. 

However, this topological perspective can still be used to classify certain sublattices of $\Thick(\scr{K})$. While the aforementioned approaches to constructing a non-monoidal Balmer spectrum are different, they have a common thread: identify a sublattice of the lattice of all thick subcategories, show that it is a spatial frame, and apply Stone duality to obtain a topological space which will parametrise certain families of thick subcategories. 

In this article, we propose an alternative approach to constructing a non-monoidal  analogue of the Balmer spectrum, which does not follow this lattice-based recipe. Instead, we tackle the problem on the level of functor categories, by considering a non-monoidal analogue of the homological spectrum. Introduced by Balmer and Balmer--Krause--Stevenson in \cite{Balmernilpotence, bks}, the homological spectrum $\hspec{\scr{K}}$ of a tensor-triangulated category $\scr{K}$ is a topological space which is defined on $\mod{\scr{K}}$, the category of finitely presented additive functors $\scr{K}^{\op}\to \ab$. Its points are the maximal Serre $\otimes$-ideals of $\mod{\scr{K}}$, with supports given by a functorial analogue of those for the Balmer spectrum. Balmer proved that there is a continuous closed surjection $\hspec{\scr{K}}\to\msf{Spc}(\scr{K})$, which was shown in \cite{BHScomparison}, to be the Kolmogorov quotient.

For an arbitrary compactly generated triangulated category $\T$ with compacts $\scr{K}=\T^{\c}$, we define a space we call the \emph{shift-homological spectrum}, $\shspec{\T^{\c}}$. In analogy with the homological spectrum, the points of $\shspec{\T^{\c}}$ are the maximal $\Sigma$-invariant Serre subcategories of $\mod{\T^{\c}}$, which we call shift-homological primes. We equip $\shspec{\T^{\c}}$ with a basis of closed sets given by a non-monoidal support theory. 

Considering maximal $\Sigma$-invariant Serre subcategories of $\mod{\T^{\c}}$ also takes its motivation from another construction. In \cite{chuanglazarev}, rank functions on $\T^{\c}$ were introduced to provide a generalisation of the theory of matrix localisations to triangulated categories. A rank function on $\T^{\c}$ is an assignment of a non-negative real number to every object and morphism in $\T^{\c}$ which is invariant under the action of $\Sigma$ and satisfies certain natural conditions. It was shown in \cite{conde2022functorial}, that if one restricts to rank functions with integer values, such rank functions have a decomposition into irreducible components. Each irreducible component corresponds to a maximal $\Sigma$-invariant Serre subcategory of $\mod{\T^{\c}}$, and thus an irreducible rank function uniquely determines a point in $\shspec{\T^{\c}}$. The interplay between rank functions and shift-homological primes plays a central role throughout the paper, as we see later in this introduction.

Associated to any shift-homological prime $\mc{B}$ is a thick subcategory of compact objects: namely, the preimage of $\mc{B}$ under the restricted Yoneda embedding. We call such a thick subcategory \emph{shift-prime}, and build a topological space $\sspec{\T^{\c}}$ with points given by these shift-prime thick subcategories and the topology defined analogously to the Balmer spectrum, with a basis of closed sets given by $\ssupp{A}=\{\msf{P}\in\sspec{\T^{\c}}:A\not\in\msf{P}\}$, as $A$ runs through all compact objects. We call this space the \emph{shift-spectrum}. We see this construction as a shadow in the non-monoidal world of the approach to the Balmer spectrum through the homological spectrum. Indeed, for tensor-triangulated categories, there is a concrete connection between our spaces and their tensor-triangular versions.

\begin{thm*}[\ref{lem:downtohomological}, \ref{thm:maptobalmer}]
Let $\T$ be a rigidly-compactly generated tensor-triangulated category generated by its tensor unit. Then every homological prime $\mc{B}\in\hspec{\T^{\c}}$ can be obtained from a shift-homological prime $\mc{S}\in\shspec{\T^{\c}}$. Furthermore, every Balmer prime is a shift-prime and the induced inclusion
\[
\msf{Spc}(\T^{\c})\to\sspec{\T^{\c}}
\]
is continuous.
\end{thm*}

Thus, our space can be seen as a candidate for a generalisation of the Balmer spectrum to the non-monoidal setting. Having introduced the spaces considered in this paper let us discuss the structure and the results in more detail. 

First, we focus on the construction of $\shspec{\T^{\c}}$. To make computations of this space more tractable, we construct a homeomorphic space, which can be obtained from $\T$ without reference to the functor category. More specifically, $\Sigma$-invariant Serre subcategories of $\mod{\T^{\c}}$ are in bijection with the closed sets of another topological space called the $\Sigma$-Ziegler spectrum of $\T$, denoted $\msf{Zg}^{\Sigma}(\T)$, whose points are the indecomposable pure injective objects. This topology is a coarsening of the classical Ziegler topology, which takes into account the action of $\Sigma$. Through this space we are able to provide the following alternative description of $\shspec{\T^{\c}}$, showing that it can be recovered from $\msf{Zg}^{\Sigma}(\T)$.

\begin{thm*}[\ref{homeomorphism}]
There is a homeomorphism 
\[
\varphi\colon \shspec{\T^{\c}}\xrightarrow{\sim} \closed{\T}^{\msf{GZ}}
\]
between the shift-homological spectrum and the set of closed points in the Kolmogorov quotient of $\msf{Zg}^{\Sigma}(\T)$, equipped with a Gabriel-Zariski topology. 
\end{thm*}

The superscript $\msf{GZ}$ indicates a retopologising of the set $\closed{\T}$, to take into account the fact that the Ziegler topology and the topology given by supports are, in some sense, dual. We note that this is similar to the tensor-triangular setting, as it was shown in \cite{birdwilliamsonhomological} that Balmer's homological spectrum can be recovered from a $\otimes$-compatible Ziegler topology.

The next stage is in formally establishing the relationship between $\shspec{\T^{\c}}$ and $\sspec{\T^{\c}}$. It is shown in \cref{prop:KQ} that there is a closed continuous surjection
\[
\shspec{\T^{\c}}\to \sspec{\T^{\c}}
\]
which realises $\sspec{\T^{\c}}$ as the Kolmogorov quotient of $\shspec{\T^{\c}}$; in particular, $\sspec{\T^{\c}}$ is a $\mrm{T}_{0}$-space.

Following this, we investigate which thick subcategories the space $\sspec{\T^{\c}}$ can parametrise. A thick subcategory is called \emph{radical} if it is the intersection of all prime thick subcategories containing it. These radical thick subcategories are precisely the thick subcategories that our space can parametrise, see \cref{bijectionradical}.

However, unlike in the tensor-triangular setting, we do not have an intrinsic characterisation of radical thick subcategories - this is one pitfall of not having a monoidal structure - and nor will, in general, every thick subcategory be radical. Nevertheless, in \cref{charofrad} we give conditions in terms of the Ziegler spectrum which ensure that a thick subcategory of $\T^{\c}$ is radical. These allow us, in subsequent sections, to illustrate the following examples of compactly generated triangulated categories where all thick subcategories are radical. 

\begin{thm*}[\ref{thm:tamheredeverythingradical}, \ref{prop:LocNoeth}, and  \ref{prop:rigideverythingradical}]
For the following examples of $\T$, every thick subcategory of $\T^\c$ is radical:
\begin{enumerate}
\item $\T=\D(\Lambda)$ for $\Lambda$ a tame hereditary algebra over an algebraically closed field;
\item $\T$ is of finite type, that is, $\T^\c$ is locally finite;
\item $\msf{T}$ is a rigidly-compactly generated tensor-triangulated category generated by its tensor unit.
\end{enumerate}
\end{thm*}

Having established some similarities with the Balmer spectrum, we also highlight some differences. As noted in \cite{Balmerhomsupp}, if $p\geqslant 5$ then the tensor-triangulated category $\msf{stmod}(k C_p)$ 
has no non-trivial homological and Balmer primes. This is quite different to our setting, where, although $\sspec{\msf{stmod}(k C_p)}$ is a point, the space $\shspec{\msf{stmod}(k C_p)}$ is not. In particular, the lack of the tensor action means that fewer functors become identified on the homological spectra level. 

Another example highlighting the difference is the bounded derived category of coherent sheaves on the projective line $\T^{\c}=\D^{\mrm{b}}(\msf{coh}(\mbb{P}^{1}_{k}))$. In this setting, the Balmer spectrum is $\mbb{P}^{1}_{k}$ and it classifies the thick $\otimes$-ideals in terms of the specialisation closed subsets of $\mbb{P}^{1}_{k}$. On the other hand, by using the triangulated equivalence with the perfect derived category of the path algebra of the Kronecker quiver, we show that $\sspec{\D^{\mrm{b}}(\msf{coh}(\mbb{P}^{1}_{k}))}$ is $\Z \sqcup \mbb{P}^{1}_{k}$ and that it parametrises all thick subcategories. So we see $\Spc(\D^{\mrm{b}}(\msf{coh}(\mbb{P}^{1}_{k})))$ and $\sspec{\D^{\mrm{b}}(\msf{coh}(\mbb{P}^{1}_{k}))}$ need not be homeomorphic.

Having discussed the relationship of $\sspec{\T^{\c}}$ with the Balmer spectrum, we also investigate the relationship to other spaces introduced as non-monoidal analogues of the Balmer spectrum, namely the fully and partially functorial spaces of Gratz-Stevenson \cite{GratzStevenson}, the central support of Krause \cite{Krausecentral}, and Matsui's spectrum \cite{Matsui}. The details of these comparisons are in \cref{comapsection}.

So far we have described properties of $\shspec{\T^{\c}}$ and $\sspec{\T^{\c}}$ and described the thick subcategories that the latter classifies. Earlier we discussed the connection between $\shspec{\T^{\c}}$ and integral rank functions, and throughout the paper understanding rank functions enables a more comprehensive understanding of both of these spaces. To see this, we first have to formalise the relationship between irreducible rank functions and $\shspec{\T^{\c}}$. Initially, we do this by introducing a topological space of irreducible rank functions, which is denoted $\msf{IrrRk}(\T^{\c})$, the points of which are the irreducible rank functions, with a basis of closed sets given by the irreducible rank functions vanishing on a given compact object. 

The following result describes the relationship between the topological space of irreducible rank functions and the shift-homological spectrum, and shows that a homeomorphism between the two can be established based on properties of the Ziegler spectrum.

\begin{thm*}[\ref{lem:RankGivesPrime}, \ref{thm:whenisKahomeo}]
Let $\T$ be a compactly generated triangulated category. There is a continuous injection
\[
\mc{K}\colon\msf{IrrRk}(\T^{\c})\hookrightarrow\shspec{\T^{\c}}
\]
from the space of irreducible rank functions to the shift-homological spectrum, given by $\rho\mapsto\msf{ker}(\widetilde{\rho})$. Furthermore, the following conditions are equivalent:
\begin{enumerate}
        \item the map $\mc{K}\colon \Irr{\T^\c} \to \shspec{\T^\c}$ is a homeomorphism;
        \item for every $[X]\in\closed{\T}$ the object $\bigoplus_{i \in \Z}\Sigma^iX$ is endofinite;
        \item any minimal $\Sigma$-invariant closed subset of $\msf{Zg}(\T)$ satisfies the isolation condition.
    \end{enumerate}
\end{thm*}

Given this, we see that the kernels of rank functions give a bountiful source of radical thick subcategories of $\T^{\c}$. Indeed, the kernel of any integral rank function is radical as illustrated in \cref{Kerrhoisradical}. Furthermore, when $\mc{K}$ is a homeomorphism, the prime thick subcategories are exactly the kernels of irreducible rank functions. As such $\sspec{\T^{\c}}$ can be interpreted as a parametrising space for intersections of kernels of integral rank functions.

Keeping in mind the utility of rank functions, we attempt to better understand them for certain derived categories. In \cref{sec:hered}, we completely describe the irreducible rank functions on $\D^{\t{b}}(\mod{R})$ for hereditary rings, and show that $\shspec{\D^{\t{b}}(\mod{R})}$ can be described in terms of $\msf{Zg}(\Mod{R})$. We then focus on tame hereditary algebras $\Lambda$, where it is shown, among other things, that the isolation condition holds for $\D(\Lambda)$, meaning that the map $\mc{K}$ is a homeomorphism, and so every thick subcategory of $\D^{\t{b}}(\mod{\Lambda})$ is parametrised by rank functions.

On the other hand, the general outlook cannot be so positive: the above discussion concerning the impossibility of a complete non-monoidal analogue of the Balmer spectrum, suggests that there are compactly generated triangulated categories such that not every thick subcategory of compact objects is radical. Indeed in \cref{hypersurfaces}, we provide an example illustrating this by calculating $\shspec{\D_{\t{sg}}(R)}$ and $\sspec{\D_{\t{sg}}(R)}$ for the singularity categories of all hypersurfaces of countable Cohen-Macaulay type. In particular, we show that for both hypersurfaces of countably infinite Cohen-Macaulay type, which are the rings $A_\infty=k[[x,y]]/x^2$ and $D_\infty=k[[x,y]]/x^2y$, the zero thick subcategory is not radical. 

\subsection*{Acknowledgements}
We are grateful to the anonymous referee for their careful reading and suggestions, in particular, for suggesting a shorter proof of \cref{orbit}.

IB and JW were supported by the project PRIMUS/23/SCI/006 from Charles University, and by Charles University Research Centre
program No. UNCE/24/SCI/022. AZ acknowledges the support by the Institute of Mathematics, Czech Academy of Sciences (RVO 67985840). 

JW would like to thank the Isaac Newton Institute for Mathematical
Sciences for support and hospitality during the programme `Topology, representation theory and higher structures' when work on this paper was undertaken. This was supported by EPSRC Grant Number EP/R014604/1. JW would also like to thank the Isaac Newton Institute for Mathematical Sciences, Cambridge, for support and hospitality during the programme `Equivariant homotopy theory in context', where later work on this paper was undertaken. This work was supported by EPSRC grant EP/Z000580/1.
 
\section{Background material}
Let us establish some notation. If $\A$ is an additive category and $\msf{X}$ is a class of objects we set \[^{\perp}\msf{X}=\{A\in\A:\Hom_{\A}(A,X)=0 \t{ for all }X\in\msf{X}\}.\] If $\A$ is abelian, we set $^{\perp_{1}}\msf{X}=\{A\in\A:\t{Ext}_{\A}^{1}(A,X)=0 \t{ for all }X\in\msf{X}\}$. We similarly define $\msf{X}^{\perp}$ and $\msf{X}^{\perp_{1}}$. In a triangulated category $\T$ we set \[^{\perp_{\Z}}\X=\{B\in\T:\Hom_{\T}(B,\Sigma^{i}X)=0 \t{ for all }i\in\Z, X\in\X\}\] and similarly define $X^{\perp_{\Z}}$.

Given a class of objects $\msf{X}\subseteq \A$, we let $\msf{add}(\msf{X})$ denote the class of objects which are summands of finite coproducts of objects in $\msf{X}$. We will denote by $\msf{Add}(\msf{X})$, resp. $\msf{Prod}(\msf{X})$, the subcategory consisting of summands of coproducts, resp. products, of objects in $\msf{X}$.
If $\msf{T}$ is a triangulated category, we write $\msf{thick}(\msf{X})$ to denote the thick closure of $\msf{X}$, the smallest triangulated subcategory containing $\msf{X}$ and closed under direct summands. We also let $\msf{add}^{\Sigma}(\msf{X})=\msf{add}(\Sigma^{i}X:i\in\Z,X\in\msf{X})$. 

\subsection{Localisation in locally coherent abelian categories}
We first recall some general facts about localisation in abelian categories. This material is standard, and can be found for instance in \cite{krbook} or \cite{stenstrom}.
\begin{chunk}
If $\mc{A}$ is an abelian category, then a full non-empty subcategory $\mc{S}\subseteq\mc{A}$ is \emph{Serre} if it is closed under subobjects, quotients, and extensions. Given a Serre subcategory $\mc{S}\subseteq\mc{A}$, there is a localisation $Q_{\mc{S}}\colon\mc{A}\to\mc{A}/\mc{S}$, where $\mc{A}/\mc{S}$ is an abelian category and $Q_{\mc{S}}$ is an exact functor. 
A Serre subcategory $\mc{S}$ is said to be \emph{localising} provided that $Q_{\mc{S}}$ admits a right adjoint $R_{\mc{S}}$, which is then automatically fully faithful. In this case, $R_{\mc{S}}$ induces an equivalence of categories 
\[
\begin{tikzcd}
\mc{A}/\mc{S} \arrow[r, "\sim"] & \{A\in\mc{A}:\Hom_{\mc{A}}(\mc{S},A)=0=\t{Ext}_{\mc{A}}^{1}(\mc{S},A)\}.
\end{tikzcd}
\]
If $\mc{A}$ is Grothendieck, then $\mc{S}$ is localising if and only if it is closed under coproducts.
\end{chunk}

\begin{chunk}
We mostly consider localisations of locally coherent categories. A Grothendieck abelian category $\mc{A}$ is \emph{locally coherent} provided every object is a filtered colimit of finitely presented objects, and the subcategory $\msf{fp}(\mc{A})$ of finitely presented objects is itself an abelian category.
Usually, the localising subcategories $\mc{L}$ of locally coherent categories $\mc{A}$ that we will consider will be those of \emph{finite type}, that is localising subcategories such that the right adjoint $R_{\mc{L}}\colon \mc{A}/\mc{L}\to \mc{A}$ preserves direct limits. These localising subcategories are closely related to torsion theories.
\end{chunk}

\begin{chunk}
A torsion theory $(\mc{T},\mc{F})$ for which $\mc{T}$ is closed under subobjects is called \emph{hereditary}. Hereditary torsion theories of a Grothendieck category $\mc{A}$ are exactly those determined by injective objects in $\mc{A}$: namely,  $(\mc{T},\mc{F})$ is hereditary if and only if there is a set of injective objects $\scr{E}\subseteq \mc{A}$  such that $\mc{T}=\{A\in\mc{A}:\Hom_{\mc{A}}(A,E)=0 \t{ for all }E\in\scr{E}\}$. In this case $\mc{F}$ consists of subobjects of products of objects in $\scr{E}$, see \cite[Chapter VI]{stenstrom}. A torsion theory $(\mc{T},\mc{F})$ such that $\mc{F}$ is closed under direct limits is said to be \emph{of finite type}. Localising subcategories of finite type satisfy the following equivalent conditions.
\end{chunk}

\begin{prop}[{\cite[Proposition A.4]{krspec}\label{prel:finitetype}}]
Let $\mc{A}$ be a locally coherent Grothendieck abelian category and $\mc{L}$ be a localising subcategory of $\mc{A}$. Then the following are equivalent:
\begin{enumerate}
\item $\mc{L}$ is of finite type;
\item there is a Serre subcategory $\mc{S}\subseteq\msf{fp}(\mc{A})$ such that $\mc{L}=\rlim\mc{S}$, where $\rlim\mc{S}$ denotes the direct limit closure of $\mc{S}$;
\item there is a hereditary torsion theory of finite type $(\mc{L},\mc{F})$.
\end{enumerate}
In this case, $\mc{S}=\msf{fp}(\mc{A})\cap\mc{L}$ and $\mc{F}=\{A\in\mc{A}:\Hom_{\mc{A}}(\mc{S},A)=0\}$.
\end{prop}

As an abuse of notation, if $\mc{S}$ is a Serre subcategory of $\msf{fp}(\mc{A})$, we shall let $Q_{\mc{S}}$ and $R_{\mc{S}}$ denote the adjoints corresponding to the localising subcategory $\rlim\mc{S}$. 

\begin{chunk}\label{prel:torsiontheories}

If we assume that  $\mc{A}$ is locally coherent and $(\mc{T},\mc{F})$ is a hereditary torsion theory of finite type, then $\mc{A}/\mc{T}$ is also locally coherent. Write $\scr{E}$ for the set of injective objects determining the hereditary torsion class $\mc{T}$. If $\mbb{E}_{\mc{T}}$ is an injective cogenerator of $\mc{A}/\mc{T}$, then the right adjoint $R_{\mc{T}}$ of the localisation functor gives an identification $\msf{Prod}(\mc{E})= \msf{Prod}(R_{\mc{T}}(\mbb{E}_{\mc{T}}))$. In particular, the set $\mc{E}$ can be chosen to be the singleton $\{R_{\mc{T}}(\mbb{E}_{\mc{T}})\}$. We shall often suppress the subscript $\mc{T}$ when the torsion theory in question is clear. Consequently, given any Serre subcategory $\mc{S}\subseteq\msf{fp}(\mc{A})$, we have $\mc{S}=\{A\in\msf{fp}(\mc{A}): \Hom_{\mc{A}}(A,R(\mbb{E}))=0\}$, where $\mbb{E}$ is the injective cogenerator of $\mc{A}/\rlim\mc{S}$.
\end{chunk}

\subsection{Purity in finitely accessible categories}
In this section, we recall some key definitions and facts regarding the pure structure of finitely accessible categories, see~\cite{psl,dac} for further details. Unless otherwise stated $\scr{C}$ will denote a finitely accessible category with products.
\begin{chunk}
An exact sequence $0\to X\xrightarrow{\alpha} Y\xrightarrow{\beta} Z\to 0$ in $\scr{C}$ is \emph{pure exact} provided the induced sequence $0\to \Hom_{\scr{C}}(A,X)\to \Hom_{\scr{C}}(A,Y)\to\Hom_{\scr{C}}(A,Z)\to 0$ is exact for all $A\in\msf{fp}(\scr{C})$. In this case, we say $\alpha$ is a \emph{pure monomorphism}, $\beta$ is a \emph{pure epimorphism}, $X$ is a \emph{pure subobject} of $Y$, and $Z$ is a \emph{pure quotient} of $Y$.
\end{chunk}

\begin{chunk}
An object $X\in\scr{C}$ is \emph{pure injective} if any pure monomorphism $X\to Y$ splits. This is equivalent to the summation map $\bigoplus_I X \to X$ factoring through the canonical map $\bigoplus_I X \to \prod_I X$ for any indexing set $I$, see~\cite[Theorem 5.4]{dac} for this and other equivalent characterisations. 
There is only a set of indecomposable pure injective objects in $\scr{C}$, which we will denote by $\msf{pinj}(\scr{C})$. 
\end{chunk}

\begin{chunk}
A full subcategory $\mc{D}\subseteq\scr{C}$ is said to be \emph{definable} if it is closed under pure subobjects, direct limits, and products. Given a class $\mc{X}\subseteq\scr{C}$, we let $\msf{Def}(\mc{X})$ denote the smallest definable subcategory containing $\mc{X}$. Explicitly, $\msf{Def}(\mc{X})$ may be obtained by first closing $\mc{X}$ under direct limits and products, and then under pure subobjects. Definable subcategories of $\scr{C}$ are uniquely determined by the indecomposable pure injective objects they contain: if $\mc{D}_{1}$ and $\mc{D}_{2}$ are definable subcategories of $\scr{C}$, then $\mc{D}_{1}=\mc{D}_{2}$ if and only if $\mc{D}_{1}\cap\msf{pinj}(\scr{C})=\mc{D}_{2}\cap\msf{pinj}(\scr{C})$. Furthermore $\mc{D} = \msf{Def}(\mc{D} \cap \msf{pinj}(\scr{C}))$, and $\mc{D}$ can be recovered from $\mc{D}\cap\msf{pinj}(\scr{C})$ as the pure subobjects of products of objects in $\mc{D}\cap\msf{pinj}(\scr{C})$, see \cite[\S 5.1.1]{psl}. 
\end{chunk}

\begin{chunk}\label{prel:Ziegler}
The \emph{Ziegler spectrum} of $\scr{C}$ is a topological space which encodes the purity information of $\scr{C}$. It is denoted by $\msf{Zg}(\scr{C})$ and its underlying set is $\msf{pinj}(\scr{C})$, and its closed sets are of the form $\mc{D}\cap\msf{pinj}(\scr{C})$ for $\mc{D}$ a definable subcategory of $\scr{C}$.
If $\mc{X}\subseteq\msf{Zg}(\scr{C})$, then its closure is $\msf{Def}(\mc{X})\cap\msf{Zg}(\scr{C})$. If $\mc{D}$ is a definable subcategory of $\scr{C}$, we define $\msf{Zg}(\mc{D})$ to be $\mc{D}\cap\msf{pinj}(\scr{C})$ endowed with the subspace topology.
\end{chunk}

\begin{chunk}\label{prel:endo}
An object $X\in\scr{C}$ is \emph{endofinite} if
\[
\t{length}_{\t{End}_{\scr{C}}(X)}\,\Hom_{\scr{C}}(A,X)<\infty
\]
for all $A\in\msf{fp}(\scr{C})$. Any coproduct or product of copies of an endofinite object is again endofinite, and any endofinite object is pure injective. Furthermore, if $X$ is endofinite then $\msf{Def}(X)=\msf{Add}(X)=\msf{Prod}(X)$; so if $X$ is indecomposable, then it is a closed point in $\msf{Zg}(\scr{C})$, since endofinite objects can be uniquely decomposed as a coproduct of indecomposable objects with local endomorphism rings. These results, among others concerning endofinite objects, can be found at \cite[Chapter 13]{krbook}.
\end{chunk}

\begin{chunk}\label{prel:deffunfinaccessible}
    A functor $F\colon \scr{C} \to \scr{D}$ between finitely accessible categories with products, is \emph{definable}, if it preserves direct limits and products. Such functors preserve pure injectives and pure exact sequences. See~\cite[\S 13]{psl} for more details.
\end{chunk}

\subsection{Purity in compactly generated triangulated categories}
While all the preceding concepts were defined for finitely accessible categories with products, as illustrated in \cite{krsmash} they are all equally well defined for compactly generated triangulated categories. We also refer the reader to \cite{birdwilliamsonhomological, krcoh} for further details. Throughout this section and the whole paper, $\T$ denotes a compactly generated triangulated category with the shift functor denoted by $\Sigma$, and with the full subcategory of compact objects denoted by $\T^{\c}$.
\begin{chunk}
The category $\Mod{\T^{\c}}$ of additive functors $(\T^{\c})^{\op}\to \ab$ is a locally coherent category, and the restricted Yoneda embedding
\[
\y\colon\T\to\Mod{\T^{\c}}, \quad X\mapsto \Hom_{\T}(-,X)\vert_{\T^{\c}}
\]
allows one to consider objects of $\T$ as additive functors. However, we note that the restricted Yoneda embedding is in general neither full nor faithful. We write $\mod{\T^\c} := \msf{fp}(\Mod{\T^\c})$ for the full subcategory of finitely presented modules.
\end{chunk}

\begin{chunk}
A triangle $X\to Y\to Z\to\Sigma X$ in $\T$ is \emph{pure} if and only if $0\to \y X\to \y Y\to \y Z\to 0$ is exact in $\Mod{\T^{\c}}$; in this case, similarly to the finitely accessible setting, the map $X\to Y$ is a \emph{pure monomorphism}, $Y\to Z$ is a \emph{pure epimorphism}, $X$ is a \emph{pure subobject} of $Y$, and $Z$ is a \emph{pure quotient} of $Y$. 
An object $X\in\T$ is \emph{pure injective} if and only if any pure monomorphism $X\to Y$ splits. We let $\msf{Pinj}(\T)$ denote the class of pure injective objects in $\T$. The functor $\y\colon \T \to \Mod{\T^\c}$ restricts to an equivalence of categories \[\y\colon\msf{Pinj}(\T)\xrightarrow{\sim} \msf{Inj}(\Mod{\T^{\c}})\] between the pure injective objects in $\T$ and the injective objects in $\Mod{\T^\c}$. In particular, there is only a set of indecomposable pure injective objects in $\T$ which we denote by $\msf{pinj}(\T)$.
\end{chunk}

\begin{chunk}
 A full subcategory $\mc{D}\subseteq\T$ is called \emph{definable} if and only if there is a Serre subcategory $\mc{S}\subseteq\mod{\T^{\c}}$ such that 
\[
\mc{D}=\{X\in\T:\Hom(f,\y X)=0\t{ for all }f\in\mc{S}\}.
\]
Similarly to the setting of finitely accessible categories with products, definable subcategories are uniquely determined by the indecomposable pure injective objects they contain, and can be completely recovered from these indecomposable objects. Again, given a class of objects $\mc{X}$ in $\T$, we let $\msf{Def}(\mc{X})$ denote the smallest definable subcategory containing $\mc{X}$. Then for a definable subcategory $\mc{D}$ we have $\mc{D}=\msf{Def}(\mc{D}\cap \msf{pinj}(\T))$. We note that definable subcategories are closed under products, coproducts, pure subobjects, and pure quotients.
\end{chunk}

\begin{chunk}\label{prel:fundamentalcorrespondence}
There is an order reversing bijection between Serre subcategories of $\mod{\T^{\c}}$ and definable subcategories of $\T$ which we refer to as the \emph{fundamental correspondence} after \cite{krcoh}, given as follows. If $\mc{D}\subseteq\T$ is definable then \[\scr{S}(\mc{D})=\{f\in\mod{\T^{\c}}:\Hom(f,\y X)=0\t{ for all }X\in\mc{D}\}\] is the corresponding Serre subcategory of $\mod{\T^\c}$, and if $\mc{S}\subseteq\mod{\T^{\c}}$ is Serre, then \[\scr{D}(\mc{S})=\{X\in\T:\Hom(f,\y X)=0\t{ for all }f\in\mc{S}\}\] is the corresponding definable subcategory of $\T$. 
\end{chunk}

\begin{chunk}
The Ziegler spectrum of $\T$ is defined in a similar way to \cref{prel:Ziegler}: the underlying set is $\msf{pinj}(\T)$ and the closed sets are $\mc{D}\cap\msf{pinj}(\T)$ for $\mc{D}\subseteq\T$ a definable subcategory. We denote the resulting topological space by $\msf{Zg}(\T)$.
\end{chunk}

\begin{chunk}
We note that the category $\msf{Flat}(\T^{\c})$ of flat functors (equivalently, cohomological functors $(\T^{\c})^{\op}\to \ab$, see \cite[Lemma 2.7]{krsmash}) is a finitely accessible category with products. The purity in $\T$ is equivalent to the purity in $\Flat{\T^{\c}}$: the restricted Yoneda embedding $\y\colon \T \to \msf{Flat}(\T^\c)$ induces a homeomorphism $\msf{Zg}(\T)\simeq\msf{Zg}(\Flat{\T^{\c}})$. Note that $M \in \Mod{\T^\c}$ is flat and pure injective if and only if it is injective.
\end{chunk}

\begin{chunk}\label{prel:endofinite}
An object $X\in\T$ is endofinite if and only if 
\[
\t{length}_{\t{End}_{\T}(X)}\,\Hom_{\T}(A,X)<\infty
\]
for all $A\in\T^{\c}$. This is equivalent to $\y X$ being an endofinite object in both $\msf{Flat}(\T^{\c})$ and $\Mod{\T^{\c}}$ by \cite[Lemma 10.6]{BelARZiegler}. We note that from \cref{prel:endo} we have that if $X\in\T$ is endofinite then $X$ is pure injective, and $\msf{Add}(X)=\msf{Prod}(X)=\msf{Def}(X)$, so that indecomposable endofinite objects of $\T$ are closed points in $\msf{Zg}(\T)$.
\end{chunk}

\begin{chunk}\label{prel:isolationcondition}
Conversely, when closed points of $\msf{Zg}(\T)$ are endofinite objects is dependent on the \emph{isolation condition}. The isolation condition holds for a closed subset $\msf{X}\subseteq\msf{Zg}(\T)$ if and only if for every $X\in\msf{X}$ such that $\{X\}$ is open in $\overline{\{X\}}$, the localisation $\mod{\T^{\c}}/\msf{ker}\,\Hom(-,\y X)$ contains a simple object.

We claim that if the isolation condition holds for a closed point $X$, then $X$ is endofinite. If $X$ is closed, then $\msf{Def}(X)$ does not contain any proper non-zero definable subcategories, so by the fundamental correspondence $\msf{ker}\,\Hom(-,\y X)$ is a maximal proper Serre subcategory of $\mod{\T^{\c}}$, and thus the localisation $\mod{\T^{\c}}/\msf{ker}\,\Hom(-,\y X)$ contains no non-zero proper Serre subcategories. In other words the Serre subcategory of $\mod{\T^{\c}}/\msf{ker}\,\Hom(-,\y X)$ generated by the simple object, which exists by assumption of the isolation condition holding, is the whole category $\mod{\T^{\c}}/\msf{ker}\,\Hom(-,\y X)$, meaning that it is a length category. It follows that $X$ is an endofinite object by \cite[Lemma 2.3, Subsection 2.4 (4)]{conde2022functorial}.
\end{chunk}

\begin{chunk}\label{prel:shiftinvpurity}
We will often use a coarsening of the Ziegler topology which is compatible with the shift. If $\mc{X}\subseteq\T$ is a $\Sigma$-invariant class of objects, then $\msf{Def}(\mc{X})$ is also $\Sigma$-invariant, as all the steps needed to obtain the definable closure commute with $\Sigma$. Given an arbitrary class of objects $\mc{X} \subseteq \T$, we write $\msf{Def}^\Sigma(\mc{X})$ for the smallest $\Sigma$-invariant definable subcategory of $\T$ containing $\mc{X}$. Note that $\msf{Def}^\Sigma(\mc{X}) = \msf{Def}(\{\Sigma^iX : X \in \mc{X},~i \in \mathbb{Z}\}).$
In particular, the set of $\Sigma$-invariant definable subcategories gives a coarser topology on $\msf{pinj}(\T)$ in which the closed sets are those of the form $\mc{D}\cap\msf{pinj}(\T)$ where $\mc{D}$ runs through all  $\Sigma$-invariant definable subcategories $\mc{D}$. We let $\msf{Zg}^{\Sigma}(\T)$ denote this space. This space was originally considered in \cite[\S 6.1.1]{wagstaffe}.

The universal property of the restricted Yoneda embedding $\y\colon\T\to\Mod{\T^{\c}}$ induces a shift on $\Mod{\T^{\c}}$, which we also denote by $\Sigma$, and which satisfies $\y \Sigma = \Sigma \y$. The  assignments $\mc{D}\mapsto\scr{S}(\mc{D})$ and $\mc{S}\mapsto\scr{D}(\mc{S})$ respect the action of the shift, and so the fundamental correspondence restricts to a bijection between $\Sigma$-invariant Serre subcategories of $\mod{\T^{\c}}$ and $\Sigma$-invariant definable subcategories of $\T$.
\end{chunk}

\begin{chunk}
    Let $\scr{C}$ be a finitely accessible category with products. A functor $H\colon \T \to \scr{C}$ is \emph{coherent} if it preserves coproducts, products, and sends pure triangles to pure exact sequences. If $H$ is a homological functor, it is coherent if and only if it preserves coproducts and products. The restricted Yoneda embedding $\y\colon \T \to \msf{Flat}(\T^\c)$ is the universal coherent functor, in the sense that for any coherent functor $H\colon \T \to \scr{C}$, there exists a unique definable functor $\widehat{H}\colon \Flat{\T^\c} \to \scr{C}$ in the sense of \cref{prel:deffunfinaccessible} such that $\widehat{H} \circ \y = H$. A functor $F\colon \T \to \msf{U}$ between compactly generated triangulated categories, is \emph{definable}, if it preserves coproducts, products, and pure triangles, or equivalently, if the composite $\y \circ F\colon \T \to \Flat{\msf{U}^\c}$ is coherent. If $F\colon \T \to \msf{U}$ is a triangulated functor, then $F$ is definable if and only if it preserves coproducts and products. For further details on coherent and definable functors see \cite{definablefunctors}.
\end{chunk}

\section{The shift-homological spectrum and rank functions}
In this section, for any compactly generated triangulated category $\T$ we introduce a space called the \emph{shift-homological spectrum}, which is built from the Serre subcategories of the functor category $\mod{\T^\c}$. We give an alternative description of this space in terms of the Ziegler spectrum of $\T$ which is convenient for calculating examples, and show that this space is intimately related to the space of irreducible rank functions on $\T^\c$. In later sections we will use this space as a stepping stone towards a space parametrising certain thick subcategories of $\T^\c$. 

\subsection{The shift-homological spectrum}
Recall that $\T$ denotes a compactly generated triangulated category with full subcategory of compact objects $\T^\c$.

\begin{defn}
A \emph{shift-homological prime} is a proper maximal $\Sigma$-invariant Serre subcategory of $\mod{\T^{\c}}$. We shall let $\shspec{\T^{\c}}$ denote the set of shift-homological primes in $\mod{\T^{\c}}$, and we call this the \emph{shift-homological spectrum}. 
\end{defn}

Later on in \cref{deftop} we will define a topology on the set $\shspec{\T^{\c}}$. We shall first investigate the pure injective objects in $\T$, or equivalently the injective objects in $\Mod{\T^{\c}}$, which determine shift-homological primes.

\begin{chunk}
If $\mc{S}$ is a $\Sigma$-invariant Serre subcategory of $\mod{\T^{\c}}$, then the localisation functor $Q_{\mc{S}}\colon \Mod{\T^{\c}}\to \Mod{\T^{\c}}/\rlim\mc{S}$ equips $\Mod{\T^{\c}}/\rlim\mc{S}$ with a shift, which we also denote by $\Sigma$. This induced shift functor commutes with $Q_{\mc{S}}$ by construction.
In this case, $\Sigma$-invariant Serre subcategories of $\mod{\T^{\c}}/\mc{S}$ are in bijection with $\Sigma$-invariant Serre subcategories of $\mod{\T^{\c}}$ that contain $\mc{S}$, see \cite[Lemma 2.2.8]{krbook}. In particular, if $\mc{B}\in\shspec{\T^{\c}}$, then $\mod{\T^{\c}}/\mc{B}$ contains no proper non-zero $\Sigma$-invariant Serre subcategories. We therefore have the following.
\end{chunk}

\begin{lem}\label{prop:DefIsSimple}
If $\mc{B}\in\shspec{\T^{\c}}$, then $\scr{D}(\mc{B})$, the definable subcategory of $\T$ corresponding to $\mc{B}$ under \cref{prel:fundamentalcorrespondence}, is a simple $\Sigma$-invariant definable subcategory, in that it is non-zero and contains no non-zero proper $\Sigma$-invariant definable subcategories.  \qed
\end{lem}

\begin{chunk}\label{definingXSigma}
Given a Serre subcategory $\mc{S}\subseteq\mod{\T^{\c}}$, it was shown in \cite[Theorem 3.7]{birdwilliamsonhomological} that the composition $Q_{\mc{S}}\circ \y$ gives an equivalence of categories \[\scr{D}(\mc{S}) \cap \msf{Pinj}(\T) \xrightarrow{\sim} \msf{Inj}(\Mod{\T^{\c}}/\rlim \mc{S}).\] By \cref{prop:DefIsSimple}, for any $\mc{B} \in \shspec{\T^\c}$ we have that if $X\in\scr{D}(\mc{B})\cap\msf{pinj}(\T)$, then $\msf{Def}^\Sigma(X)=\scr{D}(\mc{B})$. We now show how this can be used to recover $\mc{B}$.
\end{chunk}

\begin{lem}\label{lem:recoverhomprime}
Let $\mc{B}\in\shspec{\T^{\c}}$ and $X\in\scr{D}(\mc{B})$ be a non-zero pure injective object. Then 
\[
\mc{B}=\{f\in\mod{\T^{\c}}:\Hom(f,\Sigma^{i}\y X)=0 \t{ for all }i\in\Z\}.
\]
\end{lem}
\begin{proof}
As $\Sigma^{i}\y X$ is a non-zero injective object for each $i\in \Z$, the set \[\msf{S}_{\mc{B}}:=\{f\in\mod{\T^{\c}}:\Hom(f,\Sigma^{i}\y X)=0 \t{ for all }i\in\Z\}\] is a $\Sigma$-invariant Serre subcategory of $\mod{\T^{\c}}$. This subcategory is proper, since $X\in\scr{D}(\msf{S}_{\mc{B}})$ is non-zero. 
If $f\in\mc{B}$ then $\Hom(f,\y Z)=0$ for all $Z\in\scr{D}(\mc{B})$ by the fundamental correspondence. In particular, $\Hom(f,\Sigma^{i}\y X)=0$ for all $i \in \Z$, and consequently $\mc{B} \subseteq \msf{S}_{\mc{B}}$. Thus, by maximality of $\mc{B}$, the equality follows. 
\end{proof}

From \cref{prop:DefIsSimple}, we see that from a shift-homological prime $\mc{B}$ we obtain a simple $\Sigma$-invariant definable subcategory $\scr{D}(\mc{B})\subseteq\T$, which corresponds to the closure in $\msf{Zg}^{\Sigma}(\T)$ of any indecomposable pure injective in it. Conversely, from \cref{lem:recoverhomprime}, we see that from any indecomposable pure injective object in $\scr{D}(\mc{B})$ we can recover $\mc{B}$; in particular, by the fundamental correspondence, every shift-homological prime arises this way. Let us now formalise this.

\begin{chunk}\label{kolmogorovquotient}
Recall that if $\X$ is a topological space, there is an equivalence relation on $\X$ given by $x\sim y$ if and only if $\overline{\{x\}}=\overline{\{y\}}$, where $\overline{\{x\}}$ denotes the closure of the singleton set. The quotient space $\msf{K}(\X) := \X/{\sim}$ is called the \emph{Kolmogorov quotient} and is the universal approximation to $\X$ by a $\mrm{T}_0$-space. (More precisely, taking the Kolmogorov quotient is left adjoint to the inclusion of $\mrm{T}_0$-spaces into all topological spaces.) Note that the quotient map $\X \to \msf{K}(\X)$ is continuous, open, and closed.

In particular, for $\X = \msf{Zg}^{\Sigma}(\T)$ it follows that for all $X,Y\in\msf{Zg}^{\Sigma}(\T)$ one has $X\sim Y$ if and only if $\msf{Def}^{\Sigma}(X)=\msf{Def}^{\Sigma}(Y)$. Consequently, by \cref{prop:DefIsSimple}, if $\mc{B}\in\shspec{\T^\c}$ then all the points in $\scr{D}(\mc{B})\cap\msf{pinj}(\T)$ are equivalent under $\sim$. We let $\msf{KZg}^{\Sigma}(\T)$ denote the corresponding quotient space.
\end{chunk}

\begin{chunk}\label{defn:MapToKQ}
We define a map
\[
\Phi\colon\shspec{\T^{\c}}\to\msf{KZg}^{\Sigma}(\T)
\]
by sending $\mc{B}$ to $[X]$, the equivalence class under $\sim$ of any indecomposable pure injective $X\in\scr{D}(\mc{B})$. Since, as discussed in \cref{kolmogorovquotient}, all the points of $\scr{D}(\mc{B})\cap\msf{pinj}(\T)$ are in the same equivalence class, this map is well defined. The following is the first step in establishing the topological relationship between $\msf{Zg}^{\Sigma}(\T)$ and $\shspec{\T^{\c}}$.
\end{chunk}

\begin{prop}\label{prop:ShspecKZgbijection}
The map $\Phi$ gives a bijection $\shspec{\T^{\c}}\to\closed{\T}$ between the shift-homological spectrum and the set of closed points in $\msf{KZg}^{\Sigma}(\T)$.
\end{prop}
\begin{proof}

Let us first show that the image of $\Phi$ is contained within $\closed{\T}$. Take some $\mc{B} \in \shspec{\T^\c}$ and consider $\Phi(\mc{B})$. Let $X$ be a representative for $\Phi(\mc{B})$. The preimage of $\Phi(\mc{B})$ under the quotient map $\msf{Zg}^\Sigma(\T) \to \msf{KZg}^\Sigma(\T)$ consists of $Y\in\msf{pinj}(\T)$ such that $\msf{Def}^{\Sigma}(Y)=\msf{Def}^{\Sigma}(X)$, and is closed in $\msf{Zg}^{\Sigma}(\T)$ since it is of the form $\msf{Def}^{\Sigma}(X)\cap \msf{pinj}(\T)$ by \cref{prop:DefIsSimple}. This means that $\Phi(\mc{B})$ is a closed point in the quotient topology.

Let $\mc{B}, \mc{C} \in \shspec{\T^\c}$ be such that $\Phi(\mc{B}) = \Phi(\mc{C})$. Let $X$ be a representative for $\Phi(\mc{B})$ and $Y$ be a representative for $\Phi(\mc{C})$, so that $\msf{Def}^\Sigma(X) = \msf{Def}^\Sigma(Y)$. By \cref{lem:recoverhomprime} we have \[\mc{B} = \{f \in \mod{\T^\c}: \Hom(f,\Sigma^i\y X) = 0 \t{ for all }i\in\Z\}\] and similarly for $\mc{C}$. As $\msf{Def}^\Sigma(X) = \msf{Def}^\Sigma(Y)$, we see that $\Hom(f, \Sigma^i\y X) = 0$ for all $i \in \Z$ if and only if $\Hom(f,\Sigma^i\y Y) = 0$ for all $i \in \Z$. Therefore $\mc{B} = \mc{C}$, so $\Phi$ is injective.

We finally show that $\Phi$ is surjective. To this end, let $[X]\in\msf{KZg}_{\msf{Cl}}^{\Sigma}(\T)$ and consider $\msf{Def}^{\Sigma}(X)$, which is independent of the choice of representative. There is then a $\Sigma$-invariant Serre subcategory $\msf{S}_{X}:=\scr{S}(\msf{Def}^{\Sigma}(X))\subseteq\mod{\T^{\c}}$. As $X$ is not zero, $\msf{S}_X$ is proper. We will show that $\msf{S}_{X}$ is maximal, so we suppose that $\msf{S}_{X}\subseteq\mc{A}$, with $\mc{A}$ a proper $\Sigma$-invariant Serre subcategory. Then we have that $\scr{D}(\mc{A})\subseteq\scr{D}(\msf{S}_{X})=\msf{Def}^{\Sigma}(X)$ by the fundamental correspondence. It follows that if $Y\in\scr{D}(\mc{A})$, then $\overline{\{[Y]\}}\subseteq \overline{\{[X]\}}$ in $\msf{KZg}^{\Sigma}(\T)$: indeed, if $V$ is a closed subset of $\msf{KZg}^\Sigma(\T)$ containing $[X]$, then $\overline{\{X\}} \subseteq q^{-1}V$ where $q\colon \msf{Zg}^\Sigma(\T) \to \msf{KZg}^\Sigma(\T)$ denotes the quotient, so $[Y] \in V$. As we assumed that $[X]$ is a closed point, we deduce that $\overline{\{[Y]\}}=\overline{\{[X]\}}$, hence $Y\sim X$ and $\msf{Def}^{\Sigma}(X)=\scr{D}(\mc{A})$, giving $\mc{A}=\msf{S}_{X}$ so that $\msf{S}_{X}$ is maximal. By the definition of $\Phi$ we have that $\Phi(\msf{S}_{X})=[X]$, so $\Phi$ is surjective.
\end{proof}

\begin{chunk}
The use of the $\Sigma$-Ziegler topology means that, when we pass to $\msf{KZg}^{\Sigma}(\T)$, the indecomposable pure injective objects $X$ and $\Sigma^{i}X$ become identified. In $\msf{Zg}^{\Sigma}(\T)$, the points $X$ and $\Sigma^{i}X$ are topologically indistinguishable, but they are not the same. An alternative construction would have been to consider $\msf{Zg}(\T)/\Sigma$, the quotient space under the equivalence relation $X\sim Y$ if and only if $X=\Sigma^{i}Y$ for some $i\in\Z$. We will denote the projection map $\msf{Zg}(\T) \rightarrow \msf{Zg}(\T)/\Sigma$ by $\pi$. As the next result shows, the choice of approaching $\shspec{\T^{\c}}$ through $\msf{Zg}^{\Sigma}(\T)$ or $\msf{Zg}(\T)/\Sigma$ is irrelevant.
\end{chunk}

\begin{prop}\label{orbit}
The quotient map $f\colon \msf{Zg}^\Sigma(\T) \to \msf{Zg}(\T)/\Sigma$ induces a homeomorphism \[\msf{KZg}^{\Sigma}(\T)\xrightarrow{\sim} \msf{K}(\msf{Zg}(\T)/\Sigma).\]
\end{prop}

\begin{proof}
By functoriality of the Kolmogorov quotient we have the following commutative diagram:
\[
\begin{tikzcd}
\msf{Zg}^{\Sigma}(\T) \arrow[r, "f"] \arrow[d, "q_{1}"'] & \msf{Zg}(\T)/\Sigma \arrow[d, "q_{2}"] \\
\msf{KZg}^{\Sigma}(\T) \arrow[r, "\msf{K}(f)"'] & \msf{K}(\msf{Zg}(\T)/\Sigma).
\end{tikzcd}
\]
Recall that a subset $V$ of $\msf{Zg}(\T)/\Sigma$ is closed if and only if $\pi^{-1}V$ is closed in $\msf{Zg}(\T)$. Since this is a $\Sigma$-invariant closed set, this is also closed in $\msf{Zg}^{\Sigma}(\T)$ by definition, and as such, the map $f\colon \msf{Zg}^{\Sigma}(\T)\to \msf{Zg}(\T)/\Sigma$ given by $X\mapsto \pi X$ is continuous. Furthermore, this map is closed since if $U\subseteq\msf{Zg}^{\Sigma}(\T)$ is closed, then $\pi^{-1}f(U)=\{Y\in\msf{Zg}(\T):X=\Sigma^{i}Y \text{ for some } X \in U\} = U$, and hence is also closed in $\msf{Zg}(\T)$. 
Consequently, the induced map $\msf{K}(f)\colon\msf{KZg}^{\Sigma}(\T)\to\msf{K}(\msf{Zg}(\T)/\Sigma)$ is also closed and continuous. It is thus a homeomorphism if and only if it is bijective.

Let $X,Y \in \msf{pinj}(\T)$. From the definitions, we have $q_2f(X) = q_2 f(Y)$ if and only if $\overline{\{f(X)\}} = \overline{\{f(Y)\}} \in \msf{Zg}(\T)/\Sigma$ if and only if $\overline{\{X\}} = \overline{\{Y\}} \in \msf{Zg}^\Sigma(\T)$. From this we see that $\msf{K}(f)$ is an injection, and it is trivially surjective as $q_{2}\circ f$ is surjective. 
\end{proof}
Having better understood the points of the shift-homological spectrum $\shspec{\T^{\c}}$, we now endow it with a topology. The following definition is inspired by \cite{Balmernilpotence}.

\begin{defn}\label{deftop}
Given an object $A\in\T^{\c}$, define the \emph{shift-homological support} of $A$ to be
\[
\hssupp{A}=\{\mc{B}\in\shspec{\T^{\c}}:\y A\not\in\mc{B}\}.
\]
Using \cref{lem:recoverhomprime}, we see that $\mc{B}\in\hssupp{A}$ if and only if for any object $X\in\scr{D}(\mc{B})\cap\msf{pinj}(\T)$ there is an $i \in \Z$ such that $\Hom_{\T}(A,\Sigma^iX)\neq 0$.
\end{defn}


\begin{lem}\label{Lem:Supp}
The shift-homological support satisfies the following properties:
\begin{enumerate}
\item $\supph(0)=\varnothing$;
\item $\supph(X)=\supph(\Sigma X)$ for all $X \in \T^{\c}$;
\item $\supph(\oplus_{i=1}^n X_{i})=\cup_{i=1}^n\supph(X_{i})$ for all $X_i \in \T^{\c}$;
\item if $X\to Y\to Z$ is a triangle in $\T^{\c}$, then $\supph(Y)\subseteq\supph(X)\cup\supph(Z)$.
\end{enumerate}
\end{lem}

\begin{proof}
This is clear from the properties of $\Sigma$-invariant Serre subcategories.
\end{proof}

\begin{chunk}
    We put a topology on $\shspec{\T^{\c}}$ by taking a basis of closed sets to be the $\hssupp{A}$ as $A$ runs over $\T^{\c}$. That this indeed forms a basis for a topology is justified by the previous lemma.
\end{chunk}

\begin{chunk}
We now move to upgrade the bijection of \cref{prop:ShspecKZgbijection} to a homeomorphism. In order to do this, we must retopologise the set $\closed{\T}$, which, when equipped with the subspace topology coming from $\msf{Zg}^\Sigma(\T)$, is always a $\mrm{T}_{1}$-space. Note that the topology on $\shspec{\T^{\c}}$ and the Ziegler topology are on `opposite sides': for the Ziegler topology, closed sets are given by vanishing sets of functors, while the homological support is given by non-vanishing sets of functors. The relationship between Ziegler style and Zariski style topologies is discussed in \cite[\S 14]{psl}.
\end{chunk}

\begin{chunk}\label{defn:GZtoponZiegler}
We define the $\Sigma$-\emph{Gabriel-Zariski} topology on $\closed{\T}$ to have a basis of open sets given by
\[
[A]_{\Sigma}=\{[X] \in\closed{\T}:\Hom_{\T}(A,\Sigma^iX)=0 \t{ for all $i \in \Z$}\}
\]
as $A$ runs over $\T^{\c}$. First note that this is well defined, since if $X \sim Y$ in the sense of \cref{kolmogorovquotient}, then $\Hom_{\T}(A,\Sigma^iX)=0$ for all $i \in \Z$ if and only if $\Hom_{\T}(A,\Sigma^iY)=0$ for all $i \in \Z$. These sets also indeed form the basis for a topology since they cover (any $[X]$ is in $[0]_\Sigma$) and moreover $[C]_\Sigma \cap [D]_\Sigma = [C \oplus D]_\Sigma.$ We denote this topological space by $\msf{KZg}_\msf{Cl}^{\Sigma}(\T)^{\msf{GZ}}$.
\end{chunk}

The following should now come as no surprise.

\begin{thm}\label{homeomorphism}
The map \[\Phi\colon \shspec{\T^\c} \to \msf{KZg}_\msf{Cl}^{\Sigma}(\T)^{\msf{GZ}}\] of \cref{prop:ShspecKZgbijection} is a homeomorphism.
\end{thm}
\begin{proof}
We first show that $\Phi$ is continuous. It is enough to check that the preimage of a basic open set is open. If $C\in\T^{\c}$, then
\[
\Phi^{-1}([C]_{\Sigma})=\{\mc{B}\in\shspec{\T^{\c}}:\Hom_{\T}(C,\Sigma^iX)=0 \text{ for $X\in\scr{D}(\mc{B})\cap\msf{pinj}(\T)$ and all $i \in \Z$}\}.
\]
In particular, we see that $\Phi^{-1}([C]_{\Sigma}) = \supph(C)^\c$ is an open set.

We next show that $\Phi$ is closed. We see that
\[
\Phi(\supph(C))=\{\Phi(\mc{B}): \text{for any $X\in\scr{D}(\mc{B})\cap\msf{pinj}(\T)$ there is an $i \in \Z$ such that }\Hom_{\T}(C,\Sigma^iX)\neq 0\}.
\]
Since $\Phi$ is a bijection by \cref{prop:ShspecKZgbijection}, we see that this is nothing other than $[C]_{\Sigma}^{\c}$ which is closed. As such, $\Phi$ sends basic closed sets to closed sets, and hence is closed. Therefore $\Phi$ is a homeomorphism as claimed.
\end{proof}

\begin{rem}
The homeomorphism given in \cref{homeomorphism} is particularly useful for computing examples as we will demonstrate later.
\end{rem}

We now show that for triangulated categories generated by a single compact object, the space $\shspec{\T^{\c}}$ cannot be empty.

\begin{prop}\label{prop:containedinmaxserre}
Suppose $\T^{\c}$ is generated by an object $G\in\T^{\c}$, i.e., $\T^\c = \msf{thick}(G)$. Then every proper $\Sigma$-invariant Serre subcategory of $\mod{\T^{\c}}$ is contained in a shift-homological prime. In particular, $\shspec{\T^\c}$ is non-empty. 
\end{prop}  
\begin{proof}
Let $\mc{S}$ be a proper $\Sigma$-invariant Serre subcategory of $\mod{\T^\c}$. Consider
\[
\mf{P}_{\mc{S}}=\{\mc{L}\subseteq\mod{\T^{\c}}:\mc{L} \t{ is a proper $\Sigma$-invariant Serre subcategory such that $\mc{S}\subseteq \mc{L}$}\}
\]
which is non-empty as $\mc{S}\in\mf{P}_{\mc{S}}$. Note that a $\Sigma$-invariant Serre subcategory $\mc{L}$ is proper if and only if $\y G\not \in \mc{L}$. Indeed, if $\y G\in\mc{L}$ then $\y^{-1}\mc{L}$ is a thick subcategory containing $G$, and therefore $\y^{-1}\mc{L} = \T^\c$, so that $\mc{L}$ contains all finitely presented projectives and thus coincides with $\mod{\T^{\c}}$. Now, let $\{\mc{L}_{i}\}$ be a chain of elements of $\mf{P}_{\mc{S}}$. Then
\[
\bigcup_{i}\mc{L}_{i}
\]
is a $\Sigma$-invariant Serre subcategory, and it does not contain $\y G$ (otherwise some $\mc{L}_i$ is not proper). Therefore Zorn's lemma gives the existence of a maximal element $\mf{P}_{\mc{S}}$ which is a shift-homological prime.
\end{proof}

\subsection{A topology on irreducible rank functions}\label{ss:rankfunctions}
We now investigate the relationship between $\shspec{\T^{\c}}$ and the collection of irreducible rank functions on $\T^{\c}$. We first recall the definition as given in \cite{chuanglazarev}.

\begin{chunk}\label{prel:rankfunctions}
A \emph{rank function} on $\T^{\c}$ assigns to each object $X\in\T^{\c}$ a nonnegative real number $\rho(X)$ such that the following three conditions hold:
\begin{enumerate}
\item for any $X\in\T^{\c}$, we have $\rho(\Sigma X)=\rho(X)$;
\item for any $X,Y\in\T^{\c}$, we have $\rho(X\oplus Y)=\rho(X)+\rho(Y)$;
\item for any triangle $X\to Y\to Z$ in $\T^{\c}$, there is an inequality $\rho(Y)\leq \rho(X)+\rho(Z)$.
\end{enumerate}

Equivalently, rank functions can be defined as an assignment of a nonnegative real number $\rho(f)$ to each morphism $f\in\T^{\c}$, such that the following three conditions hold:
\begin{enumerate}
\item for any morphism $f\in\T^{\c}$, we have $\rho(\Sigma f)=\rho(f)$;
\item for any morphisms $f,g\in\T^{\c}$, we have $\rho(f\oplus g)=\rho(f)+\rho(g)$;
\item for any triangle $X\xrightarrow{f} Y\xrightarrow{g} Z$ in $\T^{\c}$, there is an equality $\rho(\text{id}_Y)= \rho(f)+\rho(g)$.
\end{enumerate}

Passage from a rank function defined on morphisms to a rank function defined on objects is given by $\rho(X)=\rho(\text{id}_X)$. We consider the following restrictions on a rank function $\rho$. Firstly, we suppose that $\rho(f)\in\Z$ for every morphism $f\in\T^{\c}$; such a rank function is called \emph{integral}.
Any integral rank function $\rho$ can, as shown in \cite[Theorem 4.2]{conde2022functorial}, be decomposed as
\[
\rho=\Sigma_{I}\rho_{i}
\]
where the rank functions $\rho_{i}$ are \ti{irreducible}, that is nonzero integral rank functions that cannot be expressed as a sum of nonzero integral rank functions. Furthermore,
for each $X \in \T^{\c}$, the set $\{j\in I:\rho_{j}(X)\neq 0\}$ is finite.
\end{chunk}

\begin{chunk}\label{rankviaSerre}
    Let us now relate irreducible rank functions to $\shspec{\T^{\c}}$. In order to do this, recall from \cite[Theorem 3.11]{conde2022functorial} that for any rank function $\rho$ on $\T^\c$, there is an associated $\Sigma$-invariant additive function $\tilde{\rho}\colon \mod{\T^\c} \to \mbb{R}_{\geq 0}$, defined by $\tilde{\rho}(\msf{Im}(\y f)) = \rho(f)$. From this it is immediate that $\tilde{\rho}(\y C) = \rho(C)$. By~\cite[Theorem 4.3]{conde2022functorial} (and its proof), the assignment $\rho \mapsto \msf{ker}(\tilde{\rho})$ gives a bijection between irreducible rank functions on $\T^\c$, and $\Sigma$-invariant Serre subcategories $\mc{S}$ of $\mod{\T^\c}$ such that $\mod{\T^\c}/\mc{S}$ is a length category with a single simple object up to the action of $\Sigma$. This suggests that we can consider the set of irreducible rank functions as a subset of $\shspec{\T^{\c}}$. Let us make this more precise.
\end{chunk}

\begin{chunk}
    We write $\msf{IrrRk}(\T^\c)$ for the set of irreducible rank functions on $\T^\c$. We equip this set with a topology with a basis of closed sets given by
    \[\rsupp{C} = \{\rho \in \Irr{\T^\c} : \rho(C) \neq 0\}\] as $C$ ranges over $\T^\c$. (This indeed forms the basis for a topology since $\rsupp{C} \cup \rsupp{D} = \rsupp{C \oplus D}$, and $\cap_{C \in \T^\c}\rsupp{C} = \varnothing$ as $\rho(0)=0$ for all rank functions $\rho$.)
\end{chunk}

\begin{lem}\label{lem:RankGivesPrime}
There is a continuous injective map \[\mc{K}\colon \Irr{\T^\c} \to \shspec{\T^\c}\] defined by $\mc{K}(\rho) = \msf{ker}(\tilde{\rho})$ for $\rho \in \Irr{\T^\c}$.
\end{lem}
\begin{proof}
We first show that the assignment $\rho \mapsto \msf{ker}(\tilde{\rho})$ produces a shift-homological prime. By~\cref{rankviaSerre}, the localisation $\mod{\T^{\c}}/\msf{ker}(\tilde{\rho})$ is a length category with a unique simple up to $\Sigma$. Since in length categories Serre subcategories are determined by the simple objects they contain, there is a bijection between $\Sigma$-invariant subsets of isoclasses of simple objects and $\Sigma$-invariant Serre subcategories of $\mod{\T^{\c}}/\msf{ker}(\tilde{\rho})$. It follows that $\mod{\T^{\c}}/\msf{ker}(\tilde{\rho})$ has no non-trivial  proper $\Sigma$-invariant Serre subcategories. Hence there are no proper $\Sigma$-invariant Serre subcategories of $\mod{\T^{\c}}$ containing $\msf{ker}(\tilde{\rho})$, which is thereby maximal. Injectivity follows immediately from the bijection recalled in \cref{rankviaSerre}. For continuity, we have \[\mc{K}^{-1}(\hssupp{C}) = \{\rho \in \Irr{\T^\c} : \y C \not\in \msf{ker}(\tilde{\rho})\} = \rsupp{C}\] and hence $\mc{K}$ is continuous.
\end{proof}

\begin{cor}\label{cor:surjsuffices}
If $\mc{K}$ is surjective, then it is a homeomorphism.    
\end{cor}

\begin{proof}
As a continuous closed bijection is a homeomorphism, it suffices to show that if $\mc{K}$ is surjective then it is closed. Now,
\[
\mc{K}(\rsupp{C})=\{\mc{K}(\rho):\rho(C)\neq 0\} = \{\msf{ker}(\widetilde{\rho}):\y C\not\in\msf{ker}(\widetilde{\rho})\}
\]
so if $\mc{K}$ is surjective, then every shift-homological prime is of the form $\msf{ker}(\widetilde{\rho})$, and we therefore see that $\mc{K}(\rsupp{C)}=\hssupp{C}$, proving that $\mc{K}$ is closed as desired.
\end{proof}

For an object $X\in\T$, we set \[X^{\Sigma}=\bigoplus_{i\in\Z}\Sigma^{i}X \quad \text{and} \quad X_\Sigma = \prod_{i \in \Z}\Sigma^iX\]and note that these are both $\Sigma$-invariant objects. We record the following observation about $\Sigma$-invariant endofinite objects.

\begin{lem}\label{lem:endofiniteupperlower}
Let $X\in\T$, then $X_{\Sigma}$ is endofinite if and only if $X^{\Sigma}$ is endofinite.
\end{lem}
\begin{proof}
Recall from \cref{prel:endofinite} that if $E \in \T$ is endofinite, then every object of $\msf{Def}(E)$ is also endofinite. Therefore, the statement follows if it is shown that $\msf{Def}(X_\Sigma) = \msf{Def}(X^\Sigma)$. Note that $\Sigma^{i}X\in\msf{Def}(X^{\Sigma})$ as definable subcategories are closed under summands. As they are also closed under products, we see that $X_{\Sigma}\in\msf{Def}(X^{\Sigma})$. The inclusion $X^{\Sigma}\in\msf{Def}(X_{\Sigma})$ holds by a similar argument.
\end{proof}

We now give a complete characterisation of when the map $\mc{K}$ is a homeomorphism.

\begin{thm}\label{thm:whenisKahomeo}
   The following conditions are equivalent:
    \begin{enumerate}
        \item\label{kappaisbijection} the map $\mc{K}\colon \Irr{\T^\c} \to \shspec{\T^\c}$ of \cref{lem:RankGivesPrime} is a homeomorphism;
        \item\label{endofinitecondition} for every $[X]\in\closed{\T}$ the object $X^{\Sigma}$ is endofinite;
        \item\label{isolated} any minimal $\Sigma$-invariant closed subset of $\msf{Zg}(\T)$ satisfies the isolation condition (see \cref{prel:isolationcondition}).
    \end{enumerate}
In this case, the image of the map $\Phi$ of \cref{prop:ShspecKZgbijection} consists of the equivalence classes of indecomposable endofinite objects $X\in\T$ such that $X^{\Sigma}$ is endofinite.
\end{thm}

\begin{proof}
Let us first observe that (\ref{endofinitecondition}) is well-defined. That is, if $X,Y\in\msf{pinj}(\T)$ are such that $\msf{Def}^{\Sigma}(X)=\msf{Def}^{\Sigma}(Y)$, we claim that $X^{\Sigma}$ is endofinite if and only if $Y^{\Sigma}$ is. Note that $\msf{Def}^{\Sigma}(X)=\msf{Def}(X^{\Sigma})$, and then the claim follows in a similar way to the proof of \cref{lem:endofiniteupperlower}.

We now prove the equivalence of (\ref{kappaisbijection}) and (\ref{endofinitecondition}). By \cref{cor:surjsuffices}, $\mc{K}$ is a homeomorphism if and only if it is surjective, which is equivalent to every $\mc{B}\in\shspec{\T^{\c}}$ being of the form $\mc{B}=\msf{ker}(\widetilde{\rho})$ for an irreducible rank function $\rho$. This is, by \cite[Theorem 5.5]{conde2022functorial}, further equivalent to $\scr{D}(\mc{B})$ being a simple $\Sigma$-invariant definable subcategory consisting solely of endofinite objects for any $\mc{B}\in\shspec{\T^{\c}}$. This occurs if and only if $X^{\Sigma}$ is endofinite for any $X\in\scr{D}(\mc{B})\cap\msf{pinj}(\T)$ by \cref{prop:DefIsSimple}. Combining these equivalences with the compatibility of the Kolmogorov quotient relation with endofiniteness, we see that $\mc{K}$ is a homeomorphism if and only $X^{\Sigma}$ is endofinite for any representative of $\Phi(\mc{B})$ for $\mc{B}\in\shspec{\T^{\c}}$. Since $\Phi$ is a bijection by \cref{prop:ShspecKZgbijection}, this proves the equivalence of the first two items.

For (\ref{kappaisbijection}) implies (\ref{isolated}), if $\mc{K}$ is a homeomorphism, then every $\mc{B}\in\shspec{\T^{\c}}$ gives a length category $\mod{\T^{\c}}/\mc{B}$ by \cref{lem:RankGivesPrime} and \cite[Theorem 4.3]{conde2022functorial}. Since every minimal $\Sigma$-invariant closed subset of $\msf{Zg}(\T)$ is of the form $\scr{D}(\mc{B})\cap\msf{Zg}(\T)$ for some $\mc{B}\in\shspec{\T^{\c}}$, we obtain, from \cite[Theorem 4.3]{conde2022functorial}, that such a closed set satisfies the isolation condition.

Conversely, to see that (\ref{isolated}) implies (\ref{kappaisbijection}) we suppose that $\mc{B}$ is a shift-homological prime and set $\mc{X}:=\scr{D}(\mc{B})\cap \msf{pinj}(\T)$. Then $\mc{X}$ is a minimal $\Sigma$-invariant closed subset of $\msf{Zg}(\T)$ and $\mc{B}=\scr{S}(\msf{Def}(\mc{X}))$. By assumption the isolation condition holds for $\mc{X}$, so the localisation $\mod{\T^{\c}}/\mc{B}$ contains a simple object $S$. Since $\mc{B}$ is a shift-homological prime, the localisation $\mod{\T^{\c}}/\mc{B}$ has no non-zero proper $\Sigma$-invariant Serre subcategories. Consequently, the Serre subcategory of $\mod{\T^\c}/\mc{B}$ generated by $\{\Sigma^{i}S:i\in\Z\}$ is all of $\mod{\T^{\c}}/\mc{B}$, which is therefore a length category. In particular, $\mc{B}$ coincides with $\msf{ker}(\tilde{\rho})$ for some irreducible rank function $\rho$, again by \cite[Theorem 4.3]{conde2022functorial}. Therefore $\mc{K}$ is surjective and hence is a homeomorphism by \cref{cor:surjsuffices}.
\end{proof}

Since the isolation condition holding for all of $\msf{Zg}(\T)$ means it holds for any closed subset of $\msf{Zg}(\T)$, we obtain the following immediately.
\begin{cor}\label{cor:Kishomeo}
If the isolation condition holds for $\msf{Zg}(\T)$, then $\mc{K}\colon \Irr{\T^\c} \to \shspec{\T^\c}$ is a homeomorphism. \qed
\end{cor}

\begin{ex}
The isolation condition holds in all of the following examples, and hence the map $\mc{K}$ is a homeomorphism by \cref{cor:Kishomeo}:
\begin{enumerate}
    \item $\msf{K}(\msf{Proj}(A))$ for any derived discrete algebra $A$, see \cite[Theorem A]{ZgDDA};
    \item $\D(A)$ for any derived discrete algebra $A$, see \cite[Main Theorem]{bobinskikrause};
    \item $\D(\Lambda)$ for any tame hereditary artin algebra $\Lambda$, see \cref{thm:isolationhered} below.
\end{enumerate}
\end{ex}

\section{The shift-spectrum and parametrising thick subcategories}
In the previous section, we constructed a topological space $\shspec{\T^\c}$ by considering certain Serre subcategories of the functor category $\mod{\T^\c}$. In this section, we build on this to construct another topological space whose points are certain thick subcategories: as such, we pass from the functorial level back down to the triangulated level. This construction mimics the relationship between the homological spectrum and Balmer spectrum of a rigidly-compactly generated tensor-triangulated category, and we show that there is a close similarity between certain results obtained in either passage.

\subsection{The shift-spectrum} In the tensor-triangular setting, one can view the Balmer spectrum as being obtained from the homological spectrum in two ways; the perspective of \cite{Balmernilpotence} was that the Balmer spectrum is the `preimage' of the homological spectrum under $\y$, and, on the other hand, it was proved in \cite{BHScomparison} that the Balmer spectrum can also be realised as the Kolmogorov quotient of the homological spectrum. Using these two perspectives as a motivation, we consider the analogous constructions for the shift-homological spectrum. For the first perspective, note that for any $\Sigma$-invariant Serre subcategory $\mc{S}$ of $\mod{\T^\c}$, the subcategory $\y^{-1}\mc{S}$ of $\T^\c$ is thick. 

\begin{defn}
A thick subcategory $\msf{P}\subseteq\T^{\c}$ is said to be a \emph{shift-prime} provided $\msf{P}=\y^{-1}\mc{B}$ for some $\mc{B}\in\shspec{\T^{\c}}$. The \emph{shift-spectrum} of $\T^{\c}$, denoted $\sspec{\T^{\c}}$, has as points the shift-prime thick subcategories of $\T^{\c}$, and a basis of closed sets given by
\[
\ssupp{A}=\{\msf{P}\in\sspec{\T^{\c}}:A\not\in\msf{P}\},
\]
as $A$ runs through $\T^{\c}$. That these closed sets form a basis follows from the observation that $\ssupp{C} \cup \ssupp{D} = \ssupp{C \oplus D}$, and $\cap_{C \in \T^\c}\ssupp{C} = \varnothing$.
\end{defn}

Given that the points of $\sspec{\T^{\c}}$ are determined by maximal $\Sigma$-invariant Serre subcategories, one may wonder whether maximal thick subcategories are also points of $\sspec{\T^{\c}}$. This is frequently the case, as the next result shows. 
\begin{lem}\label{lem:maximalthickisprime}
Suppose $\T^{\c}$ is generated by an object $G\in\T^{\c}$. Let $\msf{L}\subseteq\T^{\c}$ be a proper maximal thick subcategory of $\T^{\c}$. Then $\msf{L}\in\sspec{\T^{\c}}$.
\end{lem}

\begin{proof}
Let $q\colon\T^{\c}\to\T^{\c}/\msf{L}$ be the Verdier localisation and let $q^{*}\colon\mod{\T^{\c}}\to\mod{\T^{\c}/\msf{L}}$ be the unique exact functor such that $q^{*}\circ\y = \y \circ q$. It is clear that $\msf{ker}(q^{*})$ is a proper $\Sigma$-invariant Serre subcategory of $\mod{\T^{\c}}$. By \cref{prop:containedinmaxserre}, there is a homological prime $\mc{B}\in\shspec{\T^{\c}}$ containing $\msf{ker}(q^{*})$. Then $\y^{-1}\mc{B}\in\sspec{\T^{\c}}$ contains $\msf{L}$, and therefore by maximality we have $\y^{-1}\mc{B}=\msf{L}$, as desired.
\end{proof}

We see that whenever $\T^\c$ is generated by a single compact object, the set $\sspec{\T^{\c}}$ cannot be empty.

\begin{rem}
Note that the converse to \cref{lem:maximalthickisprime} is not true. For the derived category of the path algebra of the Kronecker quiver, the smallest thick subcategory containing all regular tubes, and the smallest thick subcategory containing all regular tubes but one, both give examples of shift-prime thick subcategories, see \cref{th:kronecker}. Therefore shift-prime thick subcategories need not be maximal.
\end{rem}

\begin{chunk}\label{definingalpha}
As $\Phi\colon \shspec{\T^\c} \to \closed{\T}^\msf{GZ}$ of \cref{homeomorphism} is a homeomorphism, we may define a map 
\[
\alpha\colon \closed{\T}^\msf{GZ} \to \sspec{\T^\c}
\]
via the commutative diagram
\[
\begin{tikzcd}
\closed{\T}^{\msf{GZ}} \arrow[r, "\Phi^{-1}"] \arrow[dr,swap, "\alpha"] & \shspec{\T^{\c}} \arrow[d, "\y^{-1}"] \\ {} & \sspec{\T^{\c}}.
\end{tikzcd}
\]
Therefore one sees that \[\alpha([X])=\{C\in\T^{\c}:\Hom_{\T}(C,\Sigma^{i}X)=0 \t{ for all }i\in\Z\} = {}^{\perp_\Z}X
.\] It is immediate from the definition of supports that $\y^{-1}$ is a closed and continuous surjection, and hence $\alpha$ is also a closed and continuous surjection.
\end{chunk}

In tensor-triangular geometry, the Nerves of Steel conjecture states that the map $\y^{-1}\colon\hspec{\T^\c}\to\Spc(\T^\c)$ is a homeomorphism (equivalently, injective), see~\cite{Balmernilpotence}. This is known to be true in all cases where both spaces have been computed. In our setting, such a statement would be equivalent to $\alpha$ being injective. The following example, which we give further details on later, provides an immediate counterexample to such a conjecture holding in this setting.

\begin{ex}
    Let $\T=\msf{D}(\Z)$ be the derived category of $\Z$-modules. For a fixed prime $p \in \Z$ and distinct $j,k>0$, the modules $\Z/p^j$ and $\Z/p^k$ (viewed as complexes in degree 0) give distinct points in $\msf{KZg}^\Sigma_\msf{Cl}(\msf{D}(\Z))$, see \cref{dedekind} for more details. However, $\alpha(\Z/p^j) = \alpha(\Z/p^k)$ since $\Hom_{\D(\Z)}(A,\Sigma^i\Z/p^{j})=0$ for all $i \in \Z$ if and only if $\Hom_{\D(\Z)}(A,\Sigma^i\Z/p^{k})=0$ for all $i \in \Z$. The following lemma explains that this is a frequent phenomenon.
\end{ex}

\begin{lem}\label{lem:thicksamepreimage}
	Let $[X], [Y] \in \msf{KZg}_\msf{Cl}^{\Sigma}(\T)$. If $\msf{Loc}(X) = \msf{Loc}(Y)$ or $\msf{Coloc}(X)=\msf{Coloc}(Y)$, then $\alpha([X]) = \alpha([Y])$.
\end{lem}
\begin{proof}
If $A\in\T^{\c}$, then $\{Z\in\T:\Hom_{\T}(A,\Sigma^{i}Z)=0 \t{ for all }i\in\Z\}$ is both localising and colocalising. In particular, we see that if $\msf{Loc}(X) = \msf{Loc}(Y)$ or $\msf{Coloc}(X)=\msf{Coloc}(Y)$ then $A\in\alpha([X])$ if and only if $A\in\alpha([Y])$ from the definition of $\alpha$.
\end{proof}

While a $\Sigma$-Nerves of Steel statement fails at the first hurdle, we can show that the perspective of the Balmer spectrum being the Kolmogorov quotient of the homological spectrum does hold true for our spaces. More specifically we have the following.

\begin{prop}\label{prop:KQ}
The shift-spectrum $\sspec{\T^\c}$ is a $\mrm{T}_0$-space and there is a homeomorphism \[\msf{K}(\shspec{\T^\c}) \xrightarrow{\sim} \sspec{\T^\c}.\]     
\end{prop}
\begin{proof}
Firstly, it is immediate that $\sspec{\T^\c}$ is a $\mrm{T}_0$-space, since if $\msf{P}_{1},\msf{P}_{2}\in\sspec{\T^{\c}}$ are such that $\msf{P}_{1}\neq \msf{P}_{2}$, then without loss of generality there is an $A\in\msf{P}_{1}\backslash\msf{P}_{2}$. In particular, $\ssupp{A}$ topologically distinguishes $\msf{P}_{1}$ and $\msf{P}_{2}$.

Recall from \cref{kolmogorovquotient} that the Kolmogorov quotient of $\shspec{\T^\c}$ is the quotient of $\shspec{\T^\c}$ under the equivalence relation $\mc{B}\sim \mc{C}$ if and only if $\overline{\{\mc{B}\}} = \overline{\{\mc{C}\}}$. Let $\mc{B}$, $\mc{C}$ be two points in $\shspec{\T^\c}$. We have $\overline{\{\mc{B}\}} = \overline{\{\mc{C}\}}$ if and only if the basic closed sets containing $\mc{B}$ and $\mc{C}$ coincide. This is the case if and only if
\[
\{A \in \T^{\c}: \mc{B} \in \supph(A)\} = \{A\in \T^{\c}: \mc{C} \in \supph(A)\}.
\]
 This happens if and only if an object $A\in \T^{\c}$ belongs to $\y^{-1}\mc{B}$ precisely when it belongs to $\y^{-1}\mc{C}$. So we deduce that $\overline{\{\mc{B}\}} = \overline{\{\mc{C}\}}$ if and only if $\y^{-1}\mc{B}=\y^{-1}\mc{C}$.

 Since $\sspec{\T^\c}$ is $\mrm{T}_0$, by the universal property of the Kolmogorov quotient, we obtain a continuous map $f\colon \msf{K}(\shspec{\T^\c}) \to \sspec{\T^\c}$ making the diagram
 \[
 \begin{tikzcd}
     \shspec{\T^\c} \ar[d,"q"'] \ar[dr, "\y^{-1}"] & \\
     \msf{K}(\shspec{\T^\c}) \ar[r, "f"'] & \sspec{\T^\c}
 \end{tikzcd}
 \]
 commute. Since the Kolmogorov quotient relation is the same as the equivalence relation given by $\y^{-1}$ it is immediate that $f$ is a bijection. As $\y^{-1}$ is closed we deduce that $f$ is also closed, and hence is a homeomorphism.
\end{proof}

\begin{chunk}\label{chunk:shiftspecrank}
When the conditions of \cref{thm:whenisKahomeo} hold, i.e., when $\mc{K}\colon \Irr{\T^\c} \to \shspec{\T^\c}$ is a homeomorphism, a thick subcategory of $\T^\c$ is a shift-prime if and only if it is the kernel of an irreducible rank function. In this case, we see that $\sspec{\T^{\c}}$ can be identified with the space of irreducible rank functions modulo the relation that the kernels agree. 
We note that even in very simple categories, there are examples of irreducible rank functions which differ but have the same kernel, such as the cluster category of type $A_3$, see \cite[\S 6]{conde2022functorial}.
\end{chunk}

\subsection{Relating the spaces to the homological and Balmer spectra}\label{subsec:maptobalmer}

So far we have defined two spaces, the first of which, $\shspec{\T^{\c}}$, mimics the construction of the homological spectrum $\hspec{\T^\c}$ from tensor-triangular geometry. From this, we defined $\sspec{\T^{\c}}$, which we showed can be realised as the Kolmogorov quotient of $\shspec{\T^{\c}}$. This situation replicates the relationship between the homological and Balmer spectrum. We now establish a relationship between $\shspec{\T^{\c}}$ and $\hspec{\T^{\c}}$, and as a result obtain a continuous, injective comparison map  $\Spc(\T^{\c})\to\sspec{\T^{\c}}$ in the case when $\T$ is a monogenic rigidly-compactly generated tensor-triangulated category. To begin with, we recall some preliminaries concerning tensor-triangulated categories. We refer to \cite{bks} for further details.

\begin{chunk}
Until stated otherwise, we shall assume that $\T$ is a rigidly-compactly generated tensor-triangulated category. We shall let $\otimes$ denote the monoidal product, and $\iHom$ the internal hom, so there is a natural isomorphism
\[
\Hom_{\T}(X\otimes Y,Z)\simeq \Hom_{\T}(X,\iHom(Y,Z))
\]
for all $X,Y,Z\in\T$. The $\otimes$-unit in $\T$ is denoted $\1$, and for a compact object $A$ we let $DA=\iHom(A,\1)$.
Day convolution induces a closed monoidal structure on $\Mod{\T^{\c}}$ such that $\y$ is strong monoidal. We let $\otimes$ denote the monoidal product on $\Mod{\T^{\c}}$, $\iHom$ the internal hom of functors, and note that the $\otimes$-unit on $\Mod{\T^{\c}}$ is $\y \1$. We note that $\y$ is not a closed functor, that is, the canonical map
\[\y\iHom(X,Y) \to \iHom(\y X, \y Y)\] is not an isomorphism in general. However, it is an isomorphism if either $X$ is compact (see \cite[Remark 2.4]{bks}), or if $Y$ is pure injective (see \cite[Recollection 2.4]{Balmerhomsupp}).
Abusing notation, we let $D(\y A)$ denote $\iHom(\y A,\y\1)$, and note that by the above there is a natural isomorphism $\y DA\simeq D\y A$ for any compact $A\in\T^{\c}$.
\end{chunk}

\begin{chunk}
We write $\Mod{\T^\c}^\mrm{dual}$ for the full subcategory of $\Mod{\T^\c}$ consisting of the dualisable objects, i.e., those $M$ for which $DM \otimes N \to \iHom(M,N)$ is an isomorphism for all $N \in \Mod{\T^\c}$. We claim that the restricted Yoneda embedding $\y$ gives an equivalence of categories \[\y\colon \T^\c \xrightarrow{\sim} \Mod{\T^\c}^\mrm{dual}.\] 

To see that the map is well-defined, we note that an object $M$ is dualisable if and only if the canonical map $DM\otimes M \to \iHom(M,M)$ is an isomorphism, see \cite[III.\S 1]{LMSM} for instance. If $A$ is compact, there are isomorphisms $D(\y A)\otimes \y A \simeq \y (DA\otimes A)\simeq \y \iHom(A,A)\simeq \iHom(\y A,\y A)$, showing that $\y A$ is dualisable as desired.

Therefore to prove the claim it suffices to show that the functor is essentially surjective. If $M\in\Mod{\T^{\c}}$ is dualisable, then $M\otimes -$ admits a left and right adjoint, so is exact and thus $M$ is flat by \cite[\S4.1]{birdwilliamsonhomological}. Moreover, as $M$ is dualisable, $\Hom(M,-)$ preserves direct limits, so $M$ is finitely presented as well. Therefore $M$ is a finitely presented functor which is also flat, and hence is projective so is of the form $\y A$ for some $A\in\T^{\c}$.
\end{chunk}

\begin{chunk}\label{tensorhomologicalprimes}
Recall from \cite{Balmernilpotence} that the \emph{homological spectrum} $\hspec{\T^{\c}}$ is the topological space whose points are maximal Serre $\otimes$-ideals of $\mod{\T^{\c}}$, and a basis of closed sets is given by $\hsupp{A}=\{\mc{B}\in\hspec{\T^{\c}}:\y A\not\in\mc{B}\}$.
Given a Serre $\otimes$-ideal $\mc{C}\subseteq\mod{\T^{\c}}$, following \cite{bks} we let $E_{\mc{C}}\in\T$ be the unique pure injective object such that $\y E_{\mc{C}}=R_{\mc{C}}\mbb{E}(Q_{\mc{C}}\y\1)$, where $\mbb{E}$ denotes the injective envelope. Then by \cite[Theorem 3.5]{bks} one recovers the Serre $\otimes$-ideal $\mc{C}$ from $E_{\mc{C}}$ as 
\[\mc{C}=\{f\in\mod{\T^{\c}}:f\otimes \y E_{\mc{C}}=0\}.\]
\end{chunk}

\begin{defn}
Let $\mc{S}\subseteq\mod{\T^{\c}}$ be a $\Sigma$-invariant Serre subcategory.
\begin{enumerate}
    \item For a compact $A\in\T^{\c}$, set
    \[
    \mc{S}_{A} = \{f \in \mod{\T^\c} : \y A \otimes f \in \mc{S}\}.
    \]
    \item 
    Define
    \[
    \mc{S}_{\otimes}=\bigcap_{A\in\T^{\c}}\mc{S}_{A}.
    \]
\end{enumerate}
\end{defn}

Let us quickly note some properties of $\mc{S}_{\otimes}$.
\begin{lem}\label{lem:propsofstensor}
Let $\mc{S}, \mc{R}$ be $\Sigma$-invariant Serre subcategories of $\mod{\T^\c}$. We have the following:
\begin{enumerate}
    \item\label{item:ideal} $\mc{S}_{\otimes}$ is a Serre $\otimes$-ideal of $\mod{\T^\c}$;
    \item\label{item:proper} if $\mc{S}\neq \mod{\T^{\c}}$, then $\mc{S}_{\otimes}\neq \mod{\T^{\c}}$;
    \item\label{item:orderpreserving} if $\mc{S}\subseteq\mc{R}$, then $\mc{S}_{\otimes}\subseteq\mc{R}_{\otimes}$.
\end{enumerate}
\end{lem}

\begin{proof}
For (\ref{item:ideal}), to see that $\mc{S}_\otimes$ is a Serre subcategory, it suffices to see that $\mc{S}_A$ is Serre for all $A \in \T^\c$. By definition, $\mc{S}_A = (\y A \otimes -)^{-1}(\mc{S})$. Since tensoring with $\y A$ is exact, $\mc{S}_A$ is therefore also Serre. 
To show $\mc{S}_{\otimes}$ is moreover a $\otimes$-ideal, it suffices to show that if $f\in\mc{S}_{\otimes}$ and $A\in\T^{\c}$, then $\y A\otimes f\in\mc{S}_{\otimes}$, which is equivalent to having $\y B\otimes (\y A\otimes f)\in\mc{S}$ for all $B\in\T^{\c}$. Yet since $\y B\otimes (\y A\otimes f)\simeq \y(B\otimes A)\otimes f$, and $A\otimes B\in\T^{\c}$, we have that $\y B\otimes (\y A\otimes f)\in\mc{S}$ by the assumption on $f$, which finishes the proof of (\ref{item:ideal}).

To see (\ref{item:proper}) observe that $\mc{S}=\mc{S}_{\1}$ by definition. Consequently, we have $\mc{S}_{\otimes}\subseteq\mc{S}_{\1}=\mc{S}\subsetneq \mod{\T^{\c}}$.

For (\ref{item:orderpreserving}), we have $f\in\mc{S}_{A}$ if and only if $\y A\otimes f\in\mc{S}\subseteq \mc{R}$, hence $\mc{S}_{A}\subseteq \mc{R}_{A}$ for all $A\in\T^{\c}$, and taking intersections gives $\mc{S}_{\otimes}\subseteq\mc{R}_{\otimes}$.
\end{proof}

The above provides a way to construct a Serre $\otimes$-ideal from a $\Sigma$-invariant Serre subcategory. The following definition gives a way to go in the other direction.

\begin{defn}
Let $\mc{B}$ be a Serre $\otimes$-ideal of $\mod{\T^{\c}}$. Define
\[
\mc{B}_{\Sigma}=\{f\in\mod{\T^{\c}}:\Hom(\y \1, \Sigma^i f\otimes \y E_{\mc{B}})=0 \t{ for all $i \in \Z$}\}.
\]
\end{defn}

\begin{lem}\label{lem:shiftisationisshiftclosed}
Let $\mc{B}$ be a Serre $\otimes$-ideal of $\mod{\T^\c}$. The set $\mc{B}_{\Sigma}$ is a $\Sigma$-invariant Serre subcategory, and $\mc{B}\subseteq\mc{B}_{\Sigma}$. In particular, if $\mc{B} \neq \{0\}$ then $\mc{B}_{\Sigma}\neq \{0\}$.
\end{lem}

\begin{proof}
Since $\y\Sigma^{i}E_{\mc{B}}$ is a flat functor, so is $\prod_{i\in\Z}\y\Sigma^{i}E_{\mc{B}}$. In particular, the composition 
\[\Hom(\y\1, -\otimes \prod_{i \in \Z}\Sigma^i\y E_{\mc{B}})\colon\mod{\T^{\c}}\to \ab
\]
is an exact functor, hence its kernel, which is $\mc{B}_{\Sigma}$, is Serre. As the product is taken over all shifts, it is immediate that $\mc{B}_{\Sigma}$ is $\Sigma$-invariant.
Since $\mc{B}=\msf{ker}(-\otimes \y E_{\mc{B}}) \cap \mod{\T^\c}$ by \cref{tensorhomologicalprimes}, we see that if $f\in\mc{B}$, then we also have $f\otimes \Sigma^i\y E_{\mc{B}}=0$ for all $i \in \Z$, and thus $f\in\mc{B}_{\Sigma}$.
\end{proof}

\begin{lem}\label{lem:recovertensorfromshift}
Let $\mc{B}$ be a Serre $\otimes$-ideal of $\mod{\T^\c}$. Then $(\mc{B}_{\Sigma})_{\otimes}=\mc{B}$, and moreover, if $\mc{B}$ is proper then $\mc{B}_{\Sigma}$ is proper.
\end{lem}
\begin{proof}
By definition, we have $f \in (\mc{B}_\Sigma)_\otimes$ if and only if $\y A \otimes f \in \mc{B}_\Sigma$ for all $A \in \T^\c$, if and only if $\Hom(\y\1, \Sigma^i\y A\otimes f\otimes  \y E_{\mc{B}})=0$ for all $A \in \T^\c$ and $i \in \Z$. Now
\[
\Hom(\y\1, \Sigma^i\y A\otimes f\otimes \y E_{\mc{B}}) \simeq \Hom(\y DA,\Sigma^if\otimes \y E_{\mc{B}}).
\]
Therefore by compact generation, $f \in (\mc{B}_\Sigma)_\otimes$ if and only if $f \otimes \y E_\mc{B} = 0$, which in turn is equivalent to $f \in \mc{B}$, as required. To see that $\mc{B}_{\Sigma}$ is proper if $\mc{B}$ is, if one had $\mc{B}_{\Sigma}=\mod{\T^{\c}}$ then one would have $\mc{B} = (\mc{B}_{\Sigma})_{\otimes}=\mod{\T^{\c}}$ as well.
\end{proof}

We now show that the homological spectrum can be seen by $\shspec{\T^{\c}}$ when $\T$ has a single compact generator.
\begin{lem}\label{lem:downtohomological}
Let $\T$ be a rigidly-compactly generated tensor-triangulated category such that $\T^\c = \msf{thick}(G)$ for some $G \in \T^\c$. Then for any $\mc{B}\in\hspec{\T^{\c}}$ there exists an $\mc{S}\in\shspec{\T^{\c}}$ such that $\mc{B}=\mc{S}_{\otimes}$.
\end{lem}
\begin{proof}
Let $\mc{B}\in\hspec{\T^{\c}}$, then by \cref{lem:shiftisationisshiftclosed,lem:recovertensorfromshift} we have that $\mc{B}_{\Sigma}$ is a proper $\Sigma$-invariant Serre subcategory of $\mod{\T^{\c}}$. Since we assumed that $\T^\c$ has a single generator, we may apply \cref{prop:containedinmaxserre} to obtain an inclusion $\mc{B}_{\Sigma}\subseteq\mc{S}$ for some $\mc{S}\in\shspec{\T^{\c}}$. Applying \cref{lem:propsofstensor}(\ref{item:orderpreserving}), we deduce that $(\mc{B}_{\Sigma})_{\otimes}\subseteq\mc{S}_{\otimes}$, yet by \cref{lem:propsofstensor}(\ref{item:proper}) we have that $\mc{S}_{\otimes}\neq\mod{\T^{\c}}$, and by \cref{lem:recovertensorfromshift} we have that $(\mc{B}_{\Sigma})_{\otimes}=\mc{B}$. Therefore we have $\mc{B}\subseteq\mc{S}_{\otimes}\subsetneq\mod{\T^{\c}}$, so by the maximality of $\mc{B}$ it follows that $\mc{B} = \mc{S}_{\otimes}$.  
\end{proof} 

The preceding lemma is the key ingredient in the following result, which relates the shift-spectrum and the Balmer spectrum. We say that a rigidly-compactly generated tensor-triangulated category $\T$ is \emph{monogenic} if its tensor unit $\1$ is a compact generator.
\begin{thm}\label{thm:maptobalmer}
Let $\T$ be a monogenic rigidly-compactly generated tensor-triangulated category. Then every Balmer prime is a shift-prime, and the inclusion
\[
\msf{Spc}(\T^{\c})\hookrightarrow\sspec{\T^{\c}},
\]
is continuous.
\end{thm}

\begin{proof}
Let $\p$ be a Balmer prime so that $\p=\y^{-1}\mc{B}$ for some homological prime $\mc{B} \in \hspec{\T^\c}$ by \cite[Corollary 3.9]{Balmernilpotence}. By \cref{lem:downtohomological} there is a shift-homological prime $\mc{S}\in\shspec{\T^{\c}}$ such that $\mc{B}=\mc{S}_{\otimes}$. We claim that $\y^{-1}\mc{S}=\y^{-1}\mc{S}_{\otimes}$ so that $\y^{-1}\mc{S}=\mf{p}$. Since $\mc{S}_{\otimes}\subseteq\mc{S}=\mc{S}_{\1}$, it is clear that $\y^{-1}\mc{S}_{\otimes}\subseteq\y^{-1}\mc{S}$. On the other hand, let $X\in\y^{-1}\mc{S}$ and consider $\msf{L}_{X}=\{A\in\T^{\c}:\y A\otimes\y X\in\mc{S}\}$. Since $\mc{S}$ is a $\Sigma$-invariant Serre subcategory, we have that $\msf{L}_{X}$ is a thick subcategory and contains the unit by assumption on $X$. Thus $\msf{L}_{X}=\T^{\c}$, so $\y A\otimes \y X\in\mc{S}$ for all $A\in\T^{\c}$, that is, $\y X\in\mc{S}_{\otimes}$. Therefore $\p = \y^{-1}\mc{B} = \y^{-1}(\mc{S}_\otimes) = \y^{-1}\mc{S}$ is a shift-prime.
That the inclusion map $\msf{Spc}(\T^\c) \hookrightarrow \sspec{\T^\c}$ is continuous is immediate from the definitions.
\end{proof}

\begin{rem}
    Note that the continuous injection $\tau\colon \Spc(\T^\c) \to \sspec{\T^\c}$ is a homeomorphism onto its image since $\tau(\supp(A)) = \ssupp{A} \cap \tau(\Spc(\T^\c))$ for any $A \in \T^\c$. As such $\tau$ is a homeomorphism if and only if it is surjective. It is natural to ask whether this always holds. If $\T$ is not generated by its tensor unit, then it is easy to construct counterexamples. For example, $\msf{D}^{\t{b}}(\msf{coh}\,\mbb{P}^{1}_{k})$ is equivalent to $\D^{\mrm{b}}(\mod{\Lambda})$, where $\Lambda$ is the path algebra of the Kronecker quiver. We show in \cref{th:kronecker} that $\sspec{\D^{\mrm{b}}(\mod{\Lambda})}=\Z\sqcup \mbb{P}^{1}_{k}$, while $\msf{Spc}(\msf{D}^{\t{b}}(\msf{coh}\,\mbb{P}^{1}_{k}))=\mbb{P}^{1}_{k}$, see \cite[Theorem 6.3]{Balmer}.
\end{rem}

\begin{ex}\label{tthomeo}
Suppose that $\T$ is monogenic. If $\Spc(\T^\c)$ is a point, then $\tau\colon \Spc(\T^\c) \to \sspec{\T^\c}$ is a homeomorphism since $\T^\c$ has no non-zero proper thick subcategories. For example, this holds if $\T$ is a tt-field, or if $\T = \msf{D}(R)$ for a commutative Artinian local ring $R$. 
\end{ex}

\subsection{Radical thick subcategories}

As we saw in \cref{thm:maptobalmer}, the space $\sspec{\T^{\c}}$ can be seen, in some sense, as a non-monoidal analogue of the Balmer spectrum. Since the main role of the Balmer spectrum is to classify radical thick $\otimes$-ideals, we now investigate how $\sspec{\T^{\c}}$ can be used to classify certain thick subcategories of $\T^{\c}$.

\begin{defn}
If $\msf{L}\subseteq\T$ is a thick subcategory, define its \emph{radical} to be
\[
\sqrt{\msf{L}}=\bigcap_{\substack{\msf{P}\in\sspec{\T^{\c}}\\ \msf{L}\subseteq\msf{P}}} \msf{P},
\]
and say that $\msf{L}$ is radical if $\msf{L}=\sqrt{\msf{L}}$.
\end{defn}
We note that, as an empty intersection, the whole category $\T^{\c}$ is radical, as is any $\msf{P}\in\sspec{\T^{\c}}$. Furthermore, there is a clear inclusion $\msf{L}\subseteq\sqrt{\msf{L}}$ for a thick subcategory $\msf{L}$. We let $\msf{RadThick}(\T^{\c})$ denote the set of radical thick subcategories of $\T^{\c}$.

\begin{chunk}
Given a thick subcategory $\msf{L}$ of $\T^\c$, we define its support by
\[\ssupp{\msf{L}} := \bigcup_{A \in \msf{L}} \ssupp{A} = \{\msf{P} \in \sspec{\T^\c} : \msf{L} \not\subseteq \msf{P}\}.\] 
Note that for any object $A\in\T^{\c}$ we have $\ssupp{A}= \ssupp{\msf{thick}(A)}$.
As each $\ssupp{A}$ is a closed set in $\sspec{\T^{\c}}$, we see that $\ssupp{\msf{L}}$ is a specialisation closed subset of $\sspec{\T^\c}$. We let $\msf{SC}(\sspec{\T^{\c}})$ denote the lattice of specialisation closed subsets of $\sspec{\T^\c}$. This has joins given by union, and meets given by the specialisation closure of the intersection. (Note that finite meets are just given by the intersection.)
\end{chunk}

\begin{chunk}
We write $\Thick(\T^\c)$ for the lattice of thick subcategories of $\T^\c$ and recall that this has arbitrary joins and meets given by $\bigvee_{i \in I} \msf{L}_i = \msf{thick}(\bigcup_{i \in I} \msf{L}_i)$ and $\bigwedge_{i \in I} \msf{L}_i = \bigcap_{i \in I} \msf{L}_i$. Note that $\supp_\Sigma\colon \Thick(\T^\c) \to \msf{SC}(\sspec{\T^\c})$ is a poset map which preserves arbitrary joins by construction. 
\end{chunk}

\begin{ex}\label{ex:suppmeets}
While $\ssupp{-}\colon\Thick(\T^{\c})\to\msf{SC}(\sspec{\T^{\c}})$ preserves arbitrary joins, it does not preserve finite meets in general. For a concrete example, take $\T=\D(kA_{2})$. We will show in \cref{ex:psshered} that $\sspec{\D^{\mrm{b}}(kA_{2})}=\{\msf{thick}(P_{1}), 
\msf{thick}(P_{2}), \msf{thick}(S_{2})\}$ with the discrete topology. Note that $P_{1}, P_{2}$ and $S_{2}$ are the indecomposable finite dimensional modules over $kA_{2}$ up to isomorphism, where $P_{1}$ and $P_{2}$ are indecomposable projective and $S_2$ is the simple injective module. Moreover, for any $\msf{P} \in \sspec{\msf{D}^\msf{b}(kA_2)}$ we have $\ssupp{\msf{P}} = \sspec{\msf{D}^\msf{b}(kA_2)}\backslash\{\msf{P}\}$. Therefore for distinct points $\msf{P},\msf{P}'$ in $\sspec{\msf{D}^\msf{b}(kA_2)}$, we have 
\[\ssupp{\msf{P}} \cap \ssupp{\msf{P}'} = \{\msf{P}\}^{\c} \cap \{\msf{P}'\}^{\c} = \{\msf{P}\cup\msf{P}'\}^{\c}\] whereas $\ssupp{\msf{P} \cap \msf{P}'} = \ssupp{0} = \varnothing$.
\end{ex}

We can see how the support of a thick subcategory is only dependent on its radical.

\begin{lem}\label{supportofradical}
    Let $\msf{L}$ be a thick subcategory of $\T^\c$. Then $\ssupp{\msf{L}} = \ssupp{\sqrt{\msf{L}}}$.
\end{lem}
\begin{proof}
   Since $\msf{L} \subseteq \sqrt{\msf{L}}$, we have $\ssupp{\msf{L}} \subseteq \ssupp{\sqrt{\msf{L}}}$. By definition of the radical, if $\msf{L} \subseteq \msf{P}$, then $\sqrt{\msf{L}} \subseteq \msf{P}$, so $\ssupp{\msf{L}}^{\c} \subseteq \ssupp{\sqrt{\msf{L}}}^{\c}$ and we get the equality.
\end{proof}

\begin{chunk}
Given a subset $U$ of $\sspec{\T^\c}$, we define a thick subcategory $\Psi(U)$ of $\T^\c$ by 
\[\Psi(U) := \{C \in \T^\c : \ssupp{C} \subseteq U\}.\]
The following lemma collects some properties of $\Psi$.
\end{chunk}

\begin{lem}\label{lem:propsofPsi}
Let $U$ be a subset of $\sspec{\T^\c}$. 

\begin{enumerate}
    \item\label{radicaliscomposite} $\Psi(U) = \bigcap_{\msf{P} \in U^\c} \msf{P}$ and in particular, $\Psi(\ssupp{\msf{L}}) = \sqrt{\msf{L}}$ for any thick subcategory $\msf{L}$ of $\T^\c$.

    \item\label{psiisradical} $\Psi(U)$ is a radical thick subcategory of $\T^\c$.
\end{enumerate}
\end{lem}
\begin{proof}
Repeating the proof of \cite[Lemma 4.8]{Balmer} verbatim gives $\Psi(U) = \bigcap_{\msf{P} \in U^\c} \msf{P}$. The second part of (\ref{radicaliscomposite}) follows from this by definition of a radical thick subcategory.

By (\ref{radicaliscomposite}), $\Psi(U) \subseteq \sqrt{\Psi(U)}= \Psi(\ssupp{\Psi(U)})$. Conversely, we have \[\Psi(\ssupp{\Psi(U)}) = \{A \in \T^\c : \ssupp{A} \subseteq \bigcup_{\ssupp{B} \subseteq U} \ssupp{B}\} \subseteq \{A \in \T^\c : \ssupp{A} \subseteq U\} = \Psi(U),\] so $\Psi(U) = \Psi(\ssupp{\Psi(U)})$ which gives (\ref{psiisradical}).
\end{proof}

We may now assemble these preliminaries into a classification result. Following~\cite{MatsuiTakahashi}, we define the parameter set
\[\msf{Param}(\sspec{\T^\c}) = \{\ssupp{\msf{L}} : \msf{L} \in \msf{RadThick}(\T^\c)\}.\] Note that by \cref{supportofradical} one may run over all thick subcategories or only over radical thick subcategories. From this definition, the following result is expected.

\begin{prop}\label{bijectionradical}
    The assignments     \[
    \begin{tikzcd}[column sep=2cm]
        \msf{RadThick}(\T^\c) \ar[r, "\supp_\Sigma", yshift=1mm] & \mathsf{Param}(\sspec{\T^\c}) \ar[l, "\Psi", yshift=-1mm]
    \end{tikzcd}
    \] are mutually inverse bijections.
\end{prop}
\begin{proof}
    Firstly note that $\Psi$ is well-defined by \cref{lem:propsofPsi}(\ref{psiisradical}). Since taking the radical is idempotent, by \cref{lem:propsofPsi} the composition $\Psi \circ \msf{supp}_\Sigma$ is the identity when restricted to radical thick subcategories. Conversely, any element of $\mathsf{Param}(\sspec{\T^\c})$ is of the form $\ssupp{\msf{L}}$ for some radical thick subcategory $\msf{L} = \sqrt{\msf{L}}$, and we have $\ssupp{\Psi(\ssupp{\msf{L}})} = \ssupp{\sqrt{\msf{L}}} = \ssupp{\msf{L}}$ by \cref{lem:propsofPsi}(\ref{radicaliscomposite}).
\end{proof}

We also note the following corollary which provides partial information about the radical thick subcategories just from knowledge of the topology on $\sspec{\T^\c}$.

\begin{cor}\label{surjectionradical}
    The assignment \[\Psi\colon \msf{SC}(\sspec{\T^\c}) \to \msf{RadThick}(\T^\c)\] is surjective. Equivalently, $\msf{supp}_\Sigma\colon \msf{RadThick}(\T^\c) \to \msf{SC}(\sspec{\T^\c})$ is injective.
\end{cor}
\begin{proof}
    This is immediate from \cref{bijectionradical} together with the fact that any element of the parameter set is specialisation closed.
\end{proof}

We obtain the following as an immediate consequence.
\begin{cor}\label{supportdeterminesbuilding}\leavevmode
\begin{enumerate}
    \item Let $X,Y \in \T^\c$. Then $\ssupp{X} \subseteq \ssupp{Y}$ if and only if $\msf{thick}(X) \subseteq \sqrt{\msf{thick}(Y)}$.
    \item If the zero thick subcategory is radical, then $\ssupp{C} = \varnothing$ implies $C \simeq 0$. \qed
\end{enumerate}
\end{cor}

\subsection{Characterisations of radical thick subcategories}\label{charofrad} 
In this subsection we provide conditions to determine whether a given thick subcategory is radical, and then give various examples to illustrate these.

\begin{chunk}
Let $\msf{L}\subseteq\T^{\c}$ be a thick subcategory. There is a recollement
\[
\begin{tikzcd}[column sep=1.5cm]
\msf{Loc}(\msf{L}) \arrow[r, hook, shift left = 6pt] \arrow[r, leftarrow] \arrow[r, hook, shift right = 6pt] & \T  \arrow[r, shift left= 6pt, "Q"] \arrow[r, hookleftarrow, "R" description] \arrow[r, shift right= 6pt] & \T/\msf{Loc}(\msf{L})
\end{tikzcd}
\]
associated to the smashing subcategory $\msf{Loc}(\msf{L})\subseteq\T$, see \cite[Theorem 5.6.1]{KrLocTheory}.
The category $\mc{D}_{\msf{L}}:=\msf{Loc}(\msf{L})^{\perp}$ is a compactly generated triangulated definable subcategory of $\T$ by \cite[Lemma 12.4]{krcq}, and $R$ gives an equivalence of categories $\T/\msf{Loc}(\msf{L}) \xrightarrow{\sim} \mc{D}_\msf{L}$. Note that $Q$ preserves compact objects since it is a finite localisation.
\end{chunk}

\begin{chunk}
The subset $\msf{Zg}^{\Sigma}(\mc{D}_{\msf{L}})=\mc{D}_{\msf{L}}\cap\msf{Zg}^{\Sigma}(\T)$ is a closed subset of $\msf{Zg}^{\Sigma}(\T)$, which can be identified with $\msf{Zg}^{\Sigma}(\T/\msf{Loc}(\msf{L}))$ by the above equivalence. Note that $\msf{KZg}^{\Sigma}(\mc{D}_{\msf{L}})$ is the image of $\msf{Zg}^{\Sigma}(\mc{D}_{\msf{L}})$ in $\msf{KZg}^{\Sigma}(\T)$ under the Kolmogorov quotient, and define $\msf{U}_{\msf{L}}^{\t{cl}}=\msf{KZg}^{\Sigma}(\mc{D}_{\msf{L}})\cap\closed{\T}$; thus, $\msf{U}_{\msf{L}}^{\t{cl}}$ consists of equivalence classes of indecomposable pure injective objects of $\T$ that are in $\msf{L}^{\perp}$ and in $\closed{\T}$. 
\end{chunk}

\begin{lem}\label{lem:radicalalt}
Let $\msf{L}$ be a thick subcategory of $\T^\c$. Then 
\[
\sqrt{\msf{L}} = \bigcap_{[X]\in\msf{U}_{\msf{L}}^{\mrm{cl}}} {}^{\perp_\Z} X \cap \T^\c.
\]
\end{lem}
\begin{proof}
For a shift-prime $\msf{P}=\y^{-1}\mc{B}$ the containment $\msf{L}\subseteq \msf{P}$ is equivalent to $\Hom_{\T}(\msf{L},\Sigma^i X)=0$ for all $i \in \Z$ and $[X]\in\closed{\T}$ the unique point corresponding to $\mc{B}\in\shspec{\T^{\c}}$, see \cref{definingalpha}. We see that 
\[
\sqrt{\msf{L}}=\bigcap_{\substack{\msf{P}\in\sspec{\T^{\c}}\\ \msf{L}\subseteq\msf{P}}} \msf{P} = \bigcap_{[X]\in\msf{U}_{\msf{L}}^{\mrm{cl}}} {}^{\perp_\Z} X \cap \T^\c. 
\]
\end{proof}

The following proposition gives a way to test if a thick subcategory is radical. Recall that a set of objects $\{E_i : i \in I\}$ is said to cogenerate a triangulated subcategory $\mc{S}$ if $\Hom_\mc{S}(X,\Sigma^j E_i) = 0$ for all $i \in I$, $j \in \Z$ implies that $X \simeq 0$.
\begin{prop}\label{prop:RadTest}
    Let $\msf{L}$ be a thick subcategory of $\T^\c$. Then $\msf{L}$ is radical if and only if $\{QX:[X]\in\msf{U}_{\msf{L}}^{\mrm{cl}}\}$ cogenerates the essential image of $Q\colon \T^\c \to (\T/\msf{Loc}(\msf{L}))^\c$.
\end{prop}
\begin{proof}
 By \cref{lem:radicalalt}, we have that 
\[\sqrt{\msf{L}}=\{A\in\T^\c:\Hom_{\T}(A,\Sigma^{i}X)=0 \t{ for all }i\in\Z \t{ and }[X]
\in\msf{U}_{\msf{L}}^{\t{cl}}\}.
\] 
Since $\msf{Loc}(\msf{L})^{\perp}$ consists of the $L$-local objects of the associated localisation functor $L=R\circ Q\colon\T\to \T$, for any $[X] \in \msf{U}_\msf{L}^{\t{cl}}$, the canonical map $X \to LX$ is an isomorphism by definition of $\msf{U}_\msf{L}^{\t{cl}}$. As such, 
\begin{equation}\label{eq:raddescription}\sqrt{\msf{L}}=\{A\in\T^\c:\Hom_{\T/\msf{Loc}(\msf{L})}(QA,\Sigma^{i}QX)=0 \t{ for all }i\in\Z \t{ and }[X]
\in\msf{U}_{\msf{L}}^{\t{cl}}\}.
\end{equation}

Therefore if $\{QX : X \in \msf{U}_{\msf{L}}^{\mrm{cl}}\}$ cogenerates the essential image of $Q\colon \T^\c \to (\T/\msf{Loc}(\msf{L}))^\c$, we see that $A \in \sqrt{\msf{L}}$ if and only if $QA \simeq 0$. Therefore 
\[\sqrt{\msf{L}} = \msf{ker}(Q) \cap \T^\c = \msf{Loc}(\msf{L}) \cap \T^\c = \msf{L},\] so $\msf{L}$ is radical as required.

Conversely, assume that $\msf{L}$ is radical. Suppose that $A\in\T^{\c}$ is such that $\Hom_{\T/\msf{Loc}(\msf{L})}(QA,\Sigma^i QX) = 0$ for all $i \in \Z$ and $[X] \in \msf{U}_\msf{L}^\mrm{cl}$; we must show that $QA=0$. By \cref{eq:raddescription}, the assumption implies that $A\in\sqrt{\msf{L}} = \msf{L}$. Since $\msf{ker}(Q)=\msf{Loc}(\msf{L})$, we have $QA \simeq 0$ as required.
\end{proof}

The following gives an illustrative case of when the preceding proposition can be used. 

\begin{cor}\label{cor:Rpreservescompacts}
Suppose that $\T$ satisfies the following property: 
\begin{quote}
for every indecomposable $A\in\T^{\c}$ there is an $[X]\in\closed{\T}$ with $A\in\msf{Def}^{\Sigma}(X)$.
\end{quote}
Then a thick subcategory $\msf{L}\subseteq \T^{\c}$ is radical if the right adjoint $R$ in the induced localisation sequence
\[
\begin{tikzcd}
\msf{Loc}(\msf{L}) \arrow[r, hook, shift left = 3pt] & \T \arrow[r, shift left = 3pt, "Q"] \arrow[l, shift left = 3pt, swap] & \T/\msf{Loc}(\msf{L}) \arrow[l, hook', shift left = 3pt,"R"]
\end{tikzcd}
\]
preserves compact objects.
\end{cor}


\begin{proof}
We check the conditions of \cref{prop:RadTest}. Since the functor $R$ preserves compact objects by assumption, $Q$ and $R$ induce mutually inverse equivalences $(\T/\msf{Loc}(\msf{L}))^\c\simeq  \msf{L}^{\perp}\cap \T^\c$. 

Let $A$ be an indecomposable compact object in $\T/\msf{Loc}(\msf{L})$. By the second assumption $RA$ is a compact object in $\T$, and is indecomposable as $R$ is fully faithful. Consequently, by the first assumption, there is an $[X]\in\closed{\T}$ such that $RA\in\msf{Def}^{\Sigma}(X)$. Since $\msf{Def}^\Sigma(X)$ is simple, we have $\msf{Def}^\Sigma(RA) = \msf{Def}^\Sigma(X)$. Now, since $A\in \T/\msf{Loc}(\msf{L})$ we have that $RA\in \msf{L}^{\perp}$. Therefore as $\msf{L}^\perp$ is a triangulated definable subcategory, we have $\msf{Def}^\Sigma(X) = \msf{Def}^\Sigma(RA) \subseteq \msf{L}^\perp$. In particular, $X \in \msf{L}^\perp$ and hence $[X] \in \msf{U}_\msf{L}^\t{cl}$. 

To apply \cref{prop:RadTest} it suffices to check that $\Hom_{\T/\msf{Loc}(\msf{L})}(A,\Sigma^i QX) \neq 0$ for some $i$, since then $\{QX\}$ cogenerates the compact objects in $\T/\msf{Loc}(\msf{L})$. We have $\Hom_{\T/\msf{Loc}(\msf{L})}(A,\Sigma^i QX) = \Hom_\T(RA, \Sigma^i X)$, so if these were zero for all $i \in \Z$, then we would also have $\Hom_\T(RA, RA) = 0$ since $\msf{Def}^\Sigma(X) = \msf{Def}^\Sigma(RA)$ and $RA$ is a compact object. This cannot happen as $RA\neq 0$, and we conclude that $\msf{L}$ is radical by \cref{prop:RadTest}.
\end{proof}
Another source of radical thick subcategories is integral rank functions.

\begin{lem}\label{Kerrhoisradical}
Let $\rho$ be an integral rank function on $\T^{\c}$. Then the thick subcategory $\msf{ker}(\rho)$ is radical.
\end{lem}
\begin{proof}
Since $\rho$ is integral it admits a decomposition $\rho = \sum_{I} \rho_i$ into a locally finite sum of irreducibles, as described in \cref{prel:rankfunctions}. Therefore $\msf{ker}(\rho) = \cap_{I} \msf{ker}(\rho_i) = \cap_{I} \y^{-1}\msf{ker}(\tilde{\rho_i})$. Now $\msf{ker}(\tilde{\rho_i})$ is an element of $\shspec{\T^\c}$ by \cref{lem:RankGivesPrime}, and hence $\msf{ker}(\rho)$ is radical.
\end{proof}

If the map $\mc{K}\colon \Irr{\T^\c} \to \shspec{\T^\c}$ of \cref{lem:RankGivesPrime} is a homeomorphism, then $\sspec{\T^{\c}}$ has, as points, the kernels of irreducible rank functions, see \cref{chunk:shiftspecrank}. In this case we obtain the following.
\begin{cor}\label{radicalsaskernels}
Suppose that $\mc{K}\colon \Irr{\T^\c} \to \shspec{\T^\c}$ is a homeomorphism (see \cref{thm:whenisKahomeo}). Then a thick subcategory $\msf{L}$ of $\T^\c$ is radical if and only if $\msf{L} = \cap_{i \in I}\msf{ker}(\rho_{i})$ where $\rho_{i}$ is an irreducible rank function for all $i\in I$.
\end{cor}
\begin{proof}
    The reverse direction is immediate from \cref{Kerrhoisradical}. Conversely, since $\mc{K}\colon \Irr{\T^\c} \to \shspec{\T^\c}$ is a homeomorphism, each shift prime $\msf{P}$ is of the form $\msf{ker}(\rho)$ for some irreducible rank function $\rho$, so the claim is immediate from the definition of a radical thick subcategory.
\end{proof}

As integral rank functions are locally finite sums of irreducible ones, the previous corollary shows that radical thick subcategories can be thought of as kernels of ``formal sums'' of irreducible rank functions.

We end the section by giving various examples of radical thick subcategories.
\begin{ex}\label{ex:foralgebraperpsareradical}
 Let $\T = \D(\Lambda)$ where $\Lambda$ is any finite dimensional $k$-algebra and $k$ is a field. By \cite[Example 3.6]{conde2022functorial}, for any compact object $A\in\D(\Lambda)^\c$ there is a rank function, denoted 
\[
\vartheta^{A}=\sum_{i\in\Z}\,\t{dim}_{k}\,\Hom_{\D(\Lambda)^\c}(-,\Sigma^{i}A).
\]
If $\msf{A}\subseteq\D(\Lambda)^\c$ is a set of compact objects, then the thick subcategory $^{\perp_{\Z}}\msf{A}$ of $\D(\Lambda)^\c$ is radical. Indeed, for any $A\in\msf{A}$ we see that $^{\perp_{\Z}}A=\msf{ker}\,\vartheta^{A}$ is radical by \cref{Kerrhoisradical}. Since $^{\perp_{\Z}}\msf{A}=\cap_{A\in\msf{A}}\,^{\perp_{\Z}}A$, and intersections of radical thick subcategories are radical, we obtain the result. By similar reasoning, using the rank function $\vartheta_{A}$ also defined at \cite[Example 3.6]{conde2022functorial}, we also obtain that $\msf{A}^{\perp_{\Z}}$ is radical. For a discussion of similar examples in the context of dg algebras over a field or projective schemes over a field see \cite[Examples 3.7 and 3.8]{conde2022functorial}.
\end{ex}

\begin{rem}
   Note that the previous example implies that the condition (1) in \cref{cor:Rpreservescompacts} is automatically satisfied
   for $\T = \D(\Lambda)$ where $\Lambda$ is any finite dimensional $k$-algebra. Indeed, any indecomposable compact $A$ gives an irreducible rank function $\vartheta^A$. In particular, if $A$ is an indecomposable compact in $\D(\Lambda)$, then by \cref{lem:RankGivesPrime} together with \cref{homeomorphism}, the irreducible rank function $\vartheta^A$ gives a point $[X] \in \msf{KZg}_\msf{Cl}^\Sigma(\D(\Lambda))$ such that $A\in \scr{D}(\msf{ker}(\widetilde{\vartheta}^A)) = \msf{Def}^\Sigma(X)$.
\end{rem}

\begin{ex}
We will show in \cref{prop:locfinradical} that if $\T$ is a triangulated category with $\T^\c$ locally finite, then every thick subcategory of $\T^{\c}$ is radical. 
\end{ex}

\begin{ex}
In \cref{thm:tamheredeverythingradical} we will show that if $\Lambda$ is a tame hereditary algebra over an algebraically closed field, then every thick subcategory of $\D(\Lambda)^\c$ is radical.
\end{ex}

\begin{ex}
The zero thick subcategory need not be radical, as is illustrated in \cref{hypersurfaces}.
\end{ex}

In the case of rigidly-compactly generated tensor-triangulated categories, we have the following (where the term radical is used in the sense of this paper, not \cite{Balmer}).
\begin{prop}\label{prop:rigideverythingradical}
Let $\T$ be a monogenic rigidly-compactly generated tensor-triangulated category. Then every thick subcategory of $\T^{\c}$ is radical.
\end{prop}
\begin{proof}
Since $\T$ is monogenic and rigidly-compactly generated, every thick subcategory of $\T^\c$ is radical in Balmer's sense, see \cite[Remark 4.3 and Proposition 4.4]{Balmer}. In other words, for any thick subcategory $\msf{L}$ of $\T^\c$ we have $\msf{L} = \bigcap\limits_{\substack{\mathsf{L} \subseteq \p \\ \p \in \msf{Spc}(\T^\c)}}\p.$ Therefore
\[
\sqrt{\msf{L}} = \bigcap_{\substack{\msf{L} \subseteq \msf{P} \\ \msf{P} \in \sspec{\T^\c}}}\msf{P} \subseteq \bigcap_{\substack{\mathsf{L} \subseteq \p \\ \p \in \msf{Spc}(\T^\c)}}\p =\msf{L},
\]
where we use the fact from \cref{thm:maptobalmer} that, in this setting, every Balmer prime is a shift-prime.
\end{proof}

\section{Locally finite triangulated categories}\label{sec:locallyfinite}
In this section, we consider the case of pure semisimple and locally finite triangulated categories in detail. In particular, we will compute both $\shspec{\T^{\c}}$ and $\sspec{\T^{\c}}$ in the locally finite case, which will be useful for subsequent examples.  

\subsection{The spectra for finite type triangulated categories}
\begin{chunk}
Recall from \cite[\S 9]{bel} or \cite[Theorem 2.10]{krsmash} that a compactly generated triangulated category $\T$ is \emph{pure semisimple} if every object of $\T$ is pure injective. This is equivalent to $\T^{\c}$ being locally noetherian (that is, to $\Mod{\T^\c}$ being locally noetherian). In this setting, we can show that the map of \cref{lem:RankGivesPrime} is a homeomorphism.
\end{chunk}

\begin{prop}\label{prop:LocNoeth}
Let $\T$ be a pure semisimple triangulated category. Then the map 
\[\mc{K}\colon \Irr{\T^\c} \to \shspec{\T^\c}\] of \cref{lem:RankGivesPrime} is a homeomorphism.
\end{prop}

\begin{proof}
Let $\mc{B}\in\shspec{\T^{\c}}$ and consider the localisation $\Mod{\T^{\c}}/\rlim\mc{B}$. We claim that this is a locally noetherian category. To see this, since it is locally coherent, it is enough by \cite[Theorem 11.2.12]{krbook} to show that any coproduct of injective objects in $\Mod{\T^{\c}}/\rlim\mc{B}$ is injective. By~\cite[Theorem 3.7]{birdwilliamsonhomological} there is an equivalence of categories between the injective objects in $\Mod{\T^{\c}}/\rlim\mc{B}$ and $\scr{D}(\mc{B}) \cap \msf{Pinj}(\T)$, the pure injective objects in $\scr{D}(\mc{B})$.  Therefore as $\Mod{\T^{\c}}$ is locally noetherian, it follows that coproducts of pure injective objects in $\T$ are pure injective. Hence $\scr{D}(\mc{B}) \cap \msf{Pinj}(\T)$ is closed under coproducts, so $\Mod{\T^\c}/\rlim\mc{B}$ is locally noetherian as claimed.

Any finitely generated object in the quotient $\Mod{\T^\c}/\rlim\mc{B}$ contains a maximal subobject, and hence a simple quotient $S$. Since $\Mod{\T^\c}/\rlim\mc{B}$ is locally noetherian, the simple quotient $S$ is also finitely generated, and hence is an object of $\mod{\T^\c}/\mc{B}$. Consequently $\mod{\T^{\c}}/\mc{B}$ contains a simple object, which generates a $\Sigma$-invariant Serre subcategory which must, by maximality of $\mc{B}$, be the whole of $\msf{mod}(\T^{\c})/\mc{B}$. Since everything in the $\Sigma$-invariant Serre subcategory generated by $S$ is of finite length, it follows that $\msf{mod}(\T^{\c})/\mc{B}$ is a length category, and hence yields an irreducible rank function on $\T^{\c}$ by \cref{rankviaSerre}. This shows that the map $\mc{K}$ is surjective, and hence by \cref{cor:surjsuffices}, is in fact a homeomorphism.
\end{proof}

\begin{chunk}
We now restrict our attention to a particular subclass of the pure semisimple triangulated categories. The category $\T^{\c}$ is \emph{locally finite} provided that $\T^{\c}$ and $(\T^{\c})^{\op}$ are locally noetherian. In this case we say $\T$ is of \emph{finite type}, after \cite{BelARZiegler}. Note that this is equivalent to saying that the category $\mod {\T^{\c}}$ is a finite length abelian category. 
\end{chunk}

\begin{chunk}\label{KrullSchmidt}
A crucial feature of pure semisimple (and hence finite type) triangulated categories is that any object admits a decomposition into a direct sum of compact objects with local endomorphism rings, and in particular, $\T^\c$ is a Krull-Schmidt category, see~\cite[Theorem 9.3]{bel}. It follows that the indecomposable pure injective objects are just the indecomposable compact objects.
\end{chunk}

We now turn to describing $\shspec{\T^{\c}}$ and $\sspec{\T^\c}$ for $\T$ of finite type. In light of \cref{homeomorphism}, we begin by understanding the Kolmogorov quotient of $\msf{Zg}^\Sigma(\T)$. We write $\msf{ind}(\T^\c)$ for the set of indecomposable compacts in $\T$, which as a set is equal to $\msf{pinj}(\T)$.

\begin{lem}\label{prop:orbitlocallyfinite}
Let $\T$ be a finite type triangulated category. Then for  $X\in \msf{ind}(\T^\c)$ we have $\msf{Def}^{\Sigma}(X)\cap \msf{pinj}(\T)=\{\Sigma^i X:i\in\Z\}$. In particular, for $X,Y\in \msf{ind}(\T^\c)$ we have $\msf{Def}^{\Sigma}(X)=\msf{Def}^{\Sigma}(Y)$  if and only if $X\simeq \Sigma^{i}Y$ for some $i\in \Z$.
\end{lem}
\begin{proof}
By \cite[Theorem 10.13]{BelARZiegler}, any object of $\T$ is endofinite, and thus, by \cref{prel:endofinite}, it follows that $\msf{Def}(X^{\Sigma})=\msf{Add}(X^{\Sigma})$. Suppose that $Y\in\msf{Add}(X^{\Sigma})\cap\msf{pinj}(\T)$, so that $Y$ is a summand of $\oplus_{I}X^{\Sigma}$ for some set $I$. By the Krull-Schmidt property \cref{KrullSchmidt}, we have that $Y\simeq \Sigma^{i}X$ for some $i\in\Z$. In particular, we have that $\msf{Def}^{\Sigma}(X)\cap\msf{pinj}(\T)=\{\Sigma^i X:i\in\Z\}$. 
\end{proof}

\begin{chunk}\label{chunk:indtopology}
We equip the set $\msf{ind}(\T^\c)$ with a topology with a basis of open sets given by \[(C)_{\Sigma}=\{B\in\msf{ind}(\T^{\c}):\Hom_{\T}(C,\Sigma^{i}B)=0 \t{ for all $i\in\Z$}\}\] as $C$ ranges over $\T^\c$. In light of \cref{KrullSchmidt}, one may think of this as the $\msf{GZ}$-topology on $\msf{Zg}(\T)$. We consider the space $\msf{ind}(\T^\c)/\Sigma$ with the quotient topology.
\end{chunk}

\begin{prop}\label{prop:shspeclocallyfinite}
Let $\T$ be a finite type triangulated category. There is a homeomorphism
\[
\Psi\colon\shspec{\T^\c} \xrightarrow{\sim} \msf{ind}(\T^\c)/\Sigma,
\]
with $\Psi(\mc{B}) = [X]$ where $X$ is any indecomposable compact in $\scr{D}(\mc{B})$.
\end{prop}

\begin{proof}
By \cref{homeomorphism} there is a homeomorphism $\shspec{\T^{\c}}\xrightarrow{\Phi} \closed{\T}^{\msf{GZ}}$, sending $\mc{B}$ to $[X]$ for an indecomposable pure injective $X\in \scr{D}(\mc{B})$. 
    
By \cref{prop:orbitlocallyfinite}, the equivalence classes of the Kolmogorov quotient relation on $\msf{Zg}^{\Sigma}(\T)$ coincide with the shift orbits, and consequently the sets $\msf{KZg}^{\Sigma}(\T)$ and $\msf{ind}(\T^{\c})/\Sigma$ are equal. Since $\msf{Def}^{\Sigma}(X)\cap\msf{pinj}(\T)=\{\Sigma^{i}X:i\in \Z\}$, we see that any point in $\msf{KZg}^{\Sigma}(\T)$ is closed and hence $\closed{\T}=\msf{ind}(\T^{\c})/\Sigma$ as sets.
    
The sets
\[
(C)_{\Sigma}=\{[B]\in\msf{ind}(\T^\c)/\Sigma:\Hom_{\T}(C,\Sigma^{i}B)=0 \t{ for all $i\in\Z$}\}
\] 
as $C$ ranges over $\T^\c$, form a basis of open sets for the quotient topology on $\msf{ind}(\T^\c)/\Sigma$. This topology therefore coincides with the $\msf{GZ}$-topology on $\closed{\T}$, hence the spaces $\closed{\T}^{\msf{GZ}}$ and $\msf{ind}(\T^{\c})/\Sigma$ are homeomorphic, which gives the result.
\end{proof}

\begin{rem}
Combining the last corollary with \cref{prop:LocNoeth}, we can deduce the space of irreducible rank functions in the finite type setting. Namely, $\Irr{\T^\c}$ is homeomorphic to $\msf{ind}(\T^\c)/\Sigma$. 
\end{rem}

We now describe $\sspec{\T^{\c}}$. For two indecomposable compact objects $A,B\in\msf{ind}(\T^{\c})$, we write $A\sim_{\msf{t}}B$ if and only if $\msf{thick}(A)=\msf{thick}(B)$. The space $\msf{ind}(\T^{\c})/{\sim_{\msf{t}}}$ is the quotient space of $\msf{ind}(\T^{\c})$ with respect to this relation. As the following result shows, this is also the Kolmogorov quotient of $\msf{ind}(\T^{\c})/\Sigma$.

\begin{prop}\label{locfiniteindmodsim}
Let $\T$ be a finite type triangulated category. The Kolmogorov quotient relation on $\msf{ind}(\T^{\c})/\Sigma$ is given by $[A]\sim[B]\iff A\sim_{\msf{t}}B$. Consequently, there is a homeomorphism
\[
\sspec{\T^{\c}}\to \msf{ind}(\T^{\c})/{\sim_{\msf{t}}}
\]
given by $\msf{P}=\y^{-1}\mc{B}\mapsto [A]$ where $A$ is any object in $\scr{D}(\mc{B})\cap\msf{ind}(\T^{\c})$, 
with inverse given by $[A]\mapsto \,^{\perp_{\Z}}A$. 
\end{prop}
\begin{proof}
If $[A],[B]\in\msf{ind}(\T^{\c})/\Sigma$, we have that $\overline{\{[A]\}}=\overline{\{[B]\}}$ if and only if $[A]\in(C)_{\Sigma}\iff [B]\in (C)_{\Sigma}$ for all $C\in\T^{\c}$, that is $[A]$ and $[B]$ are contained within the same basic open sets. Unravelling, we see that this is equivalent to $\Hom_{\T}(C,\Sigma^{i}A)=0\iff \Hom_{\T}(C,\Sigma^{i}B)=0$ for all $C\in\T^{\c}$ and $i\in\Z$, or equivalently $^{\perp_{\Z}}A=\,^{\perp_{\Z}}B$. 

Now, $^{\perp_{\Z}}A=\,^{\perp}\msf{thick}(A)$, and as $\T^{\c}$ is locally finite one has that $\msf{thick}(A)=(^{\perp}\msf{thick}(A))^{\perp}$ by \cite[Proposition 4.4(1)]{KrauseLocallyFinite}. In particular, we have
\[
\overline{\{[A]\}}=\overline{\{[B]\}} \iff \,^{\perp}\msf{thick}(A)=\,^{\perp}\msf{thick}(B)\iff \msf{thick}(A)=\msf{thick}(B),
\]
which proves the first claim. The induced homeomorphism now exists immediately by taking the Kolmogorov quotient of the homeomorphism in \cref{prop:shspeclocallyfinite} and using \cref{prop:KQ}, while the description of the map and its inverse also follow from  \cref{prop:shspeclocallyfinite} and the fundamental correspondence. 
\end{proof}

\begin{rem}
\cref{lem:thicksamepreimage} showed that for $[X],[Y] \in \msf{KZg}_\msf{Cl}^\Sigma(\T)$, if $\msf{thick}(X) = \msf{thick}(Y)$, then $\alpha([X]) = \alpha([Y])$, that is, the points become identified under the passage to the shift-spectrum $\sspec{\T^\c}$. If $\T$ is of finite type, then the converse is also true by \cref{prop:shspeclocallyfinite} and \cref{locfiniteindmodsim}. 
\end{rem}

We now show that for finite type triangulated categories, every thick subcategory is radical.

\begin{prop}\label{prop:locfinradical}
Let $\T$ be a finite type triangulated category. Every thick subcategory of $\T^{\c}$ is radical.
\end{prop}
\begin{proof}
For this we apply \cref{cor:Rpreservescompacts}. Let $\msf{L}\subseteq\T^{\c}$ be a thick subcategory. Let $R$ denote the right adjoint of the localisation $Q\colon \T\to \T/\msf{Loc}(\msf{L})$. The localisation $Q$ restricts to a Verdier localisation $Q\colon \T^\c \to \T^\c/\msf{L}$. The inclusion functor $\msf{L}\to \T^{\c}$ admits a right adjoint by \cite[Theorem 2.5]{KrauseLocallyFinite}, which implies that the Verdier localisation $Q\colon\T^{\c}\rightarrow \T^{\c}/ \msf{L}$ admits a right adjoint $\bar{R}$.

We first show that $R$ preserves compact objects. For any $B \in \T^\c$ there is a natural map $\theta_B\colon\bar{R}QB \to RQB$ given by the adjunct to the identity map $Q\bar{R}QB = QB \to QB$. For any $A\in\T^{\c}$ and $B\in \T^{\c}$, the map $\theta_B$ induces an isomorphism
\[
\Hom_{\T}(A,RQB)\simeq\Hom_{\T/\msf{Loc}(\msf{L})}(QA,QB)\simeq \Hom_{\T^{\c}/\msf{L}}(QA,QB)\simeq\Hom_{\T}(A,\bar{R}QB),
\]
since $QA$ and $QB$ are compact objects. In particular, by compact generation we deduce that $R(QB)=\bar{R}(QB)$, hence $R(QB)\in\T^{\c}$. Since every compact object in $\T/\msf{Loc}(\msf{L})$ is a summand of $QB$ for some $B\in\T^{\c}$ by \cite[Theorem 5.6.1]{KrLocTheory}, and $\T^{\c}$ is closed under summands, we deduce that $R$ preserves compact objects.

The condition that for every indecomposable $A\in\T^{\c}$ there is an $[X]\in\closed{\T}$ such that $A\in\msf{Def}^{\Sigma}(X)$ then follows: we may take $X = A$ which indeed satisfies $[A] \in \closed{\T}$ by \cref{prop:shspeclocallyfinite} since $\T$ is of finite type. 
\end{proof}

We now provide conditions under which $\sspec{\T^{\c}}$ is a discrete space, which will encompass all of our main examples of finite type triangulated categories.
\begin{lem}\label{lem:locallyfiniteopen}
Let $\T$ be a finite type triangulated category such that $\msf{ind}(\T^{\c})/{\sim_{\msf{t}}}$ is a finite set. Then for any thick subcategory $\msf{L}\subseteq \T^{\c}$, the set $(\msf{L}\cap \msf{ind}(\T^{\c}))/{\sim_{\msf{t}}}$ is open in $\msf{ind}(\T^{\c})/{\sim_{\msf{t}}}$. 
\end{lem}

\begin{proof}
By inspection, the quotient map $\msf{ind}(\T^{\c})\to\msf{ind}(\T^{\c})/{\sim_{\msf{t}}}$ is an open map, and thus the sets $(C)_{\Sigma}$ are open in $\msf{ind}(\T^{\c})/{\sim_{\msf{t}}}$. 
Since $\T^{\c}$ is locally finite, $ (^{\perp}\msf{L})^{\perp}=\msf{L}$ by~\cite[Proposition 4.4(1)]{KrauseLocallyFinite}. So,
\begin{align*}
(\msf{L}\cap \msf{ind}(\T^{\c}))/{\sim_{\msf{t}}}&=((^{\perp}\msf{L})^{\perp}\cap \msf{ind}(\T^{\c}))/{\sim_{\msf{t}}} \\
&=
\{[B]\in \msf{ind}(\T^{\c})/{\sim_{\msf{t}}} : \Hom({}^{\perp}\msf{L},B)=0\} \\
&= \bigcap_{C\in {}^{\perp}\msf{L}} (C)_{\Sigma}\\
&= \bigcap_{[C]\in ({}^{\perp}\msf{L}\cap \msf{ind}(\T^{\c}))/{\sim_{\msf{t}}})} (C)_{\Sigma}.
\end{align*}
The final equality uses the Krull-Schmidt property \cref{KrullSchmidt}, and that $C\sim_{\msf{t}}C'$ implies $(C)_{\Sigma}=(C')_{\Sigma}$. Therefore, $(\msf{L} \cap \msf{ind}(\T^\c))/{\sim_{\msf{t}}}$ can be realised as a finite intersection of open subsets by the assumption that $\msf{ind}(\T^{\c})/{\sim_{\msf{t}}}$ is finite, and hence $(\msf{L}\cap\msf{ind}(\T^{\c}))/{\sim_{\msf{t}}}$ is open.
\end{proof}

\begin{cor}\label{locallyfinitediscrete}
Let $\T$ be a finite type triangulated category such that $\msf{ind}(\T^\c)/{\sim_{\msf{t}}}$ is a finite set. If for every $A\in\msf{ind}(\T^{\c})$ we have $(\msf{thick}(A)\cap\msf{ind}(\T^{\c}))/{\sim_{\msf{t}}}=\{[A]\}$, then $\sspec{\T^{\c}}$ is discrete.
\end{cor}
\begin{proof}
Since $(\msf{thick}(A)\cap\msf{ind}(\T^{\c}))/{\sim_{\msf{t}}}=\{[A]\}$ for each $A \in \msf{ind}(\T^\c)$, \cref{lem:locallyfiniteopen} together with \cref{locfiniteindmodsim} shows that each point in $\sspec{\T^{\c}}$ is open. Since the space is finite, it then follows that it is discrete.
\end{proof}

\begin{rem}
We note that $\msf{ind}(\T^\c)/{\sim_{\msf{t}}}$ is a finite set if $\T^\c$ is either generated by a single object or is connected, see \cite[Proposition 4.5 and Theorem 4.6]{KrauseLocallyFinite}.
\end{rem}

\begin{ex}\label{ex:psshered}
If $A$ is a hereditary $k$-algebra of finite representation type, then $\sspec{\D(A)^\c}$ is discrete. Note that $\D(A)$ is of finite type, see Example (3) after Proposition 2.3 in \cite{KrauseLocallyFinite}. Any finite dimensional indecomposable $A$-module $M$ is exceptional by \cite[Application 1]{RingelBricks} or \cite[Lemma 4.10]{HuberyKrause}, and so $\Hom_{A}(M,M)$ is a division ring and has vanishing higher self-extensions. Consequently $\msf{thick}(M)=\{\Sigma^{i}M:i\in\Z\}$ and thus $\sspec{\D(A)^\c}$ is discrete by the preceding result.
\end{ex}

\begin{rem}
There are cases when $\sspec{\T^{\c}}$ is not discrete for $\T^{\c}$ locally finite. For example, consider $\T = \prod_{i=1}^\infty \msf{D}(k)$ for a field $k$. One sees that $\T^\c$ is the full subcategory of $\prod_{i=1}^\infty \msf{D}(k)^\c$ spanned by the tuples which have only finitely many non-zero entries. The indecomposable compacts are the objects $k[i]$ (i.e., the tuple which is zero in every entry except for $i$ where it is $k$.) By the criterion~\cite[Proposition 2.3]{KrauseLocallyFinite}, one checks that $\T^\c$ is locally finite. The sets \[(k[i])_\Sigma^\c = \{[k[j]] : \Hom(k[i], k[j]) \neq 0\} = \{[k[i]]\}\] and their finite unions form a basis of closed sets for $\sspec{\T^\c}$ (under the homeomorphism to $\msf{ind}(\T^\c)/{\sim_{\msf{t}}}$ proved in \cref{locfiniteindmodsim}). As such the closed sets are the finite unions of the singletons, and the space is not discrete. 
\end{rem}

\subsection{Tensor-triangular fields}\label{ttfields}
In this subsection we use the above results to further investigate the relationship between the spectra defined in this paper, and their tensor analogues.

\begin{chunk}
Suppose that $\scr{F}$ is a tt-field in the sense of~\cite[Definition 1.1]{bks}, that is, $\scr{F}$ is a rigidly-compactly generated tensor-triangulated category which is pure semisimple, and for which $X \otimes -\colon \scr{F} \to \scr{F}$ is a faithful functor, for all non-zero $X \in \scr{F}$. Any tt-field $\scr{F}$ is of finite type since rigidity ensures that $\scr{F}^\c \simeq (\scr{F}^\c)^\op$, see~\cite[Remark 5.8]{bks} for more details. As such the above results apply. By~\cite[Proposition 5.15 and Theorem 5.17(a)]{bks}, the tensor versions, namely the homological spectrum $\hspec{\scr{F}^\c}$ and the Balmer spectrum $\Spc(\scr{F}^\c)$ are both just a point. This need not be the case for the non-tensor versions of the spaces as we now explain. First we give an example where the shift-homological spectrum is not a point.
\end{chunk}

\begin{ex}\label{ex:ttfcyclic}
Let $k$ be a field of characteristic $p$, and consider the stable module category of the cyclic group of order $p$, $\mathsf{StMod}(kC_p)$, which is a tt-field by \cite[Proposition 5.1]{bks}. Making the identification $kC_p \cong k[t]/t^p$, and writing $\langle i\rangle$ for the object $k[t]/t^i$, the collection of indecomposable compacts is $\{\langle i\rangle : 1 \leq i \leq p-1\}$. 

The shift-homological spectrum of $\msf{StMod}(kC_p)^\c = \mathsf{stmod}(kC_p)$ has $\lceil (p-1)/2 \rceil$ points by~\cref{prop:shspeclocallyfinite}, with points corresponding to $\langle i\rangle$ for $1 \leq i \leq \lceil (p-1)/2 \rceil$ as we have $\Sigma\langle i \rangle = \langle p-i \rangle$. For any $1 \leq i \leq \lceil p-1/2 \rceil$ we have ${}^\perp \langle i \rangle = \{0\}$, so for any non-zero $C \in \T^\c$ we have $(C)_\Sigma = \varnothing$. These are the basic opens in the topology, so we see that $\shspec{\msf{stmod}(kC_p)}$ is the $\lceil (p-1)/2 \rceil$ point space with the indiscrete topology. So although 0 is the only maximal Serre $\otimes$-ideal, there are a few maximal $\Sigma$-invariant Serre subcategories (and in particular, $0$ is not maximal as a $\Sigma$-invariant Serre subcategory). 

An alternative point of view on this is provided by the results of \cref{subsec:maptobalmer}. In the notation of that section, the $\otimes$-homological prime $0$ satisfies $0 = \mathsf{Ker}(\Hom^*(\y\1,-))_\otimes$. 

On the other hand, the shift-spectrum $\sspec{\msf{stmod}(kC_p)}$ is just a point by \cref{locfiniteindmodsim}, since all indecomposable compacts in $\msf{StMod}(kC_p)$ generate the same thick subcategory (namely, the whole category of compacts). 
\end{ex}

We now give an example where the shift-spectrum (and hence the shift-homological spectrum) is not just a point.
\begin{ex}
Let $\scr{G}$ be a tt-field and consider the product of two copies of it: $\scr{F} = \scr{G} \times \scr{G}$ with tensor product given by
\[
(x_0, x_1) \otimes (y_0, y_1) = ((x_0 \otimes y_0) \oplus (x_1 \otimes y_1), (x_0 \otimes y_1) \oplus (x_1 \otimes y_0))
\] 
and unit $\1_\scr{F} = (\1_\scr{G},0).$ Then $\scr{F}$ is a tt-field by~\cite[Example 5.19]{bks}, but from \cref{locfiniteindmodsim}, we see that its shift-spectrum $\sspec{\scr{F}^\c}$ cannot be a point. Indeed, since the indecomposable compacts must take the form $(x,0)$ or $(0,x)$, no single indecomposable compact can generate the whole category under thick closure.
\end{ex}

\section{Hereditary rings and tame hereditary algebras}\label{sec:hered}
In this section we consider the case of hereditary categories $\mc{A}$, and give general results concerning the structure of the shift-spectrum and the shift-homological spectrum for the compact objects in their derived categories $\D(\mc{A})$.  We pay particular attention to the case of tame hereditary algebras $\Lambda$, proving that the isolation condition holds for $\D(\Lambda)$ and that every thick subcategory of $\D(\Lambda)^\c$ is radical.

\subsection{Hereditary categories}
\begin{chunk}
Let $\mc{A}$ be a locally coherent hereditary category, that is $\t{Ext}_{\mc{A}}^{i}(-,-)=0$ for $i\geq 2$. By \cite[Proposition 7.4]{stovicekpurity} the derived category $\D(\mc{A})$ is compactly generated by $\msf{fp}(\mc{A})$, where as usual we view objects of $\mc{A}$ as complexes concentrated in degree 0. If $X\in\D(\mc{A})$, then
\[
X\simeq\bigoplus_{i\in\Z}\Sigma^iH_{i}(X)\simeq \prod_{i}\Sigma^iH_{i}(X),
\]
where the first isomorphism can be found at \cite[Proposition 4.4.15]{krbook}, and the second equality holds since the canonical map $\bigoplus_{i\in\Z}\Sigma^iH_{i}(X)\to \prod_{i}\Sigma^i H_{i}(X)$ in $\msf{Ch}(\mc{A})$ is a quasi-isomorphism.
\end{chunk}

\begin{chunk}\label{Zieglerhereditary}
The functor $H_{i}(-)\colon\D(\mc{A})\to\mc{A}$ and the embedding $\mc{A}\to\D(\mc{A})$ preserve indecomposable pure injectives and are definable. We therefore see that $X\in\D(\mc{A})$ is pure injective if and only if $H_{i}(X)$ is pure injective for all $i$, and thus $X$ is indecomposable pure injective if and only if $X\simeq \Sigma^{i}E$ for some $E\in\msf{pinj}(\mc{A})$. By the definability of the functors, see \cite[Theorem 6.1]{definablefunctors}, there is a homeomorphism 
\begin{equation*}\label{hered:zieglerdecomp}
    \msf{Zg}(\D(\mc{A}))\simeq \bigsqcup_{\Z}\msf{Zg}(\mc{A}).
\end{equation*}
 An alternative proof for the module categories of hereditary rings can be found at \cite[Theorem 8.1]{prestgarkusha}.
\end{chunk}

We will use \cref{hered:zieglerdecomp} to show that $\shspec{\D(\mc{A})^{\c}}$ is homeomorphic to the closed points in $\msf{KZg}(\mc{A})$ equipped with a suitable Gabriel-Zariski topology. In other words, the Ziegler spectrum of $\mc{A}$ completely determines the shift-homological spectrum and the shift-spectrum of $\D(\mc{A})^{\c}$ when considered with the GZ-topology. We now make the GZ-topology on $\msf{KZg}_{\msf{Cl}}(\mc{A})$ precise.

\begin{chunk}
Define the GZ-topology on $\msf{KZg}_{\msf{Cl}}(\mc{A})$ to have a basis of open sets given by those of the form 
\[
[M]_{\mc{A}}=\{[X]\in\msf{KZg}_{\msf{Cl}}(\mc{A}):\Hom_{\mc{A}}(M,X)= 0=\t{Ext}_{\mc{A}}^{1}(M,X)\}
\]
as $M$ runs through $\msf{fp}(\mc{A})$. We write $\msf{KZg}_{\msf{Cl}}(\mc{A})^\msf{GZ}$ for this topological space. That these sets form the basis for a topology, follows from the observation that $[0]_{\mc{A}}=\msf{KZg}_{\msf{Cl}}(\mc{A})$ and $[M\oplus N]_{\mc{A}}=[M]_{\mc{A}}\cap[N]_{\mc{A}}$. 
\end{chunk}

\begin{rem}
Note that $[M]_{\mc{A}}$ is well-defined since for a finitely presented object $M$ and indecomposable pure injectives $X,X'$ such that $[X]=[X']\in \msf{KZg}(\mc{A})$ we have $\Hom_{\mc{A}}(M,X)= 0$ if and only if $\Hom_{\mc{A}}(M,X')= 0$. 
The same holds for $\t{Ext}_{\mc{A}}^{1}(M,-)$, since $\t{Ext}_{\mc{A}}^{1}(M,-)$ preserves products and direct limits. To see the statement about direct limits, one can identify $\t{Ext}_{\mc{A}}^{1}(M,-)$ with $\Hom_{\D(\mc{A})}(M,\Sigma(-))$ and the direct limit of objects in $\mc{A}$ with the homotopy colimit of the same objects viewed as complexes (which is possible, since colimits are exact in chain complexes on $\mc{A}$). The claim then follows from the fact that $M$ is compact, hence sends a filtered homotopy colimit to a direct limit.  Therefore the kernel of $\t{Ext}_{\mc{A}}^{1}(M,-)$ is a definable subcategory and we get the statement.
\end{rem}

In order to provide a description of $\shspec{\D(\mc{A})^\c}$ in terms of the Ziegler spectrum of $\mc{A}$ we require the following lemma. To avoid ambiguity we will use the notation $M[i]$ for the stalk complex concentrated in degree $i$.

\begin{lem}\label{lem:hereditaryhomeo}
The map $f\colon\msf{Zg}(\mc{A})\to\msf{KZg}^{\Sigma}(\D(\mc{A}))$ given by $M\mapsto [M[0]]$ induces a homeomorphism
\[
\varphi\colon\msf{KZg}(\mc{A})\to\msf{KZg}^{\Sigma}(\D(\mc{A})).
\]
\end{lem}

\begin{proof}
Let us first observe that the map $f$ is surjective: if $X\in\msf{Zg}(\D(\mc{A}))$, then $X\simeq E[i]$ for some $E\in\msf{Zg}(\mc{A})$ and $i\in\Z$ by \cref{Zieglerhereditary}. Since $\msf{Def}^{\Sigma}(E[i])=\msf{Def}^{\Sigma}(E[j])$ for all $j\in\Z$, it follows that $E[i]\sim E[0]$, and thus $X$ and $E[0]$ are in the same equivalence class in $\msf{KZg}^\Sigma(\D(\mc{A}))$, so $f(E) = [X]$. 

Since $(-)[0]$ is a definable functor that preserves indecomposability, it induces a continuous map $\msf{Zg}(\mc{A})\to\msf{Zg}(\D(\mc{A}))$. By definition, we may write $f$ as the composite
\[\msf{Zg}(\mc{A}) \xrightarrow{(-)[0]} \msf{Zg}(\D(\mc{A})) \xrightarrow{\mrm{id}} \msf{Zg}^\Sigma(\D(\mc{A})) \xrightarrow{\pi} \msf{KZg}^\Sigma(\D(\mc{A})).\]
Each of these maps is continuous, and hence $f$ is continuous.

We now show that $f$ is closed.
Observe that for a closed set $U\subseteq \msf{Zg}(\mc{A})$ its image $U'=\{[M[0]]:M\in U\}$ in $\msf{KZg}^{\Sigma}(\D(\mc{A}))$ is closed if and only if $\pi^{-1}U'$ is closed. We first claim that
\[
\pi^{-1}U'=\bigcap_{i\in\Z}H_{i}^{-1}(\msf{Def}(U))\cap \msf{pinj}(\D(\mc{A})).
\]
Each $H_{i}^{-1}(\msf{Def}(U))$ is a definable subcategory since they are preimages of definable subcategories under coherent functors. Therefore $\bigcap_{i\in\Z}H_{i}^{-1}(\msf{Def}(U))$ is $\Sigma$-invariant and definable. Now if $X \in \pi^{-1}U'$, then there exists $M \in U$ such that $\msf{Def}^{\Sigma}(X)=\msf{Def}^{\Sigma}(M[0])$. Since $M[0] \in \bigcap_{i \in \Z} H_i^{-1}(\msf{Def}(U))$, we also have $X \in \bigcap_{i \in \Z} H_i^{-1}(\msf{Def}(U))$ as required. For the other inclusion,  any indecomposable pure injective in $\D(\mc{A})$ is concentrated in one degree. Therefore, any object in $\bigcap_{i\in\Z}H_{i}^{-1}(\msf{Def}(U))\cap \msf{pinj}(\D(\mc{A}))$ is of the form $M[i]$ for some $M\in U$. 
This shows that $\pi^{-1}U'$ is a closed set in $\msf{Zg}^\Sigma(\D(\mc{A}))$, we deduce that $U'$ is closed as required. 

Consequently, $f$ induces a closed and continuous surjection
\[
\varphi:=\msf{K}(f)\colon\msf{KZg}(\mc{A})\to\msf{KZg}^{\Sigma}(\D(\mc{A})).
\]
So it remains to show that $\phi$ is an injection. So we take $[X], [Y] \in \msf{KZg}(\mc{A})$ such that $\varphi([X])=\varphi([Y])$. By definition of $\phi$ this means that $\msf{Def}^{\Sigma}(X[0])=\msf{Def}^{\Sigma}(Y[0])$. Now, since $X\simeq H_{0}(X[0])$, and using the fact that $H_{0}\colon\msf{D}(\mc{A})\to\mc{A}$ is a coherent functor, we have that \[H_{0}(\msf{Def}^{\Sigma}(X[0])) \subseteq \msf{Def}(H_{0}(X[0]))=\msf{Def}(X).
\]
We deduce that $Y\in\msf{Def}(X)$ in $\mc{A}$. Symmetrically we deduce that $X\in\msf{Def}(Y)$, and thus $[X]=[Y]$ in $\msf{KZg}(\mc{A})$. Hence $\varphi$ is injective, and therefore a homeomorphism.
\end{proof}

Therefore $\closed{\D(\mc{A})}$ is in bijection with the closed points in $\msf{KZg}(\mc{A})$, which we denote by $\msf{KZg}_{\msf{Cl}}(\mc{A})$. 

\begin{thm}\label{prop:hereditarycategoryspectra}
Let $\mc{A}$ be a locally coherent hereditary category. Then there is a homeomorphism
\[
\shspec{\D(\mc{A})^{\c}}\xrightarrow{\sim}\msf{KZg}_{\msf{Cl}}(\mc{A})^{\msf{GZ}}.
\]
\end{thm}
\begin{proof}
By combining \cref{homeomorphism} with \cref{lem:hereditaryhomeo}, we have 
\[\shspec{\D(\mc{A})^{\c}} \xrightarrow[\sim]{\Phi} \msf{KZg}_{\msf{Cl}}(\D(\mc{A}))^{\msf{GZ}} \xleftarrow{\varphi} \msf{KZg}_{\msf{Cl}}(\mc{A})^{\msf{GZ}}\]
in which the left arrow is a homeomorphism, and the right arrow is bijection. Therefore, it suffices to prove that $\varphi$ is open and continuous.

As recalled above, by \cite[Proposition 7.4]{stovicekpurity}, the objects of $\msf{fp}(\mc{A})$ form a set of compact generators for $\D(\mc{A})$, and thus any complex of the form $\oplus_{i=m}^{n}\Sigma^{i}A$ is compact, whenever $A\in\msf{fp}(\mc{A})$. In particular, for $A\in\msf{fp}(\mc{A})$ we see that  $\phi([A]_{\mc{A}})=\{[X[0]]:\Hom_{\mc{A}}(A,X)=0=\t{Ext}_{\mc{A}}^{1}(A,X)\}=[A[0]]_{\Sigma}$, hence $\phi$ is open in the GZ-topology.

By the locally coherent assumption, $\msf{fp}(\mc{A})$ is an abelian category. Since $\msf{fp}(\mc{A})$ is a set of compact generators for $\D(\mc{A})$, by a thick subcategory argument one checks that for any $P\in\D(\mc{A})^{\c}$, $H_{i}(P)$ is finitely presented, and that $P$ is bounded. Therefore $P$ is concentrated between degrees $m$ and $n$ and $P \simeq \bigoplus_{i=m}^n \Sigma^{i}H_i(P)$ as $\mc{A}$ is hereditary. Therefore, recalling the notation of \cref{defn:GZtoponZiegler}, we have 
\begin{align*}
[P]_{\Sigma}&=\bigcap_{i=m}^{n}[\Sigma^{i}H_{i}(P)]_{\Sigma}\\ &=\bigcap_{i=m}^{n}[H_{i}(P)]_{\Sigma}\\
&=\bigcap_{i=m}^{n}\{[X]\in\closed{\D(\mc{A})}:\Hom_{\D(\mc{A})}(H_{i}(P),\Sigma^{j}X)=0 \text{ for all }j\in\Z\}
    \end{align*}

and so
\[
\phi^{-1}([P]_{\Sigma})=\bigcap_{i=m}^{n}\{[X]\in\msf{KZg}_{\msf{Cl}}(\mc{A}):\Hom_{\mc{A}}(H_{i}(P),X)=0=\t{Ext}_{\mc{A}}^{1}(H_{i}(P),X)\} = \bigcap_{i=m}^{n}[H_{i}(P)]_{\mc{A}}.
\]
Being a finite intersection, we see that $\phi^{-1}([P]_\Sigma)$ is open, and therefore $\phi$ is continuous in the GZ-topology. We have shown that $\phi$ is an open, continuous bijection, and hence it is a homeomorphism.
\end{proof}

Let us now give an example of how \cref{prop:hereditarycategoryspectra} can be used in a concrete instance, by computing $\sspec{\D(R)^\c}$ for certain Dedekind domains, including all commutative ones. 

\begin{prop}\label{dedekind}
Let $R$ be a PI Dedekind domain which is not a division ring. Then there is a homeomorphism $\sspec{\D(R)^\c}\xrightarrow{\sim} \msf{Spec}(Z(R))$, where $Z(R)$ denotes the centre of $R$.
\end{prop}
\begin{proof}
The Ziegler spectra $\msf{Zg}(R)$ of such rings are known, see \cite[\S 5.2.1]{psl}. The points are as follows:
\begin{enumerate}
    \item the indecomposable torsion modules $R/\p^n$ where $\p$ is a prime (and hence maximal) two-sided ideal and $1 \leq n < \infty$;
    \item the $\p$-adic completions $R_\p^\wedge$ where $\p$ is a prime ideal; 
    \item the Pr\"ufer modules $R/\p^\infty = E(R/\p)$ where $\p$ is a prime ideal;
    \item the quotient division ring $Q(R)$ of $R$.
\end{enumerate}

From the discussion of the topology given in \cite[Theorem 5.2.3 and p. 226]{psl}, the closed points of $\msf{Zg}(R)$ are those of the form $R/\p^n$ (for $n$ finite) and $Q(R)$. Moreover, if $\scr{X}$ is a subset of $\msf{Zg}(R)$ which does not contain any points of the form $R/\p^n$ with $n$ finite, then $\scr{X}$ is closed if and only if it contains $Q(R)$. From these facts we deduce that
\[
\overline{\{R/\p^n\}} = \{R/\p^n\} \qquad \overline{\{R/\p^\infty\}} = \{R/\p^\infty, Q(R)\} \qquad \overline{\{R_\p^\wedge\}} = \{R_\p^\wedge, Q(R)\}.
\]
Therefore $\msf{KZg}_\msf{Cl}(R)$ consists of the points $R/\p^n$ (for $1\leq n<\infty$) and $Q(R)$. Applying \cref{prop:hereditarycategoryspectra}, \cref{definingalpha} and \cref{lem:thicksamepreimage}, we obtain that
\[
\sspec{\msf{D}(R)^{\c}} = \{{}^{\perp_{\Z}}R/\p : \text{$\p$ is a prime ideal of $R$}\} \cup \{{}^{\perp_{\Z}} Q(R)\}
\] 
as a set. 

Now, there is, by \cite[Proposition 7.9]{McRob}, an order-preserving bijection between prime two-sided ideals of $R$ and prime ideals of $Z(R)$, given by $I\mapsto I\cap Z(R)$ with inverse $J\mapsto JR$. The map $\phi\colon \sspec{\msf{D}(R)^\c} \to \msf{Spec}(Z(R))$ defined by ${}^{\perp_{\Z}} R/\p \mapsto \p\cap Z(R)$ and ${}^{\perp_{\Z}} Q(R) \mapsto 0$ is a homeomorphism. Indeed the only closed sets in $\msf{Spec}(Z(R))$ are the whole space and the finite unions of prime ideals, and 
\[\ssupp{R/\p}=  \{\msf{P} \in \sspec{\msf{D}(R)^\c} : R/\mf{p}\not\in\msf{P}\} = \{{}^{\perp_{\Z}} R/\p\}=\phi^{-1}(\p),\]
hence the map is continuous and it is clearly closed.
\end{proof}

\subsection{Rank functions and characters}
The decomposition of derived categories for hereditary categories also enables us to better understand rank functions. In this subsection we relate characters on $\mod{R}$ in the sense of \cite{CrawleyBoevey} to irreducible rank functions on $\D(R)^\c$ for arbitrary rings $R$ and then apply it in the hereditary case. 

The following result will be useful. It is well known to specialists, but we were not able to find a proof in this generality in the literature so we include the proof.

\begin{lem}\label{lem:DefFunPresEndofinite}
Let $\msf{A}$ and $\msf{B}$ be finitely accessible categories with products. If $F\colon\msf{A}\to\msf{B}$ is a definable functor, then $F$ preserves endofinite objects.
\end{lem}

\begin{proof}
Following \cite{CrawleyBoevey}, we write $\mbb{D}(\A) := \msf{Flat}(\msf{fp}(\msf{fp}(\A),\ab)^{\op})$ and $d_\A\colon \A \to \mbb{D}(\A)$ for the canonical functor induced by the Yoneda embedding. By \cite[Corollary 10.5]{Krauselfp}, for any definable functor $F\colon \A \to \B$ there is an adjunction
\[
\begin{tikzcd}
\mbb{D}(\B) \arrow[r, shift left=1ex, "\Lambda"] \arrow[r, leftarrow, swap, shift right=1ex , "F^{*}"]& \mbb{D}(\A)
\end{tikzcd}
\]
where $F^{*}$ is the unique definable functor extending $F$, and the functor $\Lambda$ is exact and preserves finitely presented objects. 

By \cite[\S 3.6, Lemma 1]{CrawleyBoevey}, $d_\A$ identifies the endofinite objects in $\A$ with the endofinite injective objects of $\mbb{D}(\A)$. Since $F^*$ preserves injectives (being the right adjoint to an exact functor), it therefore suffices to show that $F^{*}$ preserves endofinite objects. If $B\in\msf{fp}(\mbb{D}(\B))$, then there is an isomorphism
\[
\Hom_{\mbb{D}(\B)}(B,F^{*}X)\simeq\Hom_{\mbb{D}(A)}(\Lambda(B),X)
\] 
of $\t{End}_{\mbb{D}(\A)}(X)$-modules. As $\Lambda(B)$ is finitely presented, it follows that $\Hom_{\mbb{D}(\B)}(B,F^{*}X)$ is of finite length over $\t{End}_{\mbb{D}(\A)}(X)$, and thus also over $\t{End}_{\mbb{D}(\B)}(F^{*}X)$, and thus $F^{*}$ preserves endofinite objects, as desired.
\end{proof}

\begin{chunk}\label{chunk:charactersendo}
Recall from \cite[\S 3.6]{CrawleyBoevey} that any endofinite $R$-module $M$ gives a character $\chi_M$ on $\mod{R}$ defined by $\chi_M = \mrm{length}_{\mrm{End}(M)}\Hom_{R}(-,M)$. Moreover, this assignment restricts to a bijection between indecomposable endofinite $R$-modules and irreducible characters, see~\cite[\S 3.6, Theorem]{CrawleyBoevey}.
\end{chunk}

This motivates the need to understand the relationship between endofinite objects in $\Mod{R}$ and $\Sigma$-invariant endofinite objects in $\D(R)$ for a ring. Recall that $X^{\Sigma}$ denotes the object $\bigoplus_{i\in\Z}\Sigma^{i} X$. 

\begin{lem}\label{lem:CoprodStalksEndo}
Let $R$ be any ring and $X\in\Mod{R}$ be an endofinite object. Then $X^{\Sigma}$
is a $\Sigma$-invariant endofinite object in $\D(R)$.
\end{lem}

\begin{proof}
Consider the functor $F\colon \Mod{R}\to \Mod{\D(R)^{\c}}$ given by $X\mapsto \y\left(X^{\Sigma}\right)$.
The functor $F$ preserves direct limits by \cite[Corollary 3.5]{LakingVitoria}, while it preserves products since they are computed degreewise in $\D(R)$. Therefore $F$ is a definable functor. Thus, if $X\in\Mod{R}$ is endofinite, by \cref{lem:DefFunPresEndofinite} the object $FX\in\Mod{\D(R)^{\c}}$ is an endofinite cohomological functor. The restricted Yoneda embedding restricts to an equivalence between endofinite objects in $\D(R)$ and endofinite cohomological functors in $\Mod{\D(R)^{\c}}$ (see~\cref{prel:endofinite}), and thus $X^{\Sigma}$
is itself endofinite in $\D(R)$. The object $X^\Sigma$ is $\Sigma$-invariant by construction.
\end{proof}

\begin{prop}\label{prop:EmbedCharsinRank}
Let $R$ be any ring. Then there is an injective map 
\[
\iota\colon\{\t{Irreducible characters on $\mod{R}$}\}\to \{\t{Irreducible rank functions on $\D(R)^\c$}\}
\]
induced by the functor $(-)^\Sigma$. On the other hand, there is a map
\[
\pi\colon\{\t{Irreducible rank functions on $\D(R)^\c$}\} \to \{\t{Characters on $\mod{R}$}\}
\]
induced by $H_{0}(-)$.
Moreover, the composition $\pi\circ\iota$ is the identity when restricted to irreducible characters on $\mod{R}$, and thus every irreducible character on $\mod{R}$ is in the image of $\pi$.
\end{prop}

\begin{proof}
In \cite[Theorem 5.5]{conde2022functorial}, there is a bijection between $\Sigma$-invariant endofinite objects in $\D(R)$ and irreducible rank functions on $\D(R)^{\c}$. Applying \cref{lem:CoprodStalksEndo}, we see that $(-)^\Sigma$ sends endofinite $R$-modules to $\Sigma$-invariant endofinite objects of $\D(R)$. As such, by \cref{chunk:charactersendo} we obtain an irreducible rank function from an irreducible character. This assignment is injective by construction and the fact that endofinite objects decompose uniquely into a coproduct of indecomposable summands.

For the second assignment, the functor $H_{0}(-)\colon\D(R)\to\Mod{R}$ is coherent. In particular, by \cite[Theorem 3.2]{definablefunctors} there is a unique definable functor $\widehat{H_{0}}(-)\colon\Flat{\D(R)^{\c}}\to\Mod{R}$ such that $\widehat{H_{0}}\circ\y=H_{0}$. Since $\y$ and $\widehat{H_{0}}$ preserve endofinite objects, it follows that $H_{0}$ also does. In particular, given a $\Sigma$-invariant endofinite object $E\in\D(R)$, we deduce that $H_{0}(E)$ is an endofinite object in $\Mod{R}$, so it gives a character on $\mod{R}$. The composition $\pi \circ \iota$ is the identity when restricted to irreducible characters since $H_0(M^\Sigma) = M$ for any $M \in \Mod{R}$.
\end{proof}

\begin{rem}
The character induced by homology will, in general, not be irreducible since $H_{0}(X)$ need not be indecomposable. However by \cite[Proposition 13.2.2]{krbook}, there are pairwise non-isomorphic indecomposable endofinite $R$-modules $E_{1},\cdots,E_{n}$ such that $H_{0}(X)\simeq E_{1}^{(\lambda_{1})}\oplus\cdots\oplus E_{n}^{(\lambda_{n})}$, and each of these $E_{i}$ gives an irreducible character.
\end{rem}

\begin{rem}
    Note that it is not true that starting with an irreducible character $\chi$ on $\mod R$ and considering the rank function $\iota(\chi)$ on $\D(R)^{c}$ one has $\chi(M)=\iota(\chi)(M)$ for $M\in \D(R)^{\c} \cap \mod R$. 
\end{rem}

\begin{cor}
Let $R$ be a hereditary ring. The maps $\pi$ and $\iota$ induce a bijection between irreducible characters on $\mod{R}$ and irreducible rank functions on $\D(R)^\c$.
\end{cor}
\begin{proof}
    Since $R$ is hereditary, any $M \in \D(R)$ satisfies $M \simeq \bigoplus_{i \in \Z} \Sigma^i H_iM$. Therefore, $(H_0 M)^\Sigma \simeq M$ from which the claim follows.
\end{proof}

Note that in \cite[Section 6.1]{chuanglazarev} the connection between localising rank functions and Sylvester rank functions was already explored for hereditary algebras.

\subsection{Tame hereditary algebras}
From now on we shall consider finite dimensional tame hereditary algebras $\Lambda$ over $k$. For simplicity we will assume that $k$ is an algebraically closed field. Since hereditary algebras of finite representation type were covered in \cref{sec:locallyfinite} we will usually assume that $\Lambda$ is not of finite representation type. For simplicity we will also assume $\Lambda$ to be connected. Let us first recall the structure of $\mod{\Lambda}$ for such an algebra $\Lambda$. The finite dimensional indecomposable modules fit into three families: the preprojective modules $\mathbf{p}$, the preinjective modules $\mathbf{q}$, and the regular modules $\mathbf{r}$. We will denote the full additive subcategories of $\mod{\Lambda}$ containing those indecomposable modules in the same way. These classes satisfy the orthogonality conditions \[\Hom_{\Lambda}(\mathbf{q},\mathbf{r})=\Hom_{\Lambda}(\mathbf{r},\mathbf{p})=\Hom_{\Lambda}(\mathbf{q},\mathbf{p})=0\] and any morphism $\mathbf{p}\to\mathbf{q}$ factors through $\mathbf{r}$.

The Ziegler spectrum of $\Mod{\Lambda}$ is also known, and consists of the following points:
\begin{enumerate}
\item the finite dimensional modules;
\item the generic module $G$;
\item the Pr\"{u}fer modules;
\item the adic modules.
\end{enumerate}
A survey of these facts, as well as the definition of these modules, can be found in \cite[\S 8.1]{psl}. 

The following fact concerning the topology of $\msf{Zg}(\Lambda)$ is all we need to describe $\msf{KZg}_{\msf{Cl}}(\Lambda)$, which is, by \cref{prop:hereditarycategoryspectra}, homeomorphic to $\shspec{\D(\Lambda)^{\c}}$ when equipped with the GZ-topology. If $\msf{X}\subset\msf{Zg}(\Lambda)$ contains an infinitely generated module and no finite dimensional modules, then $\msf{X}$ is closed if and only if it contains the generic module, see \cite[Theorem A]{RingelTH}.

The finite dimensional modules and the generic module are all endofinite, and are therefore closed in $\msf{Zg}(\Lambda)$, and therefore are all distinct in $\msf{KZg}(\Lambda)$. By the preceding topological fact, we see that, for any adic or Pr\"{u}fer point $X\in\msf{Zg}(\Lambda)$, the set $\{X,G\}$ is closed. In particular, we cannot have $\msf{Def}(X)=\msf{Def}(Y)$ for any distinct points $X,Y\in\msf{Zg}(\Lambda)$. Consequently, the points of $\msf{KZg}_{\msf{Cl}}(\Lambda)$ are the finite dimensional modules, and the generic module. Together with \cref{prop:hereditarycategoryspectra} this gives the following description of the points of the shift-homological spectrum.

\begin{prop}
   Let $\Lambda$ be a tame hereditary algebra over a field. Then the shift-homological spectrum $\shspec{\D(\Lambda)^{\c}}$ is in bijection with the set of indecomposable finite dimensional modules of $\Lambda$ and the generic module over $\Lambda$. \qed
\end{prop}

Moreover, the next proposition shows that the isolation condition holds for $\D(\Lambda)$, and hence we get a  homeomorphism $\mc{K}\colon \Irr{\T^\c} \xrightarrow{\sim} \shspec{\D(\Lambda)^{\c}}$ by \cref{cor:Kishomeo}. In particular, every shift-prime thick subcategory is the kernel of an irreducible rank function. Recall that a nonzero object in an additive category with infinite coproducts is \emph{superdecomposable} if it has no indecomposable direct summands. 

\begin{prop}\label{thm:isolationhered}
The isolation condition holds for $\D(\Lambda)$, where $\Lambda$ is a tame hereditary artin algebra.
\end{prop}
\begin{proof}
Let us first observe that if $\D(\Lambda)$ contains no superdecomposable pure injective object, then it satisfies the isolation condition. This follows from \cite[Theorem 11.2]{kredc} and passing the result through the equivalence of categories $\y\colon\msf{Pinj}(\T)\xrightarrow{\sim}\msf{Pinj}(\msf{Flat}(\T^{\c}))$. Now, suppose that $X$ is a superdecomposable pure injective object of $\D(\Lambda)$. As $X\simeq\oplus_{i \in \Z}\Sigma^{i}H_{i}X$, it follows that there is an $i$ such that $\Sigma^{i}H_{i}X$ is a superdecomposable pure injective object, and thus $H_{i}X$ is a superdecomposable pure injective $\Lambda$-module. Yet tame hereditary artin algebras have no superdecomposable pure injective modules, see \cite[Summary 8.71]{jl} or \cite[\S5.1]{Prestsuperdec}, giving a contradiction.
\end{proof}

\begin{rem}
As a corollary of the above proof, we see that any algebra that is derived equivalent to a tame hereditary artin algebra also satisfies the isolation condition. Suppose $A$ is such an algebra. If its category of modules contained a superdecomposable pure injective module, say $E$, then $E$ is a pure injective object in $\msf{D}(A)$ which remains superdecomposable, and this cannot happen by the above. Thus $A$ admits no superdecomposable pure injective modules so satisfies the isolation condition.
\end{rem}

\begin{chunk}\label{th:wide}  
Let us now turn our attention to proving \cref{thm:tamheredeverythingradical}, that all thick subcategories of $\D(\Lambda)^{\c}$ are radical. By \cref{thm:isolationhered}, we know that thick subcategories are given as intersections of kernels of rank functions or equivalently as perpendiculars to sets of endofinite objects. We will realize every thick subcategory in this way. First, let us state the result of Br\"uning reducing the classification of thick subcategories of $\D(\Lambda)^\c\simeq \D^{\mrm{b}}(\mod{\Lambda})$ to the classification of wide subcategories of $\mod{\Lambda}$. Recall that a full subcategory $\mathcal{W}$ of an abelian category is called \textit{wide}, if it is closed under extensions, and kernels and cokernels of morphisms in $\mathcal{W}$. Note that wide subcategories are in turn abelian with the induced abelian structure.
\end{chunk}

\begin{thm}[{\cite[Theorem 1.1]{bruning}}]\label{brunbij}
Let $\mathcal{A}$ be a hereditary abelian category.
The assignments 
\[\msf{L}\mapsto \{H_{0}(X): X\in \msf{L}\} \quad \text{ and } \quad \mathcal{W}\mapsto \{X\in \D^{\mrm{b}}(\mathcal{A}): H_{n}(X)\in \mathcal{W} \text{ for all } n\in \mathbb{Z}\}\] 
give a bijection between the set of thick subcategories of $\D^{\mrm{b}}(\mathcal{A})$ and
the set of wide subcategories of $\mathcal{A}$.
\end{thm}

\begin{chunk}\label{th:exceptional}
Thick subcategories of $\D(\Lambda)^\c$, or equivalently wide subcategories of $\mod{\Lambda}$, can be split into two families. The first consists of those which can be obtained from an exceptional sequence in $\mod{\Lambda}$, and the second consists of those which cannot.

Recall that a module $M\in\mod{\Lambda}$ is \emph{exceptional} provided $M$ is indecomposable and $\t{Ext}_{\Lambda}^{1}(M,M)=0$. A sequence $\bar{X}=(X_{1},\cdots,X_{t})$ of $\Lambda$-modules is \emph{exceptional} if each $X_{j}$ is exceptional and additionally $\Hom_{\Lambda}(X_{j},X_{i})=0 =\t{Ext}_{\Lambda}^{1}(X_{j},X_{i})$ for all $1\leq i<j\leq t$.
\end{chunk}

Given an exceptional sequence $\bar{X}=(X_{1},\cdots,X_{t})$, one can consider the smallest wide subcategory containing the modules $X_{1},\dots,X_{t}$, and we denote this wide subcategory by $\mc{W}(X_{1},\cdots,X_{t})$ or $\mc{W}(\bar{X})$. Those wide subcategories of $\mod{\Lambda}$ arising from exceptional sequences have particularly pleasant properties.

\begin{thm}[{\cite[Theorem A.4]{HuberyKrause}}]\label{ththm:WideEx} Let $\Lambda$ be a hereditary artin algebra and $\mathcal{W}$ be a wide subcategory of $\mod{\Lambda}$. The following are equivalent:
\begin{enumerate}
    \item  $\mathcal{W}=\mathcal{W}(\bar{X})$ for some exceptional sequence $\bar{X}$;
    \item the inclusion $\mathcal{W}\rightarrow\mod\Lambda$ admits a left adjoint;
    \item the inclusion $\mathcal{W}\rightarrow\mod\Lambda$ admits a right adjoint;
    \item there is a finite homological epimorphism $f\colon \Lambda\rightarrow B$ such that restriction of scalars along $f$
induces an equivalence $\mod{B}\simeq \mathcal{W}$;
    \item there is an exceptional sequence $\bar{Y}$ in $\mod{\Lambda}$ such that $\mathcal{W}=\bar{Y}^{\perp_{0,1}}$;
    \item there is an exceptional sequence $\bar{Z}$ in $\mod{\Lambda}$ such that $\mathcal{W}=\,^{\perp_{0,1}} \bar{Z}$.
    \end{enumerate}
\end{thm}

We first check that thick subcategories given by an exceptional sequence in $\mod{\Lambda}$ are radical.

\begin{lem}\label{lem:ExceptionalIsRadical}
    Let $\msf{L}$ be a thick subcategory of $\D(\Lambda)^\c$ corresponding to a wide subcategory $\mathcal{W}={}^{\perp_{0,1}}Z$ for some object $Z\in \mod{\Lambda}$, then $\msf{L}$ is radical. In particular, if $\msf{L}$  corresponds to a wide subcategory given by an exceptional sequence, then $\msf{L}$ is radical. 
\end{lem}

\begin{proof}
Let $\msf{L}=\{X\in\D(\Lambda)^\c : \Hom_{\D(\Lambda)^\c}(X,\Sigma^jZ)=0\text{ for all }j\in\Z\}$, which is a radical thick subcategory since $Z\in\D(\Lambda)^{\c}$ as $\Lambda$ is hereditary, see \cref{ex:foralgebraperpsareradical}. It is clear that $X\in\msf{L}$ if and only if $H_{i}(X)\in \mc{W}$ for all $i\in \Z$, so by the bijection in \cref{brunbij} the result follows. The claim about exceptional sequences follows from \cref{ththm:WideEx}.
\end{proof}

\begin{chunk}
For the second family of wide (equivalently thick) subcategories, those which are not orthogonal to an exceptional sequence, more work is required. Any such wide subcategory consists only of regular modules, see \cite[Proposition 6.14]{koehler}. For that reason, let us recall the structure of the regular part of $\mod{\Lambda}$ in more detail. Note that the regular part $\mathbf{r}$ is a wide subcategory of $\mod{\Lambda}$, so it can be considered as an abelian category with the induced structure.

 The regular part of the module category is of the form 
\[
\prod_{j\in \mathbb{P}^1_k}\mathcal{H}_j\times \prod_{i=1}^s\mathcal{U}_{n_i},
\]
where the components $\mathcal{H}_j$ are \emph{homogeneous tubes} of rank $1$ and the components $\mathcal{U}_{n_i}$ are \emph{non-homogeneous tubes} of rank $n_i>1$ with $s\leq 3$ \cite{Ringelinfdimreps}. 
We will be mostly working with the non-homogeneous tubes, so let us concentrate on their structure more closely. Note that we call these components tubes since their Auslander-Reiten quiver is of the form 
$\Z A_\infty/\langle \tau^{n}\rangle$, where $n$ denotes the rank. The tube is then the additive closure of its indecomposable objects, that is, the objects appearing in its Auslander-Reiten quiver. 

A non-homogeneous tube $\mathcal{U}_n$ has $n$ quasi-simple modules $R_1,\dots,R_{n}$ on the rim of its Auslander-Reiten quiver (simple objects of $\mathbf{r}$), and the whole tube is precisely the subcategory of modules admitting a filtration with factors $R_1,\dots,R_{n}$. The category $\mathcal{U}_n$ is a uniserial subcategory of $\mod{\Lambda}$ and each indecomposable object of $\mathcal{U}_n$ is determined up to isomorphism by its filtration with quasi-simple factors, in particular, it is determined by its regular socle (the unique quasi-simple submodule of $M$) and regular length (the length of the filtration with factors $R_1,\dots,R_{n}$). We will denote such an object $M$ by $R_i^m$, where $m$ is the length of the filtration and $R_i$ is the regular socle of $M$. 
\end{chunk}

\begin{chunk}
An explicit description of the thick subcategories of $\D(\Lambda)^\c$, or equivalently wide subcategories of $\mod{\Lambda}$, was obtained in \cite{Dichev} in a special case and in \cite{koehler} in full generality; we will follow the exposition from \cite{koehler}. 

 For a wide subcategory $\mathcal{W}\subseteq\mod{\Lambda}$, a tube of rank one is either completely contained in $\mathcal{W}$ or none of its objects are in $\mathcal{W}$.

The situation with tubes of higher rank is a little bit more complicated and requires some combinatorics. Take a tube of rank $n$ and consider a circle with $n$ marked points denoted $(1,\dots,n)$ in a clockwise order. Modules $R_i^m$ of regular length $\leq n$ can be represented as arcs  without self intersections $(i,i+m \mbox{ } \mod n)$ on the circle $(1,\dots,n)$. Two arcs will be called \emph{non-crossing} if their interiors do not intersect, and if their starting points, respectively their endpoints, do not coincide. Note that the starting point of one arc can coincide with the endpoint of the other and the arcs will still be considered non-crossing (see \cite{koehler} for details). 

\end{chunk}

\begin{prop}[{\cite[Proposition 6.10]{koehler}}]\label{wideandarcs}
Let $\mc{U}_n$ be a tube with $n$ simple objects. There is a bijective correspondence between the set of wide subcategories of $\mc{U}_n$ and the set of non-crossing collections of arcs on a circle with $n$ points.
\end{prop}

This bijection assigns to a collection of arcs $x=(i_t,j_t)_{t\in I}$ the smallest wide subcategory containing objects $R_{i_t}^{m_t}$, where $m_t$ is the smallest positive integer with $m_t= j_t-i_t \mbox{ } \mod n$. To a wide subcategory $\mathcal{W}$ of $\mc{U}_n$ one associates the collection of the simple objects in $\mathcal{W}$, which turns out to be a collection of pairwise orthogonal bricks $R_{i_t}^{m_t}$, which gives a collection of non-crossing arcs $(i_t,i_t+m_t \mbox{ }  \mod n)$. 

A non-crossing collection of arcs $x$ is called \emph{exceptional} if there is no subset of arcs in $x$ that can be arranged into a family of arcs $(i_1,j_1),\dots,(i_l,j_l)$ such that $j_1=i_2$, $j_2=i_3$, ... and $j_l=i_1$. The wide subcategories of $\mod{\Lambda}$, contained in the regular part $\prod\limits_{j\in \mathbb{P}^1_k}\mathcal{H}_j\times \prod_{i=1}^s\mathcal{U}_{n_i}$ can be parametrised by 
\[\{(p, x_1, . . . , x_s) : p \in  2^{\mathbb{P}^1_k}, \, x_i \text{ a non-crossing collection of arcs on } n_i\}.
\]
The corresponding wide  subcategory $\mathcal{W}$ is the additive closure of the homogeneous tubes parametrised by the subset $p$ of $\mathbb{P}^1_k$ and the wide subcategories of $\mathcal{U}_{n_i}$, corresponding to $x_i$. Finally, $\mathcal{W}$ is generated by an exceptional sequence, if $p=\varnothing$ and each $x_i$ is exceptional, see \cite[Theorem 6.17]{koehler}.

\begin{chunk}
The goal for the rest of the section is to prove that all thick subcategories of $\D(\Lambda)^\c$ are radical. By \cref{lem:ExceptionalIsRadical} we only need to deal with wide subcategories not corresponding to an exceptional sequence, in particular, only with wide subcategories contained in the regular part. 

Throughout we will frequently use the Auslander-Reiten formulae, which for finite dimensional hereditary algebras take the form
\begin{equation}\label{ARform}
    D\t{Ext}_{\Lambda}^{1}(X,Y)\simeq \Hom_{\Lambda}(Y,\tau X) \quad \t{and} \quad \t{Ext}_{\Lambda}^{1}(Y,X)\simeq D\Hom_{\Lambda}(\tau^{-1}X,Y),
\end{equation}
where $X$ is finite dimensional and $Y$ is an arbitrary $\Lambda$-module (see e.g. \cite[Corollary 1.7]{stovthesis}). Here $D$ stands for the $k$-linear dual. 

Let us first consider the case of one tube of rank $n>1$ separately. 

\end{chunk}

\begin{prop}\label{prop:thwideisperp}
Let $\mc{U}_{n}$ be a tube of rank $n>1$ considered as an abelian category. Then any wide subcategory of $\mc{U}_{n}$ is of the form $\mathcal{W}={}^{\perp_{0,1}} Z$ for some $Z\in \mc{U}_{n}$.
\end{prop}

\begin{proof}
Let $x$ be a non-crossing collection of arcs corresponding to $\mathcal{W}$ under the bijection of \cref{wideandarcs}. Let us first consider the case when $x$ is not exceptional, that is, when $x$ has a subset of arcs that can be
arranged into a family of arcs $(i_1,j_1),\dots,(i_l,j_l)$ such that $j_1=i_2$, $j_2=i_3$, ... and $j_l=i_1$; the remaining arcs in $x$ will be denoted by $(k_1,h_1),\dots,(k_r,h_r)$. As before, the objects corresponding to $(i_1,j_1),\dots,(i_l,j_l)$ are denoted by $R_{i_1}^{m_1},\dots, R_{i_l}^{m_l}$. 

The wide subcategory $\mathcal{W}(R_{i_1}^{m_1},\dots, R_{i_l}^{m_l})$ consists of direct sums of indecomposable objects having a filtration with factors $R_{i_1}^{m_1},\dots, R_{i_l}^{m_l}$. Clearly such indecomposable objects are of the form $R_{i_u}^{m}$, where $m=an+m_u+m_{u+1}+\dots + m_{u+h}$ with $h<l$ and  $a\geq 0$. There is an equality
\[
\msf{ind}(\mc{W}) = \msf{ind}(\mc{W}(R_{i_1}^{m_1},\dots, R_{i_l}^{m_l}))\sqcup \msf{ind}(\mathcal{W}(R_{k_1}^{w_1},\dots, R_{k_r}^{w_r})),
\]
of indecomposable objects, and for brevity we set $\mathcal{W}'=\mathcal{W}(R_{k_1}^{w_1},\dots, R_{k_r}^{w_r})$.

Consider the  object \[Z_1=\bigoplus\limits_{i=1}^l R_{i_1}^{m_1-1}\] where by convention an object of regular length $0$ is $0$. Let us compute ${}^{\perp_{0,1}} Z_1$. For computations its convenient to use rays and corays on the Auslander-Reiten quiver of the tube. A ray starting at a regular simple $R_i^1$ consists of indecomposable objects with the regular socle $R_i^1$, so objects of the form $R_i^m$, $m\geq 1$. A coray finishing at a regular simple $R_i^1$ consists of indecomposable objects with the regular top $R_i^1$, that is the objects of the form $R_j^m$ with $j+m \equiv i+1 \mbox{ } \mod n$. This computation is illustrated with \cref{figure}.

Indecomposable objects $R$ with $\Hom_{\Lambda}(R,R_{i_u}^{m_u-1})\neq 0$ are from the coray finishing at $R_{i_u}^1$ of any regular length, the coray finishing at $R_{i_u+1}^1$ of regular length at least $2$, and so on, from the coray finishing at $R_{i_u+m_u-2}^1$ of regular length at least $m_u-1$. 

Indecomposable objects $R$ with $\t{Ext}_{\Lambda}^{1}(R,R_{i_u}^{m_u-1})\neq 0$ are, by the Auslander-Reiten formula (\ref{ARform}), those with $\Hom_{\Lambda}(\tau^{-1}R_{i_u}^{m_u-1},R)\neq 0$. Since $\tau^{-1}R_{i_u}^{m_u-1}=R_{i_u+1}^{m_u-1}$, we get $R$'s from the ray starting at $R_{i_u+1}^1$ of regular length at least $m_u-1$, the ray starting at $R_{i_u+2}^1$ of regular length at least $m_u-2$, and so on, from the ray starting at $R_{i_{u}+m_{u}-1}^1=R_{i_{u+1}-1}^1$ of any regular length.

Thus we get that 
\[
\{X\in\mc{U}_{n}: \Hom_{\Lambda}(X,Z_1)= 0\}=\mathcal{W}(R_{i_1}^{m_1},\dots,R_{i_l}^{m_l}) \oplus \bigoplus_{u=1}^l \mathcal{W}(R_{i_u+1}^{1},R_{i_u+2}^{1} \dots,R_{i_{u+1}-2}^{1}).
\]

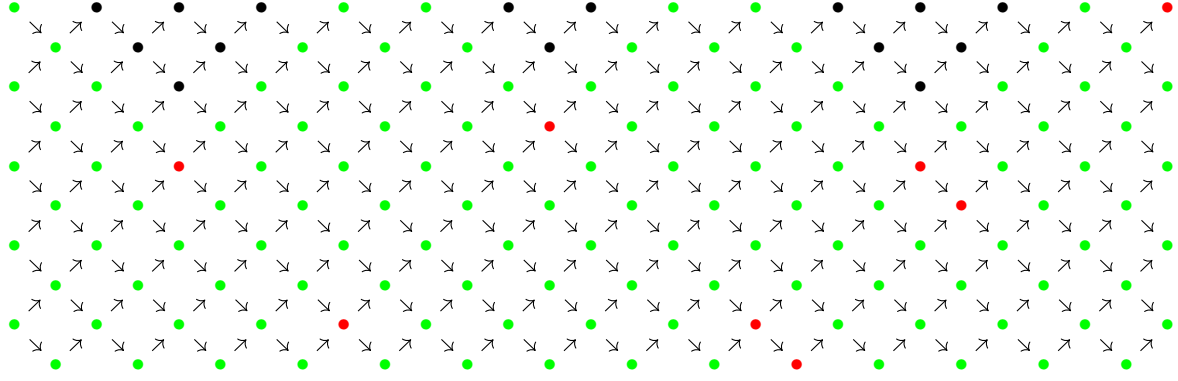
\begin{figure}[h!]
\begin{tikzcd}[cramped, sep=4.5pt]
	{\color{green} \bullet} && \bullet && \bullet && \bullet && {\color{green} \bullet} && {\color{green} \bullet} && \bullet && \bullet && {\color{green} \bullet} && {\color{green} \bullet} && \bullet && \bullet && \bullet && {\color{green} \bullet} && {\color{red} \bullet} \\
	& {\color{green} \bullet} && \bullet && \bullet && {\color{green} \bullet} && {\color{green} \bullet} && {\color{green} \bullet} && \bullet && {\color{green} \bullet} && {\color{green} \bullet} && {\color{green} \bullet} && \bullet && \bullet && {\color{green} \bullet} && {\color{green} \bullet} \\
	{\color{green} \bullet} && {\color{green} \bullet} && \bullet && {\color{green} \bullet} && {\color{green} \bullet} && {\color{green} \bullet} && {\color{green} \bullet} && {\color{green} \bullet} && {\color{green} \bullet} && {\color{green} \bullet} && {\color{green} \bullet} && \bullet && {\color{green} \bullet} && {\color{green} \bullet} && {\color{green} \bullet} \\
	& {\color{green} \bullet} && {\color{green} \bullet} && {\color{green} \bullet} && {\color{green} \bullet} && {\color{green} \bullet} && {\color{green} \bullet} && {\color{red} \bullet} && {\color{green} \bullet} && {\color{green} \bullet} && {\color{green} \bullet} && {\color{green} \bullet} && {\color{green} \bullet} && {\color{green} \bullet} && {\color{green} \bullet} \\
	{\color{green} \bullet} && {\color{green} \bullet} && {\color{red} \bullet} && {\color{green} \bullet} && {\color{green} \bullet} && {\color{green} \bullet} && {\color{green} \bullet} && {\color{green} \bullet} && {\color{green} \bullet} && {\color{green} \bullet} && {\color{green} \bullet} && {\color{red} \bullet} && {\color{green} \bullet} && {\color{green} \bullet} && {\color{green} \bullet} \\
	& {\color{green} \bullet} && {\color{green} \bullet} && {\color{green} \bullet} && {\color{green} \bullet} && {\color{green} \bullet} && {\color{green} \bullet} && {\color{green} \bullet} && {\color{green} \bullet} && {\color{green} \bullet} && {\color{green} \bullet} && {\color{green} \bullet} && {\color{red} \bullet} && {\color{green} \bullet} && {\color{green} \bullet} \\
	{\color{green} \bullet} && {\color{green} \bullet} && {\color{green} \bullet} && {\color{green} \bullet} && {\color{green} \bullet} && {\color{green} \bullet} && {\color{green} \bullet} && {\color{green} \bullet} && {\color{green} \bullet} && {\color{green} \bullet} && {\color{green} \bullet} && {\color{green} \bullet} && {\color{green} \bullet} && {\color{green} \bullet} && {\color{green} \bullet} \\
	& {\color{green} \bullet} && {\color{green} \bullet} && {\color{green} \bullet} && {\color{green} \bullet} && {\color{green} \bullet} && {\color{green} \bullet} && {\color{green} \bullet} && {\color{green} \bullet} && {\color{green} \bullet} && {\color{green} \bullet} && {\color{green} \bullet} && {\color{green} \bullet} && {\color{green} \bullet} && {\color{green} \bullet} \\
	{\color{green} \bullet} && {\color{green} \bullet} && {\color{green} \bullet} && {\color{green} \bullet} && {\color{red} \bullet} && {\color{green} \bullet} && {\color{green} \bullet} && {\color{green} \bullet} && {\color{green} \bullet} && {\color{red} \bullet} && {\color{green} \bullet} && {\color{green} \bullet} && {\color{green} \bullet} && {\color{green} \bullet} && {\color{green} \bullet} \\
	& {\color{green} \bullet} && {\color{green} \bullet} && {\color{green} \bullet} && {\color{green} \bullet} && {\color{green} \bullet} && {\color{green} \bullet} && {\color{green} \bullet} && {\color{green} \bullet} && {\color{green} \bullet} && {\color{red} \bullet} && {\color{green} \bullet} && {\color{green} \bullet} && {\color{green} \bullet} && {\color{green} \bullet}\\
	\arrow[from=1-1, to=2-2]
	\arrow[from=1-3, to=2-4]
	\arrow[from=1-5, to=2-6]
	\arrow[from=1-7, to=2-8]
	\arrow[from=1-9, to=2-10]
	\arrow[from=1-11, to=2-12]
	\arrow[from=1-13, to=2-14]
	\arrow[from=1-15, to=2-16]
	\arrow[from=1-17, to=2-18]
	\arrow[from=1-19, to=2-20]
	\arrow[from=1-21, to=2-22]
	\arrow[from=1-23, to=2-24]
	\arrow[from=1-25, to=2-26]
	\arrow[from=1-27, to=2-28]
	\arrow[from=2-2, to=1-3]
	\arrow[from=2-2, to=3-3]
	\arrow[from=2-4, to=1-5]
	\arrow[from=2-4, to=3-5]
	\arrow[from=2-6, to=1-7]
	\arrow[from=2-6, to=3-7]
	\arrow[from=2-8, to=1-9]
	\arrow[from=2-8, to=3-9]
	\arrow[from=2-10, to=1-11]
	\arrow[from=2-10, to=3-11]
	\arrow[from=2-12, to=1-13]
	\arrow[from=2-12, to=3-13]
	\arrow[from=2-14, to=1-15]
	\arrow[from=2-14, to=3-15]
	\arrow[from=2-16, to=1-17]
	\arrow[from=2-16, to=3-17]
	\arrow[from=2-18, to=1-19]
	\arrow[from=2-18, to=3-19]
	\arrow[from=2-20, to=1-21]
	\arrow[from=2-20, to=3-21]
	\arrow[from=2-22, to=1-23]
	\arrow[from=2-22, to=3-23]
	\arrow[from=2-24, to=1-25]
	\arrow[from=2-24, to=3-25]
	\arrow[from=2-26, to=1-27]
	\arrow[from=2-26, to=3-27]
	\arrow[from=2-28, to=1-29]
	\arrow[from=2-28, to=3-29]
	\arrow[from=3-1, to=2-2]
	\arrow[from=3-1, to=4-2]
	\arrow[from=3-3, to=2-4]
	\arrow[from=3-3, to=4-4]
	\arrow[from=3-5, to=2-6]
	\arrow[from=3-5, to=4-6]
	\arrow[from=3-7, to=2-8]
	\arrow[from=3-7, to=4-8]
	\arrow[from=3-9, to=2-10]
	\arrow[from=3-9, to=4-10]
	\arrow[from=3-11, to=2-12]
	\arrow[from=3-11, to=4-12]
	\arrow[from=3-13, to=2-14]
	\arrow[from=3-13, to=4-14]
	\arrow[from=3-15, to=2-16]
	\arrow[from=3-15, to=4-16]
	\arrow[from=3-17, to=2-18]
	\arrow[from=3-17, to=4-18]
	\arrow[from=3-19, to=2-20]
	\arrow[from=3-19, to=4-20]
	\arrow[from=3-21, to=2-22]
	\arrow[from=3-21, to=4-22]
	\arrow[from=3-23, to=2-24]
	\arrow[from=3-23, to=4-24]
	\arrow[from=3-25, to=2-26]
	\arrow[from=3-25, to=4-26]
	\arrow[from=3-27, to=2-28]
	\arrow[from=3-27, to=4-28]
	\arrow[from=4-2, to=3-3]
	\arrow[from=4-2, to=5-3]
	\arrow[from=4-4, to=3-5]
	\arrow[from=4-4, to=5-5]
	\arrow[from=4-6, to=3-7]
	\arrow[from=4-6, to=5-7]
	\arrow[from=4-8, to=3-9]
	\arrow[from=4-8, to=5-9]
	\arrow[from=4-10, to=3-11]
	\arrow[from=4-10, to=5-11]
	\arrow[from=4-12, to=3-13]
	\arrow[from=4-12, to=5-13]
	\arrow[from=4-14, to=3-15]
	\arrow[from=4-14, to=5-15]
	\arrow[from=4-16, to=3-17]
	\arrow[from=4-16, to=5-17]
	\arrow[from=4-18, to=3-19]
	\arrow[from=4-18, to=5-19]
	\arrow[from=4-20, to=3-21]
	\arrow[from=4-20, to=5-21]
	\arrow[from=4-22, to=3-23]
	\arrow[from=4-22, to=5-23]
	\arrow[from=4-24, to=3-25]
	\arrow[from=4-24, to=5-25]
	\arrow[from=4-26, to=3-27]
	\arrow[from=4-26, to=5-27]
	\arrow[from=4-28, to=3-29]
	\arrow[from=4-28, to=5-29]
	\arrow[from=5-1, to=4-2]
	\arrow[from=5-1, to=6-2]
	\arrow[from=5-3, to=4-4]
	\arrow[from=5-3, to=6-4]
	\arrow[from=5-5, to=4-6]
	\arrow[from=5-5, to=6-6]
	\arrow[from=5-7, to=4-8]
	\arrow[from=5-7, to=6-8]
	\arrow[from=5-9, to=4-10]
	\arrow[from=5-9, to=6-10]
	\arrow[from=5-11, to=4-12]
	\arrow[from=5-11, to=6-12]
	\arrow[from=5-13, to=4-14]
	\arrow[from=5-13, to=6-14]
	\arrow[from=5-15, to=4-16]
	\arrow[from=5-15, to=6-16]
	\arrow[from=5-17, to=4-18]
	\arrow[from=5-17, to=6-18]
	\arrow[from=5-19, to=4-20]
	\arrow[from=5-19, to=6-20]
	\arrow[from=5-21, to=4-22]
	\arrow[from=5-21, to=6-22]
	\arrow[from=5-23, to=4-24]
	\arrow[from=5-23, to=6-24]
	\arrow[from=5-25, to=4-26]
	\arrow[from=5-25, to=6-26]
	\arrow[from=5-27, to=4-28]
	\arrow[from=5-27, to=6-28]
	\arrow[from=6-2, to=5-3]
	\arrow[from=6-2, to=7-3]
	\arrow[from=6-4, to=5-5]
	\arrow[from=6-4, to=7-5]
	\arrow[from=6-6, to=5-7]
	\arrow[from=6-6, to=7-7]
	\arrow[from=6-8, to=5-9]
	\arrow[from=6-8, to=7-9]
	\arrow[from=6-10, to=5-11]
	\arrow[from=6-10, to=7-11]
	\arrow[from=6-12, to=5-13]
	\arrow[from=6-12, to=7-13]
	\arrow[from=6-14, to=5-15]
	\arrow[from=6-14, to=7-15]
	\arrow[from=6-16, to=5-17]
	\arrow[from=6-16, to=7-17]
	\arrow[from=6-18, to=5-19]
	\arrow[from=6-18, to=7-19]
	\arrow[from=6-20, to=5-21]
	\arrow[from=6-20, to=7-21]
	\arrow[from=6-22, to=5-23]
	\arrow[from=6-22, to=7-23]
	\arrow[from=6-24, to=5-25]
	\arrow[from=6-24, to=7-25]
	\arrow[from=6-26, to=5-27]
	\arrow[from=6-26, to=7-27]
	\arrow[from=6-28, to=5-29]
	\arrow[from=6-28, to=7-29]
	\arrow[from=7-1, to=6-2]
	\arrow[from=7-1, to=8-2]
	\arrow[from=7-3, to=6-4]
	\arrow[from=7-3, to=8-4]
	\arrow[from=7-5, to=6-6]
	\arrow[from=7-5, to=8-6]
	\arrow[from=7-7, to=6-8]
	\arrow[from=7-7, to=8-8]
	\arrow[from=7-9, to=6-10]
	\arrow[from=7-9, to=8-10]
	\arrow[from=7-11, to=6-12]
	\arrow[from=7-11, to=8-12]
	\arrow[from=7-13, to=6-14]
	\arrow[from=7-13, to=8-14]
	\arrow[from=7-15, to=6-16]
	\arrow[from=7-15, to=8-16]
	\arrow[from=7-17, to=6-18]
	\arrow[from=7-17, to=8-18]
	\arrow[from=7-19, to=6-20]
	\arrow[from=7-19, to=8-20]
	\arrow[from=7-21, to=6-22]
	\arrow[from=7-21, to=8-22]
	\arrow[from=7-23, to=6-24]
	\arrow[from=7-23, to=8-24]
	\arrow[from=7-25, to=6-26]
	\arrow[from=7-25, to=8-26]
	\arrow[from=7-27, to=6-28]
	\arrow[from=7-27, to=8-28]
	\arrow[from=8-2, to=7-3]
	\arrow[from=8-2, to=9-3]
	\arrow[from=8-4, to=7-5]
	\arrow[from=8-4, to=9-5]
	\arrow[from=8-6, to=7-7]
	\arrow[from=8-6, to=9-7]
	\arrow[from=8-8, to=7-9]
	\arrow[from=8-8, to=9-9]
	\arrow[from=8-10, to=7-11]
	\arrow[from=8-10, to=9-11]
	\arrow[from=8-12, to=7-13]
	\arrow[from=8-12, to=9-13]
	\arrow[from=8-14, to=7-15]
	\arrow[from=8-14, to=9-15]
	\arrow[from=8-16, to=7-17]
	\arrow[from=8-16, to=9-17]
	\arrow[from=8-18, to=7-19]
	\arrow[from=8-18, to=9-19]
	\arrow[from=8-20, to=7-21]
	\arrow[from=8-20, to=9-21]
	\arrow[from=8-22, to=7-23]
	\arrow[from=8-22, to=9-23]
	\arrow[from=8-24, to=7-25]
	\arrow[from=8-24, to=9-25]
	\arrow[from=8-26, to=7-27]
	\arrow[from=8-26, to=9-27]
	\arrow[from=8-28, to=7-29]
	\arrow[from=8-28, to=9-29]
	\arrow[from=9-1, to=8-2]
	\arrow[from=9-1, to=10-2]
	\arrow[from=9-3, to=8-4]
	\arrow[from=9-3, to=10-4]
	\arrow[from=9-5, to=8-6]
	\arrow[from=9-5, to=10-6]
	\arrow[from=9-7, to=8-8]
	\arrow[from=9-7, to=10-8]
	\arrow[from=9-9, to=8-10]
	\arrow[from=9-9, to=10-10]
	\arrow[from=9-11, to=8-12]
	\arrow[from=9-11, to=10-12]
	\arrow[from=9-13, to=8-14]
	\arrow[from=9-13, to=10-14]
	\arrow[from=9-15, to=8-16]
	\arrow[from=9-15, to=10-16]
	\arrow[from=9-17, to=8-18]
	\arrow[from=9-17, to=10-18]
	\arrow[from=9-19, to=8-20]
	\arrow[from=9-19, to=10-20]
	\arrow[from=9-21, to=8-22]
	\arrow[from=9-21, to=10-22]
	\arrow[from=9-23, to=8-24]
	\arrow[from=9-23, to=10-24]
	\arrow[from=9-25, to=8-26]
	\arrow[from=9-25, to=10-26]
	\arrow[from=9-27, to=8-28]
	\arrow[from=9-27, to=10-28]
	\arrow[from=10-2, to=9-3]
	\arrow[from=10-4, to=9-5]
	\arrow[from=10-6, to=9-7]
	\arrow[from=10-8, to=9-9]
	\arrow[from=10-10, to=9-11]
	\arrow[from=10-12, to=9-13]
	\arrow[from=10-14, to=9-15]
	\arrow[from=10-16, to=9-17]
	\arrow[from=10-18, to=9-19]
	\arrow[from=10-20, to=9-21]
	\arrow[from=10-22, to=9-23]
	\arrow[from=10-24, to=9-25]
	\arrow[from=10-26, to=9-27]
	\arrow[from=10-28, to=9-29]
\end{tikzcd}
\caption{
The start of the AR quiver of $\mc{U}_{n}$, which continues downwards indefinitely, and wraps horizontally with periodicity $n$. The red points are the indecomposable objects in $\mathcal{W}(R_{i_1}^{m_1},\dots,R_{i_l}^{m_l})$, the green points are the indecomposable objects with $\Hom_{\Lambda}(R,Z_1)\neq 0$ or $\t{Ext}_{\Lambda}^{1}(R,Z_1)\neq 0$, and the black points are the indecomposable objects in $\bigoplus_{u=1}^l \mathcal{W}(R_{i_u+1}^{1},R_{i_u+2}^{1} \dots,R_{i_{u+1}-2}^{1})$. 
}\label{figure}
\end{figure}

Since the wide subcategory $\mathcal{W}''=\bigoplus_{u=1}^l \mathcal{W}(R_{i_u+1}^{1},R_{i_u+2}^{1} \dots,R_{i_{u+1}-2}^{1})$ is given by an exceptional collection $(R_{i_1+1}^{1},R_{i_1+2}^{1} \dots,R_{i_2-2}^{1}, \dots, R_{i_l+1}^{1},R_{i_l+2}^{1} \dots,R_{i_{1}-2}^{1})$, by \cref{ththm:WideEx}(4) we have that $\mathcal{W}''\simeq \mod{B}$ is an image of the restriction of scalars along a finite homological epimorphism to an algebra $B$, which in this case is clearly hereditary of finite representation type. Note that $\Hom$ and $\t{Ext}$ are the same in $\mod{B}$ and in $\mod{\Lambda}$, 
by the definition of a homological epimorphism.
Any wide subcategory of $\mod{B}$ is given by an exceptional sequence inside of $\mod{B}$ and so is of the form ${}^{\perp_{0,1}}Z_B$ in $\mod{B}$ for some module $Z_B$. 
As $\mathcal{W}'\subseteq \mathcal{W}''$ there exists an object $Z_2\in \mathcal{W}''$ such that $\mathcal{W}'= {}^{\perp_{0,1}}Z_2 \cap \mathcal{W}''$ (the object $Z_2$ is the image of the corresponding $Z_B$ under the restriction of scalars). 

We obtain that 
\[ {}^{\perp_{0,1}}(Z_1\oplus Z_2)=(\mathcal{W}(R_{i_1}^{m_1},\dots,R_{i_l}^{m_l}) \bigoplus \mathcal{W}'') \cap {}^{\perp_{0,1}}Z_2=\mathcal{W}(R_{i_1}^{m_1},\dots,R_{i_l}^{m_l}) \bigoplus \mathcal{W}'\] in $\mc{U}_{n}$.
 The last equality follows since $\mathcal{W}(R_{i_1}^{m_1},\dots,R_{i_l}^{m_l})  \subseteq {}^{\perp_{0,1}}Z_2$. Hence for any non-crossing collection of arcs $x$ the corresponding wide subcategory is of the form ${}^{\perp_{0,1}}Z$.

 Now let us consider the case when the collection of arcs $x$ corresponding to $\mathcal{W}$ is exceptional. Since $x$ does not have a subset of arcs that can be
arranged into a family of arcs $(i_1,j_1),\dots,(i_l,j_l)$ such that $j_1=i_2$, $j_2=i_3$, $\cdots$, and $j_l=i_1$, there exists $i$ such that $x\cup (i,i+1)$ is still a non-crossing collection of arcs. Let us consider the wide subcategory $\mathcal{W}'$ corresponding to the collection of arcs $(i+1,i+2),(i+2,i+3),\dots,(i-1,i)$. This collection of arcs is exceptional, $\mathcal{W}'$ corresponds to the image of a finite homological epimorphism to a hereditary algebra of finite representation type $B$, and so $\mathcal{W}\subseteq \mathcal{W}'$ can be obtained inside of $\mathcal{W}'$ as ${}^{\perp_{0,1}}Z_2$ for some $Z_2\in \mathcal{W}'$. 

Let us construct $\mathcal{W}'$ as ${}^{\perp_{0,1}}Z_1$ for some object $Z_1\in \mathcal{U}_n$. We claim that $Z_1=R_i^n$ works. As before, the indecomposable objects $R$ with $\Hom_{\Lambda}(R,R_{i}^{n})\neq 0$ are from the coray finishing at $R_{i}^1$ of any regular length, the coray finishing at $R_{i+1}^1$ of regular length at least $2$, and so on, from the coray finishing at $R_{i-1}^1$ of regular length at least $n$. Indecomposable objects $R$ with $\t{Ext}_{\Lambda}^{1}(R,R_{i}^{n})\neq 0$ are the objects from the ray starting at $R_{i+1}^1$ of regular length at least $n$, the ray starting at $R_{i+2}^1$ of regular length at least $n-1$, and so on, from the ray starting at $R_{i}^1$ of any regular length. We see that $\mathcal{W}'={}^{\perp_{0,1}}R_{i}^n$.
Now $\mathcal{W}={}^{\perp_{0,1}}(Z_1\oplus Z_2)$ as desired, which finishes the proof.
\end{proof}

We are now in a position to prove the main result of this subsection. 
\begin{thm}\label{thm:tamheredeverythingradical}
    Let $\Lambda$ be a finite dimensional tame hereditary algebra over an algebraically closed field. Then every thick subcategory of $\D(\Lambda)^\c$ is radical.
\end{thm}
\begin{proof}
First let us check that the thick subcategory consisting of the shifts of all the regular modules is radical. Let $G$ be the generic module over $\mod{\Lambda}$. This is an indecomposable endofinite module, so the assignments in \cref{prop:EmbedCharsinRank} give an irreducible rank function $\rho_{G}$ on $\D(\Lambda)^{\c}$. A direct comparison using \cref{brunbij} shows that the kernel of $\rho_{G}$, which is radical by \cref{Kerrhoisradical}, is precisely the thick subcategory corresponding to the wide subcategory ${}^{\perp_{0,1}}G$. In order to compute ${}^{\perp_{0,1}}G$ let us recall some facts from \cite{RR06}. 

Firstly, it is shown in \cite[\S 3.2]{RR06} that there is a torsion pair $(\scr{D},\scr{R})$ in $\Mod{\Lambda}$ where $\scr{D}=\,^{\perp_{0}}\mathbf{r}=\mathbf{r}^{\perp_{1}}$. We note that we have torsion classes on the left of the pair unlike \cite{RR06}. It is stated on \cite[p. 192]{RR06} that $\mathbf{p},\mathbf{r}\subseteq\scr{R}$ and $\mathbf{q}\subseteq\scr{D}$. In \cite[Corollary 8.1]{RR06} it is shown that $\scr{D}=\msf{Gen}(G)$, the full subcategory consisting of quotients of coproducts of $G$. Since $\mathbf{q}\subseteq\scr{D}$, it follows that $\Hom_{\Lambda}(G,q)\neq 0$ for all nonzero $q\in \mathbf{q}$. Applying the Auslander-Reiten formula (\ref{ARform})  tells us that $\t{Ext}_{\Lambda}^{1}(q,G)\neq 0$ for all nonzero $q\in \mathbf{q}$. The same argument shows that $\t{Ext}_{\Lambda}^{1}(\mathbf{p},G)=0$. Combining \cite[\S 3.2]{RR06} with the fact that $\Hom_{\Lambda}(\msf{Gen}(\mathbf{q}),\msf{Add}(G))=0$, which appears on \cite[p. 206]{RR06}, we see that $\Hom_{\Lambda}(\mathbf{q},G)=0$.

It is shown in \cite[\S 14, p. 220]{RR06} that $\t{Ext}_{\Lambda}^{1}(G,\mathbf{r})=0$, in a discussion about cotorsion pairs.
Since for each $M\in\mathbf{r}$ we have $\tau^{-1}M \in\mathbf{r}$, it follows from the Auslander-Reiten formula (\ref{ARform}) that $\Hom_{\Lambda}(\mathbf{r},G)=0$.

The last thing to describe is $\Hom_{\Lambda}(\mathbf{p},G)$. In \cite[Theorem 4.1]{RR06}, it is shown that for all $p\in\mathbf{p}$ there is a monomorphism $p\to \Omega$, where $\Omega$ is a module that is a coproduct of Pr\"{u}fer modules and copies of the generic module $G$. Yet, on \cite[p. 205]{RR06} it is further shown that $\Omega$ is actually a finite coproduct of copies of $G$. In particular, as the map $p\to \Omega$ is a monomorphism, we deduce that $\Hom_{\Lambda}(p,G)\neq 0$ for all nonzero $p\in \mathbf{p}$.

Combining all this we have that $^{\perp_{0}}G = \mathbf{q} \cup\mathbf{r}$ and $^{\perp_{1}}G = \mathbf{p}\cup\mathbf{r}$. Consequently $\mathbf{r}={}^{\perp_{0,1}}G$ is radical.

Now any wide subcategory corresponding to an exceptional sequence is radical by \cref{lem:ExceptionalIsRadical}. The wide subcategories not corresponding to exceptional sequences consist of regular modules and can be parametrised by 
$\{(p, x_1, . . . , x_s) : p \in  2^{\mathbb{P}^1_k}, x_i \text{ a non-crossing collection of arcs on } n_i\}$. Since $\mathbf{r}={}^{\perp_{0,1}}G$ is radical we can work inside of the regular component. Since the tubes are pairwise Hom- and Ext-orthogonal, all such wide subcategories can be obtained as intersections of subcategories of the form ${}^{\perp_{0,1}}Z$: for $p \in  2^{\mathbb{P}^1_k}$ take the complementary set in $\mathbb{P}^1_k$; for $x_i$, a non-crossing collection of arcs on $n_i$, take the $Z\in\mc{U}_{n}$ constructed in \cref{prop:thwideisperp}, if $x_i$ is empty take the whole tube. 
\end{proof}

\begin{ex}\label{th:kronecker}
Let us explicitly compute $\sspec{\D(\Lambda)^\c}$ for $\Lambda$ the path algebra of the Kronecker quiver. In this case, the regular component $\mathbf{r}$ of $\mod{\Lambda}$ is comprised of a $\mathbb{P}^1_k$-indexed family of tubes, the indecomposable modules of which we denote by $\{R_{\lambda}^{i}:\lambda\in\mathbb{P}^1_k,i\geq 1\}$, while $\mathbf{p}=\{P_{i}:i\geq 0\}$ and $\mathbf{q}=\{Q_{i}:i\geq 0\}$. Note that $P_{0}$ and $P_{1}$ denote the indecomposable projective modules, and $Q_{0}$ and $Q_{1}$ denote the indecomposable injective modules. We let $G$ denote the generic module.

For a $\lambda\in\mathbb{P}^1_k$, we set $\mathbf{r}_{\neq \lambda}=\msf{add}\{R_{\mu}^{i}:\mu\in\mathbb{P}^1_k\backslash \{\lambda\}, i\geq 1\}$ to be $\mathbf{r}$ with the tube at $\lambda$ removed. 

To understand $\sspec{\D(\Lambda)^\c}$ we use the map $\alpha$ of \cref{definingalpha}. Since $M\simeq\oplus_{i\in\Z}\Sigma^iH_{i}(M)$, we see that 
\[
M\in\alpha([X])\iff \Hom_{\Lambda}(H_{i}(M),X)=0=\t{Ext}_{\Lambda}^{1}(H_{i}(M),X) \t{ for all }i\in\Z.
\]
To compute the image of $\alpha$ explicitly, we make use of the Auslander-Reiten formula (\ref{ARform}). We emphasise that below all orthogonals are considering only finitely presented modules. For ease of exposition, in the following list we sometimes do not make a distinction between sets of indecomposable objects and their additive closures.
\begin{enumerate}
\item For $P_{n}\in\mathbf{p}$, we have that $^{\perp_{0}}P_{n}=\{P_{m}:m\geq n+1\}\cup\mathbf{r}\cup\mathbf{q}$ and $^{\perp_{1}}P_{n}=\{P_{m}:m\leq n+1\}$, so $\alpha([P_{n}])=\{M\in\D(\Lambda)^\c:H_{i}(M)\in\msf{add}(P_{n+1}) \t{ for all }i\in\Z\} = \msf{add}^\Sigma(P_{n+1})$.

\item For $Q_{n}\in\mathbf{q}$ with $n\geq 1$, we have $^{\perp_{0}}Q_{n}=\{Q_{m}:m\leq n-1\}$ and $^{\perp_{1}}Q_{n}= \mathbf{p} \cup \mathbf{r}\cup \{Q_{m}:m\geq n-1\}$
so $\alpha([Q_{n}])=\{M\in\D(\Lambda)^\c:H_{i}(M)\in\msf{add}(Q_{n-1})\t{ for all }i\in\Z\} = \msf{add}^\Sigma(Q_{n-1})$.

\item For $n=0$ we have $^{\perp_{0}}Q_{0}=\{P_{0}\}$, and as $Q_{0}$ is injective, $^{\perp_{1}}Q_{0}=\mod{\Lambda}$. Therefore we see that $\alpha([Q_{0}])=\{M\in\D(\Lambda)^\c:H_{i}(M)\in\msf{add}(P_{0}) \t{ for all }i\in\Z\} =\msf{add}^\Sigma(P_{0}).$

\item For $\lambda\in\mathbb{P}^1_k$ and $j\geq 1$ we have $^{\perp_{0}}R_{\lambda}^{j}=\mathbf{r}_{\neq \lambda}\cup \mathbf{q}$ and $^{\perp_{1}}R_{\lambda}^{j}=\mathbf{r}_{\neq \lambda}\cup \mathbf{p}$. Thus we obtain $\alpha([R_{\lambda}^{j}])=\{M\in\D(\Lambda)^\c:H_{i}(M)\in\msf{add}(\mathbf{r}_{\neq\lambda}) \t{ for all }i\in\Z\} = \msf{add}^\Sigma(\mathbf{r}_{\neq\lambda})$.

\item We saw in the proof of \cref{thm:tamheredeverythingradical} that $^{\perp_{0}}G = \mathbf{q} \cup\mathbf{r}$ and $^{\perp_{1}}G = \mathbf{p}\cup\mathbf{r}$, so we obtain that $\alpha([G])=\{M\in\D(\Lambda)^\c:H_{i}(M)\in \mathbf{r}\t{ for all }i\in\Z\}= \msf{add}^\Sigma(\mathbf{r})$.
\end{enumerate}
We therefore see that
\[
\sspec{\D(\Lambda)^\c}=\{\msf{add}^{\Sigma}(P_{n}), \msf{add}^{\Sigma}(Q_{n}):n\in\Z_{\geq 0}\}\cup \{\msf{add}^{\Sigma}(\mathbf{r}_{\neq \lambda}):\lambda\in\mbb{P}^{1}_{k}\}\cup\msf{add}^{\Sigma}(\mathbf{r})
\]
as a set. To determine the topology, we compute the shift support of indecomposables. If $A$ is a compact object, then $\msf{add}^\Sigma( P_{i}) \in \ssupp{A}^{\c}$ if and only if $H_{i}(A)\in\msf{add}(P_{i})$ for all $i\in\Z$. 
It follows that $\ssupp{P_{i}}=\sspec{\D(\Lambda)^\c}\setminus\{\msf{add}^{\Sigma}( P_{i})\}$. By the same reasoning we obtain $\ssupp{Q_{i}}=\sspec{\D(\Lambda)^\c}\setminus\{\msf{add}^{\Sigma}( Q_{i})\}$. If $R_{\lambda}^{i}$ is a regular module, then $\ssupp{R_{\lambda}^{i}}=\{\msf{add}^{\Sigma}( P_{n}), \msf{add}^{\Sigma}(Q_{n}), \msf{add}^{\Sigma}(\mathbf{r}_{\neq \lambda})\}$. 

Thus the basic open sets are finite intersections of: 
\begin{enumerate}
    \item $\{\msf{add}^{\Sigma}( P_{i})\}$ for $i \geq 0$;
    \item $\{\msf{add}^{\Sigma}(Q_{i})\}$ for $i\geq 0$;
    \item $\left(\bigcup\limits_{\mu\in\mbb{P}^{1}_{k}\setminus\lambda}\{\msf{add}^{\Sigma}(\mathbf{r}_{\neq \mu})\}\right)\cup \{\msf{add}^{\Sigma}(\mathbf{r})\}$ for any point $\lambda\in \mbb{P}^{1}_{k}$.
\end{enumerate} 

From this description we see that \[\sspec{\D(\Lambda)^\c} \simeq \Z\sqcup\mbb{P}^{1}_{k}\] where $\Z$ is considered with the discrete topology, since an open set in $\Z\sqcup \mbb{P}^{1}_{k}$ is a union of an arbitrary subset of $\Z$ with either a cofinite subset of $\mbb{P}^{1}_{k}$ or an empty subset of $\mbb{P}^{1}_{k}$.
\end{ex}

\section{One dimensional hypersurfaces of countable Cohen-Macaulay type}\label{hypersurfaces}
In this section we compute the shift-homological spectrum and the shift-spectrum for the singularity categories of all one dimensional hypersurfaces of countable Cohen-Macaulay type. It transpires that, for the hypersurfaces of countably infinite Cohen-Macaulay type, the zero thick subcategory is no longer radical, which provides a contrast to all other examples seen so far. Before computing the shift-homological spectrum and the shift-spectrum, we provide some background material concerning maximal Cohen-Macaulay modules.

\begin{chunk}
Throughout, we let $k$ denote an algebraically closed field of characteristic not equal to $2$. We shall consider complete one-dimensional hypersurface singularities of the form $R=k[[x,y]]/(f)$, where $f\in(x,y)^{2}$.

The category of \emph{maximal Cohen-Macaulay} modules is denoted by by $\msf{MCM}(R)$, and as $R$ is one-dimensional, we have
\[
\msf{MCM}(R)=\{M\in\mod{R}:\Hom_{R}(k,M)=0\}.
\]
As illustrated in \cite{buchweitz}, $\msf{MCM}(R)$ is a Frobenius category, whose projective-injective objects are given by $\msf{proj}(R)$. There is a triangulated equivalence
\begin{equation}\label{eqn:buchweitz}
\msf{StMCM}(R)\simeq \D_{\t{sg}}(R),
\end{equation}
where $\D_{\t{sg}}(R)$ is the \emph{singularity category} of $R$, defined as the Verdier quotient $\msf{D}^{\t{b}}(\msf{mod}(R))/\D(R)^\c$.

As shown in \cite{Krausest}, the category $\msf{D}_{\t{sg}}(R)$ is equivalent to the category of compact objects in $\msf{K}_{\t{ac}}(\msf{Inj}(R))$, and we therefore may compute $\shspec{\D_{\t{sg}}(R)}$ and $\sspec{\D_{\t{sg}}(R)}$. We do this for a particular family of hypersurfaces, determined by their representation type.
\end{chunk}

\begin{chunk}
The ring $R$ is of \emph{countable Cohen-Macaulay type} if $\msf{MCM}(R)$ has countably many indecomposable objects up to isomorphism. A complete description of the hypersurfaces of this type was given in \cite{bgs}: $R$ is of countable Cohen-Macaulay type if and only if $f$ is a simple singularity, that is: 
\[
f=
\begin{cases}
(A_{n}):&x^{2}+y^{n+1}, \t{ for }1\leq n\leq \infty,\\
(D_{n}): & x^{2}y+y^{n-1}, \t{ for }4\leq n\leq \infty, \\
(E_{6}):& x^{3}+y^{4},\\
(E_{7}):&x^{3}+xy^{3},\\
(E_{8}):&x^{3}+y^{5},
\end{cases}
\]
where by convention $y^{\infty}=0$. In fact, the ring is of \emph{finite} Cohen-Macaulay type if and only if it is $A_{n},D_{n},E_{6},E_{7}$ or $E_{8}$ for $n<\infty$.
\end{chunk}

\begin{chunk}
In the case that $R$ is of finite Cohen-Macaulay type, the category $\D_{\t{sg}}(R)$ is locally finite, see \cite[p.5]{KrauseLocallyFinite}, hence the spaces $\shspec{\D_{\t{sg}}(R)}$ and $\sspec{\D_{\t{sg}}(R)}$ can be deduced from the results of \cref{sec:locallyfinite}, together with a description of the finitely many indecomposable maximal Cohen-Macaulay modules, which can be found in, for instance, \cite{Yoshino}. 

Furthermore, if $R$ is of finite Cohen-Macaulay type, the space $\sspec{\D_{\t{sg}}(R)}$ will be a point, given by the thick subcategory $\{0\}$. This is because thick subcategories of $\D_{\t{sg}}(R)$ biject with the specialisation closed subcategories of the singular locus of $R$ by the work in \cite{stevenson} and \cite{takahashi}; as the hypersurfaces of finite Cohen-Macaulay type are isolated singularities, the singular loci are $\{0\}$, and hence there are two thick subcategories, which are $\{0\}$ and $\D_{\t{sg}}(R)$. It can only be the case that $\{0\}$ is prime.
\end{chunk}

\begin{chunk}
Therefore, to completely exhaust the case of countable Cohen-Macaulay type curves, we only need to compute the spaces for the $A_{\infty}$ and $D_{\infty}$ cases. We approach the problem through Gorenstein homological algebra.

Since $R$ is Gorenstein, the category $\msf{MCM}(R)$ coincides with the finitely presented objects in the category of \emph{Gorenstein flat} modules, which we denote as $\msf{GFlat}(R)$. The Gorenstein flat modules have a description as
\[
\msf{GFlat}(R)=\{M\in\Mod{R}:\t{Tor}_{i}^{R}(E,M)=0 \t{ for all }E\in\msf{Inj}(R) \t{ and }i\geq 1\},
\]
by \cite[Theorem 10.3.8]{EnochsJenda}, and there is an equality $\msf{GFlat}(R)=\rlim\msf{MCM}(R)$, as illustrated by combining \cite[Theorem 10.3.8]{EnochsJenda} and \cite[Corollary 11.5.4]{EnochsJenda}.

The category $\msf{GFlat}(R)$ has a Frobenius subcategory, consisting of the modules which are simultaneously Gorenstein flat and \emph{cotorsion}, where an $R$-module $C$ is cotorsion if and only if $\t{Ext}_{R}^{1}(F,C)=0$ for all flat $R$-modules $F$; we let $\msf{Cot}(R)$ denote the full subcategory of cotorsion $R$-modules. We shall let $\msf{GFC}(R)=\msf{GFlat}(R)\cap\msf{Cot}(R)$ denote this Frobenius category. The projective-injective objects in $\msf{GFC}(R)$ are $\msf{Flat}(R)\cap\msf{Cot}(R)$, and we let $\msf{StGFC}(R)$ denote the corresponding stable category. See \cite[\S4]{CET} for proofs of these facts.

In the case of hypersurfaces, the category $\msf{StGFC}(R)$ is compactly generated as shown in \cite{birdgfc}, and the compact objects are $\msf{StMCM}(R)$. As such, there is a triangulated equivalence $\msf{StGFC}(R)\simeq\msf{K}_{\t{ac}}(\msf{Inj}(R))$. The advantage of approaching $\shspec{\D_{\t{sg}}(R)}$ this way is that the Ziegler spectrum of $\msf{GFlat}(R)$ is already known for both $A_{\infty}$ and $D_{\infty}$, see \cite{pun1} and \cite{PunLos} respectively.
\end{chunk} 

\begin{chunk}\label{Zgdecomp}
As shown in \cite[Theorem 3.13]{birdgfc}, there is a homeomorphism $\msf{Zg}(\msf{GFlat}(R))\simeq \msf{Zg}(\msf{StGFC}(R))\sqcup \msf{Zg}(\msf{Flat}(R))$. Since the isolation condition holds for $\msf{GFlat}(R)$ for $R$ the $A_{\infty}$ and $D_{\infty}$ hypersurfaces by \cite[Theorem 6.9]{pun1} and \cite[Theorem 7.7]{PunLos} respectively, it follows that the isolation condition also holds in $\msf{Zg}(\msf{StGFC}(R))$, as it is a closed subspace of $\msf{Zg}(\msf{GFlat}(R))$.
Consequently, we are in the setting of \cref{thm:whenisKahomeo}, so it suffices to understand the indecomposable $\Sigma$-invariant endofinite objects of $\msf{StGFC}(R)$. We first do this for $A_{\infty}$ and then, via similar arguments, consider $D_{\infty}$.
\end{chunk}

\subsection{The \texorpdfstring{$A_{\infty}$-}{A-infinity }singularity}\label{ainfinity}
We now restrict to the case that $R=k[[x,y]]/x^{2}$, the $A_{\infty}$-singularity. We let $\mf{p}=(x)$ denote the unique non-maximal prime of $R$. The Ziegler spectrum of $\msf{GFlat}(R)$ was described in \cite{pun1} and has points
\[
\{(x,y^{i})\subset R:0\leq 1\leq\infty\}\cup \{Q(R),\overline{R},L\}
\]
where the first set consists of the indecomposable maximal Cohen-Macaulay modules with the convention that $y^{0}=1$, $Q(R)$ is the ring of quotients of $R$, $\overline{R}$ is the normalisation of $R$ (the integral closure of $R$ in $Q(R)$) viewed as an $R$-module via the nontrivial integral ring extension $R\to \bar{R}$, and $L=k((y))$, the Laurent series over $y$, is viewed as an $R$-module along the quotient $R\to k[[y]]=R/(x)$.

The flat modules among this list are $(x,y^{0})=R$ and $Q(R)$, as the latter is a localisation. To see that $\bar{R}$ is not flat, note that $\bar{R}=R+xQ$ by \cite[Remark 2.1]{pun1}, and thus $R$ is properly contained in $\bar{R}$. The  ring map $R\to \bar{R}$ then corresponds to a map of schemes $\t{Spec}(\bar{R})\to\t{Spec}(R)$ which is flat if and only if $\bar{R}$ is flat as an $R$-module. Yet as $\t{Spec}(\bar{R})$ is normal by definition, if this map were flat, then $\t{Spec}(R)$ would also be normal by \cite[Proposition 2.1.13]{EGA}, and this is false as $R$ is not normal. It is also the case that $L$ is not flat, for if it were $L_{\p}=k((y))_{\p}$ would be flat over $R_{\p}=k((y))[x]/(x^{2})$, and this is clearly not the case as it is the residue field of this local ring.
Therefore we see that the non-zero indecomposable pure injective objects in $\msf{StGFC}(R)$ are $\{(x,y^{i}):1\leq i\leq \infty\}\cup\{\overline{R},L\}$. 

\begin{chunk}\label{ainfinityshiftspec}
Since the isolation condition holds, the closed points in $\msf{Zg}(\msf{StGFC}(R))$ correspond to endofinite objects, see \cref{prel:isolationcondition}. As such, any $\Sigma$-invariant indecomposable endofinite object of $\msf{StGFC}(R)$ gives a closed point of $\msf{Zg}(\msf{GFlat}(R))$. This is because the embedding $\msf{Zg}(\msf{StGFC}(R))\hookrightarrow\msf{Zg}(\msf{GFlat}(R))$ arising in the disjoint union described in \cref{Zgdecomp} is closed. Yet the closed points of $\msf{Zg}(\msf{GFlat}(R))$ are just, by \cite[Lemma 6.7]{pun1}, the Laurent series $L$ and the ring of quotients $Q$; as we saw, $Q$ is flat, so we deduce that $L$ is the only closed point, and thus endofinite object, in $\msf{Zg}(\msf{StGFC}(R))$. 

In addition, the point $L$ is also $\Sigma$-invariant. This is because there is a short exact sequence $0\to L\to Q\to L\to 0$ in $\msf{GFC}(R)$, see \cite[Lemma 8.1]{pun1}, which illustrates $\Sigma L\simeq L$. Furthermore, the shift on $\msf{StMCM}(R)\simeq\msf{D}_{\t{sg}}(R)$ is also the identity, since the matrix factorisations exhibiting each indecomposable maximal Cohen-Macaulay module are symmetric. We note that since no finitely generated maximal Cohen-Macaulay module is closed in $\msf{Zg}(\msf{GFlat}(R))$, none of their images in $\msf{StGFC}(R)$ is endofinite. Therefore we have that $\shspec{\D_{\t{sg}}(A_\infty)}$ is the one-point space
\[
\shspec{\D_{\t{sg}}(A_{\infty})}=\{\{f\in\mod{\D_{\t{sg}}(R)}:\Hom(f,\y L)=0\}\}.
\]
\end{chunk}

\begin{chunk}\label{chunk:identifyalphaL}
Let us now give a concrete description of $\sspec{\D_{\t{sg}}(R)}$, which we now know is a single point, given by 
\[\
\alpha(L)=\{M\in\msf{StMCM}(R):\underline{\Hom}_{R}(M,L)=0\}.
\]
Using the fact that $\Sigma L=L$, there is an isomorphism $\underline{\Hom}_{R}(M,L)\simeq \t{Ext}_{R}^{1}(M,L)$ for any $M\in\msf{StMCM}(R)$, see \cite[Lemma 3.3.3]{krbook}. Since $L$ is $\p$-local, that is $L_{\p}\simeq L$, there is, for each $n\geq 0$, an isomorphism
\[
\t{Ext}_{R}^{n}(M,L)\simeq \t{Ext}_{R_{\p}}^{n}(M_{\p},L).
\]
Viewing $L=k((y))$ as the residue field of $R_{\p}\simeq k((y))[[x]]/x^{2}$, we therefore have that
\[
\underline{\Hom}_{R}(M,L)=0 \iff \t{Ext}_{R}^{1}(M,L)=0 \iff \t{Ext}_{R_{\p}}^{1}(M_{\p},k((y)))=0 \iff M_{\p}\in\msf{proj}(R_{\p}).
\]
Now, since $\msf{Spec}(R)=\{\p,\m=(x,y)\}$, we have that the punctured spectrum is just $\{\p\}$; in other words, from the above equivalences, we have that
\[
M\in\alpha(L)\iff M \t{ is locally free on the punctured spectrum}.
\]
Here we appeal to Auslander-Reiten theory, for it is well known that in $\msf{MCM}(R)$ a module is locally free on the punctured spectrum if and only if there is an Auslander-Reiten sequence ending at it, see \cite[Theorem 3.4]{Yoshino}.

The Auslander-Reiten quiver for $A_{\infty}$, see \cite{schreyer} or \cite{pun1}, illustrates that the only module which does not appear at the end of an Auslander-Reiten sequence is $(x,y^{\infty})$. In other words,
\[
\alpha(L)=\msf{thick}\{(x,y^{i}):0<i<\infty\}.
\]
In particular, we have that the only radical thick subcategories of $\D_{\t{sg}}(R)$ are $\msf{thick}\{(x,y^{i}):0<i<\infty\}$ and $\D_{\t{sg}}(R)$. This exhibits an example when $\{0\}$ is not a radical thick subcategory.
\end{chunk}

\subsection{The \texorpdfstring{$D_{\infty}$-}{D-infinity }singularity}

We now let $R=k[[x,y]]/x^{2}y$ be the $D_{\infty}$-singularity. Let us define, for $1\leq k<\infty$, $M_{k}=(y^{k+1},xy)$ and $Y_{k}=(y^{k},x)$, which are ideals of $R$. We also consider the following submodules of $R^{(2)}$, given by $X_{k}=\langle (y^{k},x),(xy,0)\rangle_{R}$ and $N_{k}=\langle (y^{k},x),(x,0)\rangle_{R}$. Then a complete list of indecomposable objects of $\msf{MCM}(R)$ is 
\[
\{R,(x),(x^{2}),(xy),(y)\}\cup\{M_{k},Y_{k},X_{k},N_{k}:1\leq k<\infty\}.
\]

\begin{chunk}
The indecomposable endofinite objects in $\msf{StGFC}(R)$ - or equivalently the closed points in $\msf{Zg}(\msf{StGFC}(R))$ as the isolation condition holds - are precisely the non-flat closed points in $\msf{Zg}(\msf{GFlat}(R))$ via the decomposition in \cref{Zgdecomp}. By \cite[Lemmas 6.10 and 7.6]{PunLos}, the closed points of $\msf{Zg}(\msf{GFlat}(R))$ are $Q(R)$ and the Laurent series rings $k((x))$ and $L=k((y))$. The module $Q(R)$ is always flat, while $k((x))$ is the injective hull of $R/(y)$. As $R$ is Gorenstein, $k((x))$ having finite injective dimension also means it has finite flat dimension, and thus $k((x))$ is Gorenstein flat of finite flat dimension, hence is flat by \cite[Corollary 10.3.4]{EnochsJenda}. 
We note that $L=k((y))$ is not flat, for if it were it would, by extension of scalars, be flat over the $A_{\infty}$-singularity, viewed as the quotient $R/x^{2}\simeq k[[x,y]]/x^{2}$. Yet we know this is not the case by the discussion in \cref{ainfinity}.

Thus, the only closed point in $\msf{Zg}(\msf{StGFC}(R))$ is $L$. We note that this object is $\Sigma$-invariant. Indeed, since $\Sigma\colon\msf{StGFC}(R)\to\msf{StGFC}(R)$ is an autoequivalence, it induces a closed map on $\msf{Zg}(\msf{StGFC}(R))$, in particular $\Sigma k((y))$ is a closed point. Yet as $L$ is the only closed point, we deduce that $L\simeq \Sigma L$. 
\end{chunk}

\begin{chunk}
Therefore, using a similar argument as in the case of the $A_{\infty}$-singularity, the sets $\shspec{\D_{\t{sg}}(R)}$ and $\sspec{\D_{\t{sg}}(R)}$ are both homeomorphic to the point: we have that
\[
\shspec{\D_{\t{sg}}(R)}=\{\{f\in\mod{\T^{\c}}:\Hom(f,\y L)=0\}\}
\]
while
\[
\sspec{\D_{\t{sg}}(R)}=\{\{M\in\msf{StMCM}(R):\underline{\Hom}_{R}(M,L)=0\}\}.
\]
We again give a more concrete description of the latter by appealing to Auslander-Reiten theory.

By a similar argument to that used in \cref{chunk:identifyalphaL}, any module that is locally free on the punctured spectrum, equivalently which ends an Auslander-Reiten sequence in $\msf{MCM}(R)$, will lie in $\alpha(L)$. There is an Auslander-Reiten sequence ending at all indecomposable maximal Cohen-Macaulay modules apart from the modules $(x)$ and $(xy)$. Hence
\begin{equation}\label{ex:Dthick}
\msf{thick}(\{(x^{2}),(y),M_{k},N_{k},X_{k},Y_{k}:1\leq k<\infty\})\subseteq \alpha(L).
\end{equation}
We now consider $(x)$ and $(xy)$. In $\msf{MCM}(R)$ there is a 2-periodic ray 
\[
(xy)\xrightarrow{\t{inc}}(x) \xrightarrow{\cdot y}(xy)\xrightarrow{\t{inc}}(x) \xrightarrow{\cdot y}\cdots
\]
whose direct limit in $\msf{GFlat}(R)$ is $L$. Since none of the maps in the system are zero in $\msf{StMCM}(R)$, the colimit maps $(x)\to L$ and $(xy)\to L$ are also non-zero in $\msf{StGFC}(R)$; in particular $(x),(xy)\not\in \alpha(L)$ and so the inclusion in \cref{ex:Dthick} is an equality.
We therefore again have a situation where $\{0\}$ is not radical, as the only radical thick subcategories are $\alpha(L)$ and $\D_{\t{sg}}(R)$.
\end{chunk}

\section{Comparisons to other spaces}\label{comapsection}
To end, we give a brief comparison between $\sspec{\T^{\c}}$ and other spaces parametrising certain thick subcategories that appear in the literature. We do not recall the construction of these spaces here, and instead refer the reader to the relevant articles dedicated to their study. 

The spaces we consider are the \emph{partially functorial spectrum} $\msf{Spcnt}(\T^{\c})$ and the \emph{fully functorial spectrum} $\msf{fSpcnt}(\T^\c)$ of \cite{GratzStevenson}, the space $\msf{Spc}_{\msf{cent}}(\T^{\c})$ derived from the central support of \cite{Krausecentral}, and the \emph{Matsui spectrum} $\msf{Spc}_{\msf{M}}(\T^{\c})$ of \cite{Matsui}. In general, we do not have continuous comparison maps, but by example we can see that the spaces have different behaviours. We demonstrate this in \cref{table:comparisons}, and provide further elaboration below. In the following table, we write $(-)^\vee$ for the Hochster dual operation on topological spaces, and $S$ for the Sierpi\'nski space (that is, the 2 point space \{0,1\}, with closed sets $\varnothing$, $\{0\}$, and $\{0,1\}$). The table also contains the Balmer spectrum when it exists for comparison.

\begin{table}[h]
\begin{tabular}{c| c c c c}
& $\D^{\mrm{b}}(kA_{2})$ & $\D^{\mrm{b}}(\msf{coh}(\mathbb{P}^1))$ & $\msf{D}_{\t{sg}}(A_{\infty})$ & $\D(\Z)^{\c}$ \\ \hline \rule{0pt}{4ex}
$\sspec{\T^{\c}}$ & $\ast \sqcup \ast \sqcup \ast$ & $\Z \sqcup \mbb{P}^1$ & $\ast$ & $\msf{Spec}(\Z)$ \\ \rule{0pt}{4ex}
$\msf{Spcnt}(\T^{\c})$ & $\varnothing$ & $\varnothing$ & $S$ & $\msf{Spec}(\Z)^{\vee}$ \\ \rule{0pt}{4ex}
$\msf{Spc}_{\msf{cent}}(\T^{\c})$ & $\ast$ & $\ast$ & $S$ & $\msf{Spec}(\Z)^{\vee}$ \\ \rule{0pt}{4ex}
$\msf{Spc}_{\msf{M}}(\T^{\c})$ & $\ast \sqcup \ast \sqcup \ast$ & $\Z \sqcup \mbb{P}^1$ & $S$ & $\msf{Spec}(\Z)$ \\ \rule{0pt}{4ex}
$\msf{Spc}(\T^\c)$ & - & $\mbb{P}^1$ & - & $\msf{Spec}(\Z)$ \\
\end{tabular}
\caption{Comparisons between spaces parametrising thick subcategories of $\T^\c$}
\label{table:comparisons}
\end{table}

\begin{chunk}[Computations for $\msf{D}^{\t{b}}(kA_{2})$]
As $kA_{2}$ is hereditary of finite representation type, the space $\sspec{\T^{\c}}$ is the discrete space with three points by \cref{locfiniteindmodsim} and \cref{ex:psshered}. The lattice of thick subcategories of $\msf{D}^{\t{b}}(kA_{2})$ is
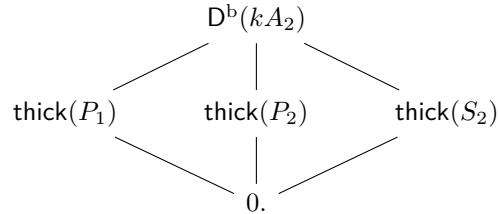
\[
\begin{tikzcd}
& \D^{\mrm{b}}(kA_{2}) & \\
\msf{thick}(P_{1}) \arrow[ur,no head]  & \msf{thick}(P_{2}) \arrow[u,no head]& \msf{thick}(S_{2}) \arrow[ul,no head] \\
& 0. \arrow[ul,no head] \arrow[u,no head] \arrow[ur,no head]&
\end{tikzcd}
\]
We have $\msf{Spcnt}(\D^{\mrm{b}}(kA_{2}))=\varnothing$ by ~\cite[Lemma 2.3.7]{GratzStevenson}. The lattice of central thick subcategories is $Z(\msf{D}^{\mrm{b}}(kA_{2}))=\{0,\D^{\mrm{b}}(kA_{2})\}$, thus $\msf{Spc}_{\msf{cent}}(\msf{D}^{\mrm{b}}(kA_2))=\ast$ is the one-point space, see~\cite[Example 4.13]{Krausecentral}. 
From the definition of $\msf{Spc}_{\msf{M}}$, one sees that the only Matsui primes are $\msf{thick}(P_1), \msf{thick}(P_{2})$, and $\msf{thick}(S_{2})$, and that the space is discrete by definition of the basis of the topology.
\end{chunk}

\begin{chunk}[Computations for $\msf{D}^{\mrm{b}}(\msf{coh}(\mbb{P}^1))$]
    We considered this example in detail in \cref{th:kronecker}. Justifications for the other spaces can be found in \cite[Example 7.1.4]{GratzStevenson}, \cite[Example 4.15]{Krausecentral}, \cite[Example 4.10]{Matsui}, and \cite[Corollary 5.6]{Balmer} respectively.
\end{chunk}

\begin{chunk}[Computations for $\msf{D}_{\t{sg}}(A_{\infty})$]
We considered $\sspec{\D_{\t{sg}}(A_{\infty})}$ in \cref{ainfinity}. The lattice of thick subcategories is equivalent to $\{0{<}1{<}2\}$. We see that $\msf{Spc}_{\msf{M}}(\D_{\t{sg}}(A_{\infty}))$ then consists of the points corresponding to $0$ and $1$, and calculating the basic closed sets one sees that $\msf{Spc}_{\msf{M}}(\D_{\t{sg}}(A_{\infty}))$ is the Sierpi\'nski space $S$. Since the thick subcategories are all ordered by inclusion, it is clear from the definition that all thick subcategories are central. From this, one easily calculates that the corresponding space $\msf{Spc}_\msf{cent}(\msf{D}_{\t{sg}}(A_\infty))$ under Stone duality is $S$. The same calculation also shows that $\msf{Spcnt}(\msf{D}_{\t{sg}}(A_\infty)) = S$.
\end{chunk}

\begin{chunk}[Computations for $\D(\Z)^{\c}$]
For $\D(\Z)^{\c}$, we saw in \cref{dedekind} that there is a homeomorphism $\sspec{\D(\Z)^{\c}}\simeq \msf{Spec}(\Z)$. The remaining identifications follow from \cite[Example 7.1.1]{GratzStevenson}, \cite[Corollary 4.4]{Krausecentral}, \cite[Corollary 2.17]{Matsui}, and \cite[Corollary 5.6]{Balmer}.
\end{chunk}

\begin{rem}
    In \cite[Definition 5.1.1 and Proposition 5.3.3]{GratzStevenson}, the fully functorial spectrum $\msf{fSpcnt}(\T^\c)$ is defined via Stone duality, but it is then shown that it is homeomorphic to $\msf{Thick}(\T^\c)$ topologised with closed sets being finite unions of subsets of the form $\upcl(\msf{L}) = \{\msf{L}' \in \msf{Thick}(\T^\c) : \msf{L} \subseteq \msf{L}'\}$ where $\msf{L} \in \msf{Thick}(\T^\c)$. From this description, it is easy to see that the inclusion of the shift-prime thick subcategories into all thick subcategories yields an injective continuous map \[\sspec{\T^\c}^\vee \to \msf{fSpcnt}(\T^\c)\] if $\sspec{\T^\c}$ is noetherian. Indeed, the noetherian assumption ensures that the open sets of the Hochster dual of the shift-spectrum are just the specialisation closed subsets in the usual topology, from which continuity easily follows. Alternatively, one may give a proof by Stone duality using the join-preserving poset map $\msf{supp}_\Sigma\colon \Thick(\T^\c) \to \msf{SC}(\sspec{\T^\c})$. We note that the noetherian assumption and the appearance of the Hochster dual are consequences of our differing conventions on whether support theories assign open subsets (as in Gratz--Stevenson~\cite{GratzStevenson}) or closed subsets (as in this paper and Balmer--Ocal~\cite{balmerocal}).
\end{rem}

\bibliographystyle{abbrv}
\bibliography{references.bib}
\end{document}